\documentclass[11pt]{article}
\usepackage{amssymb}
\usepackage{amsmath}

\newtheorem{theorem}{Theorem}[section]

\newtheorem{corollary}[theorem]{Corollary}

\newtheorem{definition}[theorem]{Definition}
\newtheorem{example}[theorem]{Example}

\newtheorem{lemma}[theorem]{Lemma}

\newtheorem{proposition}[theorem]{Proposition}
\newtheorem{remark}[theorem]{Remark}

\newenvironment{proof}[1][Proof]{\emph{#1.} }{\ \rule{0.5em}{0.5em}}

\newcommand{\K}{\mathbb K}

\newcommand{\ben}{\begin{enumerate}}
\newcommand{\een}{\end{enumerate}}

\begin{document}

\title{ Gerstenhaber-Schack Cohomology for Hom-bialgebras  and Deformations }
\author{Khadra  Dekkar and Abdenacer  Makhlouf}
\date{}

\maketitle

\begin{abstract}
  Hom-bialgebras and Hom-Hopf algebras are generalizations of bialgebra and Hopf algebra structures, where  associativity and coassociativity conditions are twisted by a homomorphism. The purpose of this paper is to define  a Gerstenhaber-Schack  cohomology complex for
Hom-bialgebras and then study one parameter formal  deformations. 
\end{abstract}

\section*{Introduction}

The first instance of Hom-type algebras appeared in various papers dealing with $q$-deformations of algebras of vector fields, mainly Witt and Virasoro algebras, see for example \cite{AizawaSaito}. Then  Hartwig, Larsson and  Silvestrov introduced and studied the concept of Hom-Lie algebra, which is a generalization of Lie algebra where the Jacobi condition is twisted by a homomorphism, see \cite{HJ
2006,D L 2005}. Hom-associative algebras  were introduced
and studied by the second author  and Silvestrov in \cite{AM}, where it is shown that usual functors between associative algebras and Lie algebras extend to Hom-type algebras. 
Moreover,  Hom-analogues of coalgebras, bialgebras and Hopf algebras have been introduced in
\cite{AMS 2007Hbial,AM
2008HBial}. The original definition of a Hom-bialgebra involved two linear maps, one twisting the
associativity condition and the other one the coassociativity condition. Later, two directions of
study on Hom-bialgebras were developed, one in which the two maps coincide,  they are still
called Hom-bialgebras,  and another one, started in \cite{SC 2009}, where the two maps are assumed to be
inverse to each other,  they are called monoidal Hom-bialgebras.
In the last years, many concepts and properties from classical algebraic theories have been
extended to the framework of Hom-structures, see for instance  \cite{F A 2010,Chen2013,Elhamdadi-Makhlouf,Y F 2009,A G 2010,Mak:Almeria,A M 2014,A M 2015,Sheng,DY2009-1,yauhomyb1,DY 2010,DY 2012,DYa 2009,D Y 2009bis,yauhomyb2,Zhao}.
The cohomology and deformations of Hom-associative algebra were initiated in 
\cite{AM 2007} and then completed in \cite{F. A 2010}.


In this paper,  we  define a cohomology complex for Hom-bialgebras, generalizing Gerstenhaber-Schack  cohomology in \cite{G.S90,G.S92}, 
and then study  one-parameter formal deformations. It is organized as
follows. In Section 1, we recall the basics on Hom-bialgebras and Hom-Hopf algebras,
and provide some concrete examples. 
Section 2 is dedicated to the study of Hom-type version of module over
algebras (resp. comodule over coalgebras). Moreover, we discuss their tensor
products.  In Section 3, we define a cohomology complex for Hom-bialgebras,
generalizing Gerstenhaber-Schack complex. Finally in Section 4, we study  Hom-bialgebras 
1-parameter formal deformation theory.

\section{Definitions and Preliminaries}

In this section, we first recall basics on  unital Hom-associative algebras, counital Hom-coalgebras and Hom-bialgebras.
We fix some conventions and notations. In this paper $\Bbbk $ denotes an
algebraically closed field of characteristic zero, even if the general
theory does not require it. Vector spaces, tensor products, and linearity
are all meant over $\Bbbk $, unless otherwise specified. We denote by $\tau
_{i,j}:V_{1}\otimes ...\otimes V_{i}\otimes ...\otimes V_{j}\otimes
...\otimes V_{n}\longrightarrow V_{1}\otimes ...\otimes V_{j}\otimes
...\otimes V_{i}\otimes ...\otimes V_{n}$ the flip isomorphism where $\tau
_{i,j}\left( x_{1}\otimes x_{2}\otimes ...\otimes x_{i}\otimes ...\otimes
x_{j}\otimes ...\otimes x_{n}\right) =\left( x_{1}\otimes x_{2}\otimes
...\otimes x_{j}\otimes ...\otimes x_{i}\otimes ...\otimes x_{n}\right) .$

We use  in the sequel Sweedler's notation for the comultiplication, $
\Delta \left( x\right) =\sum_{\left( x\right) }x_{\left( 1\right) }\otimes
x_{\left( 2\right) }
$,  and sometimes the multiplication is denoted by a dot for simplicity and  when there is no confusion.

\subsection{Unital Hom-associative algebras}

\begin{definition} A Hom-associative algebra is a triple $(A,\mu ,\alpha )$ consisting
of a $\Bbbk $-vector space $A$, a linear map $\mu :A\otimes A \rightarrow A
$ \ (multiplication) and a homomorphism $\alpha :A\rightarrow A$
 satisfying the
Hom-associativity condition 
\begin{equation}
\mu \circ \left( \alpha \otimes \mu \right) =\mu \circ \left( \mu \otimes
\alpha \right) .  \label{(1)}
\end{equation}
We assume moreover in this paper that   $\alpha \circ \mu =\mu \circ \alpha ^{\otimes 2}$. 

A Hom-associative algebra $A$ is called unital if there exists a
linear map $\eta :\Bbbk \rightarrow A$ $\ $ such that $\alpha \circ \eta
=\eta $ and%
\begin{equation}
\mathit{\ }\mu \circ \left( \eta \otimes id_{A}\right) =\mu \circ \left(
id_{A}\otimes \eta \right) =\alpha . \label{(2)}
\end{equation}
\end{definition}
The unit element is $1_{A}=\eta \left( 1_{\Bbbk }\right) .$

Hom-associativity and unitality conditions (\ref{(1)}) and (\ref{(2)})
may be expressed by the following commutative diagrams.

\begin{equation*}
\begin{array}{ccc}
A\otimes A\otimes A & \overset{\ \mu \otimes \ \alpha }{\longrightarrow } & 
\ A\otimes A \\ 
{\small \ \alpha \otimes \mu \ }\downarrow \qquad  &  & \downarrow {\small %
\mu } \\ 
\quad A\otimes A\mathit{\ } & \overset{\mu }{\longrightarrow } & A~~%
\end{array}%
\qquad 
\begin{array}{ccccc}
\ \Bbbk \otimes A & \overset{\eta \otimes id_{A}}{\longrightarrow } & 
A\otimes A & \ \overset{\ id_{A}\otimes \eta }{\longleftarrow } & \ A\otimes
\Bbbk  \\ 
\cong  &  & \downarrow \mathit{\ }\mu  &  & \cong  \\ 
A & \overset{\alpha }{\longrightarrow } & A\mathit{\ } & \overset{\ \alpha }{%
\longleftarrow } & A.%
\end{array}%
\end{equation*}

\begin{remark}
\

\begin{enumerate}
\item We recover the classical \ associative algebra when the twisting map $%
\alpha $ is the identity map.

\item We have $\alpha \circ \eta \left( 1_{\Bbbk }\right) =\eta \left(
1_{\Bbbk }\right) $ then$\ \alpha \left( 1_{A}\right) =1_{A}$ and $\mu
\left( 1_{A}\otimes 1_{A}\right) =1_{A}$.

\item We call Hom-associator the linear map $\mathfrak{as}_{A}$ defined on $A^{\otimes 3}$ by $%
\mu \circ \left( \alpha \otimes \mu -\mu \otimes \alpha \right) $.
\end{enumerate}
\end{remark}

\begin{example}
\label{A}
\

\begin{enumerate}
\item The tensor product of two unital Hom-associative algebras $%
(A_{1},\mu _{1},\eta _{1},\alpha _{1})$ and $\left( A_{2},\mu _{2},\eta
_{2},\alpha _{2}\right) $ is defined by $(A_{1}\otimes A_{2},\tilde{\mu},%
\tilde{\eta},\tilde{\alpha})$ such that 
$ 
\tilde{\mu}=\left( \mu _{1}\otimes \mu _{2}\right) \circ \tau _{2,3},\ 
\tilde{\eta}=\eta _{1}\otimes \eta _{2}, \ \tilde{\alpha}=\alpha
_{1}\otimes \alpha _{2},
$  where $\tau
_{23}=id_{A_{1}}\otimes \tau _{A_{2}\otimes A_{1}}\otimes id_{A_{2}}$ and $%
\tau _{A_{2}\otimes A_{1}}:A_{2}\otimes A_{1}\longrightarrow A_{1}\otimes
A_{2};$ $\tau _{A_{2}\otimes A_{1}}\left( x_{2}\otimes x_{1}\right)
=x_{1}\otimes x_{2},$ is the linear `flip' map.

\item Given a Hom-associative algebra $A=(A,\mu ,\alpha ),$ we define the 
opposite Hom-associative algebra $A^{op}=(A,\mu ^{op},\alpha )$ as
the Hom-associative algebra with the same underlying vector space  $A,$
but with a multiplication defined by 
$ 
\mu ^{op}=\mu \circ \tau _{A\otimes A}, \ \mu ^{op}\left( x\otimes
y\right) =\mu \left( y\otimes x\right) .
$ 

A Hom-associative algebra $(A,\mu ,\alpha )$ is commutative if and
only if $\ \mu ^{op}=\mu .$
\end{enumerate}
\end{example}

Let $(A,\mu ,\alpha )$ and $(A^{\prime },\mu ^{\prime },\alpha ^{\prime })$
be two Hom-associative algebras. A linear map $f:A\rightarrow A^{\prime }$
is said to be a \emph{ Hom-associative algebras morphism} if

\begin{equation}
\mu ^{\prime }\circ f^{\otimes 2}=f\circ \mu \text{ and \ }f\circ
\alpha =\alpha ^{\prime }\circ f.  \label{(3)}
\end{equation}

It is said to be a \emph{weak morphism}  
if holds only the first condition. If further  the Hom-associative algebras are unital with respect to  $\eta$ and $\eta ^{\prime }$,  then $f\circ
\eta =\eta ^{\prime }$.

If $A=A^{\prime },$ then the Hom-associative algebras (resp. unital
Hom-associative algebras) are \emph{isomorphic} if there exists a
bijective linear map $~f:A\rightarrow A$ such that%
\begin{equation}
\mu=f^{-1}\circ \mu ^{\prime }\circ f^{\otimes 2} ,\text{ \ \ }\alpha
=f^{-1}\circ \alpha ^{\prime }\circ f,\   \label{(5)}
\end{equation}
\begin{equation}
\text{(resp. }\mu=f^{-1}\circ \mu ^{\prime }\circ f^{\otimes 2},\text{ \ \ }%
\alpha =f^{-1}\circ \alpha \prime \circ f~\ \ and~\ \ \ \eta =f^{-1}\circ
\eta ^{\prime }\text{)}.
\end{equation}

\begin{proposition}
Let $\left( A_{1},\mu ,\eta ,\alpha \right) ,\left( A_{2},\mu ^{\prime
},\eta ^{\prime },\alpha ^{\prime }\right) $ be two unital Hom-associative
algebras.

The maps  $%
\begin{array}{cccc}
i_{1}: & A_{1} & \rightarrow  & A_{1}\otimes A_{2} \\ 
& x & \rightarrow  & i_{1}\left( x\right) =\alpha \left( x\right) \otimes
1_{A}
\end{array}
$  and 
$ 
\begin{array}{cccc}
i_{2}: & A_{2} & \rightarrow  & A_{1}\otimes A_{2} \\ 
& y & \rightarrow  & i_{2}\left( y\right) =1_{A}\otimes \alpha ^{\prime }(y)
\end{array}
$ 
are  morphisms of unital Hom-associative
algebras.

\end{proposition}

\begin{proof}
First we check condition (\ref{(3)}) for the map $i_{1}.$
It holds if and only if 
\begin{equation}
\tilde{\mu}\circ \left( i_{1}\otimes i_{1}\right) =i_{1}\circ \mu ,\text{ }%
i_{1}\circ \alpha =\left( \alpha \otimes \alpha ^{\prime }\right) \circ
i_{1},\text{ }and\text{ }i_{1}\circ \eta =\eta \otimes \eta ^{\prime },
\end{equation}%
where $\tilde{\mu}$ is defined as in Example \ref{A}.
 For all $x_{1},y_{1}\in A_{1},$ we have 
\begin{align*}
\tilde{\mu}\circ \left( i_{1}\otimes i_{1}\right) (x_{1}\otimes y_{1})&
=\left( \ \mu \otimes \mu ^{\prime }\right) \left( id_{A_{1}}\otimes \tau
_{A_{2}\otimes A_{1}}\otimes id_{A_{2}}\right) \left( i_{1}\left(
x_{1}\right) \otimes i_{1}(y_{1})\right)  \\
& =\left( \ \mu \otimes \mu ^{\prime }\right) \left( id_{A_{1}}\otimes \tau
_{A_{2}\otimes A_{1}}\otimes id_{A_{2}}\right) \left( \alpha \left(
x_{1}\right) \otimes 1_{A_{2}}\otimes \alpha \left( y_{1}\right) \otimes
1_{A_{2}}\right)  \\
& =\mu (\alpha \left( x_{1}\right) \otimes \alpha \left( y_{1}\right)
)\otimes \mu ^{\prime }(1_{A_{2}}\otimes 1_{A_{2}}) \\
& =\alpha \left( \mu (x_{1}\otimes y_{1})\right) \otimes 1_{A_{2}} \\
& =i_{1}\circ \mu (x_{1}\otimes y_{1}).
\end{align*}

So the first condition is satisfied.
 For all $x\in A_{1}$, we have%
\begin{equation*}
\begin{array}{c}
\tilde{\alpha}\circ i_{1}\left( x\right) =\left( \alpha \otimes \alpha
^{\prime }\right) \left( \alpha \left( x\right) \otimes 1_{A_{2}}\right) 
=\alpha \left( \alpha \left( x\right) \right) \otimes 1_{A_{2}} 
=i_{1}\circ \alpha \left( x\right) .%
\end{array}%
\end{equation*}

So the second condition is satisfied.

Finally%
\begin{equation*}
\begin{array}{c}
i_{1}\circ \eta \left( 1_{\Bbbk }\right) =1_{A_{1}}\otimes 1_{A_{2}} 
=\eta \left( 1_{\Bbbk }\right) \otimes \eta ^{\prime }\left( 1_{\Bbbk
}\right) ,%
\end{array}%
\end{equation*}

which shows that $i_{1}$ is a unital Hom-associative algebras morphism.
Proof for $i_{2}$ is similar.
\end{proof}

\begin{proposition}
\label{C}(\cite{DY2009-1}) Let $\left( A,\mu
,\eta ,\alpha \right) $ be a unital Hom-associative algebra and $\beta
:A\rightarrow A$ be a weak morphism of Hom-associative algebra, i.e. $%
\beta \circ \mu =\mu \circ \beta ^{\otimes 2},$ $\beta \circ \alpha =\alpha
\circ \beta $, and $\beta \circ \eta =\eta .$ Then $\left( A,\mu _{\beta
}=\beta \circ \mu ,\eta _{\beta }=\beta \circ \eta ,\alpha _{\beta }=\beta
\circ \alpha \right) $ is a unital Hom-associative algebra.

Hence, we denote by $\beta ^{n}$ the $n$-fold composition of $n
$ copies of $\beta $, with $\beta ^{0}=id_{A},$ $\beta ^{n}\circ \mu =\mu
\circ \left( \beta ^{\otimes 2}\right) ^{n}$, then $\left( A,\mu _{\beta
^{n}}=\beta ^{n}\circ \mu ,\eta _{\beta ^{n}}=\beta ^{n}\circ \eta ,\alpha
_{\beta ^{n}}=\beta ^{n}\circ \alpha \right) $ is a unital Hom-associative
algebra.
\end{proposition}

\begin{proof}
We have 
\begin{eqnarray*}
\mu _{\beta }\circ \alpha _{\beta }^{\otimes 2} &=&\left( \beta \circ \mu
\right) \circ \left( \beta \circ \alpha \right) \otimes \left( \beta \circ
\alpha \right) =\left( \beta \circ \mu \right) \circ \left( \beta ^{\otimes
2}\circ \alpha ^{\otimes 2}\right)  \\
&=&\beta \circ \beta \circ \left( \mu \circ \alpha ^{\otimes 2}\right)
=\beta \circ \left( \beta \circ \alpha \right) \circ \mu =\left( \beta \circ
\alpha \right) \circ \left( \beta \circ \mu \right)  \\
&=&\alpha _{\beta }\circ \mu _{\beta }.
\end{eqnarray*}%
Since $\beta $ is a a weak morphism of Hom-associative algebra, so $\alpha
_{\beta }$ is a weak morphism of Hom-associative algebra.

We show that $\left( A,\mu _{\beta },\eta _{\beta },\alpha _{\beta }\right) $
satisfies the Hom-associativity. Indeed%
\begin{eqnarray*}
&& \mu _{\beta }\left( \mu _{\beta }\otimes \alpha _{\beta }\right)  =\left(
\beta \circ \mu \right) \circ \left( \beta \circ \mu \otimes \beta \circ
\alpha \right)  
=\beta \circ \left( \beta ^{\otimes 2}\circ \left( \mu \circ \left( \mu
\otimes \alpha \right) \right) \right)  \\
&&\overset{\eqref{(1)}}{=}\beta \circ \left( \beta ^{\otimes 2}\circ \left( \mu \circ \left( \alpha
\otimes \mu \right) \right) \right)  
=\beta \circ \left( \mu \circ \left( \beta \circ \alpha \otimes \beta
\circ \mu \right) \right)  
=\left( \beta \circ \mu \right) \circ \left( \beta \circ \alpha \otimes
\beta \circ \mu \right)  \\
&&=\mu _{\beta }\left( \alpha _{\beta }\otimes \mu _{\beta }\right) .
\end{eqnarray*}

The second assertion is proved similarly, so $\left( A,\mu _{\beta },\eta
_{\beta },\alpha _{\beta }\right) $ is a unital associative algebra.
\end{proof}
\begin{remark}
In particular, if $\alpha =id_{A},$ one can construct a Hom-associative
algebra starting from an associative algebra and an algebra endomorphism.
\end{remark}

\subsection{Counital Hom-coassociative coalgebras}

We define first the fundamental notion of  Hom-coalgebra, which is dual to
that of a Hom-associative algebra,  in the sense that if we reverse all the
arrows in the defining diagrams of a Hom-associative algebra, we get the concept of a
Hom-coalgebra.

\begin{definition} A Hom-coassociative coalgebra is a triple $\left( C,\Delta ,\beta
\right) $ where $C$ is a $\Bbbk $-vector space and $\Delta :C\longrightarrow
C\otimes C,$ is a linear map, and $\beta :C\longrightarrow C$ \ is a
homomorphism 
satisfying the Hom-coassociativity condition,
\begin{equation}
\left( \Delta \otimes \beta \right) \circ \Delta =\left( \beta \otimes
\Delta \right) \circ \Delta.  \label{(7)}
\end{equation}%
We assume moreover that  $\Delta \circ \beta =\beta ^{\otimes 2}\circ \Delta .$\\
A Hom-coassociative coalgebra is said to be counital if there
exists a linear map $\varepsilon :C\longrightarrow \Bbbk ~\ $ such that $%
\varepsilon \circ \beta =\beta $ and%
\begin{equation}
\left( \varepsilon \otimes id_{C}\right) \circ \Delta =\left( id_{C}\otimes
\varepsilon \right) \circ \Delta =\beta.  \label{(8)}
\end{equation}
\end{definition} 
Conditions (\ref{(7)}) and (\ref{(8)}) are respectively equivalent to
the following commutative diagrams

\begin{equation*}
\begin{array}{ccc}
C & \overset{\ \Delta }{\longrightarrow } & C\otimes C\  \\ 
{\small \ \ \Delta }\downarrow \qquad  &  & \downarrow \beta \otimes {\small %
\Delta } \\ 
\quad C\otimes C & \overset{\Delta \otimes \beta }{\longrightarrow } & 
C\otimes C\mathit{\ }~~%
\end{array}%
\qquad 
\begin{array}{ccccc}
\ \Bbbk \otimes C & \overset{\varepsilon \otimes id_{C}}{\longleftarrow } & 
C\otimes C & \ \overset{\ id_{C}\otimes \varepsilon }{\longrightarrow } & \
C\otimes \Bbbk  \\ 
\cong  &  & \uparrow \mathit{\ }\Delta  &  & \cong  \\ 
C & \overset{\beta }{\longleftarrow } & C\mathit{\ } & \overset{\ \beta }{%
\longrightarrow } & C%
\end{array}%
\end{equation*}

\begin{equation*}
\begin{array}{ccccc}
\ \Bbbk \otimes C & \overset{\varepsilon \otimes id_{C}}{\longleftarrow } & 
C\otimes C & \ \overset{\ id_{C}\otimes \varepsilon }{\longrightarrow } & \
C\otimes \Bbbk  \\ 
\cong  &  & \uparrow \mathit{\ }\Delta  &  & \cong  \\ 
C & \overset{\beta }{\longleftarrow } & C\mathit{\ } & \overset{\ \beta }{%
\longrightarrow } & C.%
\end{array}%
\end{equation*}

\begin{remark}

\begin{enumerate}
\item We recover the classical \ coassociative coalgebra when the twisting
map $\beta $ is the identity map.

\item Given a Hom-coassociative coalgebra $C=\left( C,\Delta ,\beta \right) ,
$ we define the \emph{coopposite} Hom-coassociative coalgebra $%
C^{cop}=(C,\Delta ^{cop},\beta )$ to be  the Hom-coassociative coalgebra with
the same underlying vector space as $C$ and with comultiplication defined
by 
$
\Delta ^{cop}=\tau _{C\otimes C}\circ \Delta .
$

\item A Hom-coassociative coalgebra $\left( C,\Delta ,\beta \right) $ is 
\emph{cocommutative} if and only if $\ \Delta ^{cop}=\Delta .$
\end{enumerate}
\end{remark}

Let $\left( C,\Delta ,\beta )\right) $and $(C^{\prime },\Delta ^{\prime
},\beta ^{\prime })$ be two Hom-coassociative coalgebras. A linear map $%
f:C\rightarrow C^{\prime }$ is a \emph{ Hom-coassociative
coalgebras morphism  }  if%
\begin{equation}
f^{\otimes 2}\circ \Delta =\Delta ^{\prime }\circ f\ \quad \ \ and\ \
f\circ \beta =\beta ^{\prime }\circ f.  \label{9}
\end{equation}

It is said to be a \emph{weak morphism}  if holds only
the first condition. If furthermore the Hom-coassociative coalgebras admit
counits $\varepsilon $ and $\varepsilon ^{\prime }$, we have moreover $%
\varepsilon =\varepsilon ^{\prime }\circ f$.

We say that a Hom-coassociative coalgebra $\left( C,\Delta ,\beta \right) $
is \emph{isomorphic} to a Hom-coassociative coalgebra $(C^{\prime },\Delta
^{\prime },\beta ^{\prime })$ if there exists a bijective Hom-coalgebra
morphism $f:C\longrightarrow C^{\prime }$, and we denote this by $C\cong
C^{\prime }\ $when the context is clear, such that 
\begin{equation*}
\Delta ^{\prime }=f^{\otimes 2}\circ \Delta \circ f^{-1},\ \quad \varepsilon
^{\prime }=\varepsilon \circ f^{-1}~\ and~\ \ \text{ }\beta ^{\prime }=\beta
\circ f^{-1}.
\end{equation*}

Next, we describe the tensor product Hom-coassociative coalgebra
construction.

\begin{proposition}
Let $(C_{1},\Delta _{1},\varepsilon _{1},\beta _{1})$ and $\left(
C_{2},\Delta _{2},\varepsilon _{2},\beta _{2}\right) $ be two counital
Hom-coassociative coalgebras. Then the composite map 
\begin{equation*}
C_{1}\otimes C_{2}\overset{\Delta _{1}\otimes \Delta _{2}}{\longrightarrow }%
\left( C_{1}\otimes C_{1}\right) \otimes \left( C_{2}\otimes C_{2}\right) 
\overset{id_{C_{1}}\otimes \tau _{C_{2}\otimes C_{1}}\otimes id_{C_{2}}}{%
\longrightarrow }\left( C_{1}\otimes C_{2}\right) \otimes \left(
C_{1}\otimes C_{2}\right),
\end{equation*}%
where : $\tau _{C_{2}\otimes
C_{1}}:C_{2}\otimes C_{1}\longrightarrow C_{1}\otimes C_{2}$ is the linear
`twist' map, defines a Hom-coassociative comultiplication $\tilde{\Delta}=\left(
id_{C_{1}}\otimes \tau _{C_{2}\otimes C_{1}}\otimes id_{C_{2}}\right) \circ
\Delta _{1}\otimes \Delta _{2}$ on $C_{1}\otimes C_{2},$ and with  counits $%
\varepsilon _{1}$ of $C_{1}$ and $\varepsilon _{2}$ of $C_{2}$. The map 
$
\varepsilon _{1}\otimes \varepsilon _{2}:C_{1}\otimes C_{2}\longrightarrow
\Bbbk 
$
is a counit of $C_{1}\otimes C_{2}$.
\end{proposition}

\begin{definition}
\label{F} \emph{Tensor product} $C_{1}\otimes C_{2}$ of two counital
Hom-coassociative coalgebras $(C_{1},\Delta _{1},\varepsilon _{1},\beta _{1})
$ and $\left( C_{2},\Delta _{2},\varepsilon _{2},\beta _{2}\right) $ is
defined by 
$
(C_{1}\otimes C_{2},\tilde{\Delta},\tilde{\varepsilon},\tilde{\beta})
$
such that $\ \tilde{\Delta}=\left( id_{C_{1}}\otimes \tau _{C_{2}\otimes
C_{1}}\otimes id_{C_{2}}\right) \circ \Delta \otimes \Delta ^{\prime },$ $
\tilde{\varepsilon}=\varepsilon _{1}\otimes \varepsilon _{2}$ and $\ \tilde{%
\beta}=\beta _{1}\otimes \beta _{2}.$
\end{definition}

\begin{proposition}
\label{D} Let $\left( C,\Delta ,\varepsilon
,\alpha \right) $ be a counital Hom-coassociative coalgebra and $\beta
:C\longrightarrow C$ be a weak morphism of Hom-coassociative coalgebra, i.e. $\Delta \circ \beta =\beta ^{\otimes 2}\circ \Delta $  and $ \varepsilon\circ \beta =\varepsilon $.
Then $\left( C,\Delta _{\beta }=\Delta \circ \beta ,\varepsilon _{\beta
}=\varepsilon \circ \beta ,\alpha _{\beta }=\alpha \circ \beta \right) $ is
a counital Hom-coassociative coalgebra.

Hence, we denote by $\beta ^{n}$ the $n$-fold composition of $n
$ copies of $\beta $, with $\beta ^{0}=id_{C},$ $\Delta \circ \beta
^{n}=\left( \beta ^{\otimes 2}\right) ^{n}\circ \Delta $. Then $\left(
C,\Delta _{\beta ^{n}}=\Delta \circ \beta ^{n},\varepsilon ,\alpha _{\beta ^{n}}=\alpha \circ \beta
^{n}\right) $ is a counital Hom-coassociative coalgebra.
\end{proposition}

\begin{proof}
We have 
\begin{eqnarray*}
\alpha _{\beta }^{\otimes 2}\circ \Delta _{\beta } &=&\left( \beta ^{\otimes
2}\circ \alpha ^{\otimes 2}\right) \circ \left( \Delta \circ \beta \right) 
\\
&=&\beta ^{\otimes 2}\circ \left( \alpha ^{\otimes 2}\circ \Delta \right)
\circ \beta =\beta ^{\otimes 2}\circ \Delta \circ \alpha \circ \beta =\Delta
\circ \beta \circ \alpha \circ \beta  \\
&=&\Delta _{\beta }\circ \alpha _{\beta }.
\end{eqnarray*}%
Since $\beta $ is a weak morphism of Hom-coassociative coalgebra, so $\alpha
_{\beta }$ is a morphism of Hom-coassociative coalgebra.

We show that $\left( C,\Delta _{\beta },\varepsilon ,\alpha _{\beta
}\right) $ satisfies the Hom-coassociativity. Indeed%
\begin{eqnarray*}
&& \left( \Delta _{\beta }\otimes \alpha _{\beta }\right) \circ \Delta _{\beta
} =\left( \Delta \circ \beta \otimes \alpha \circ \beta \right) \circ
\left( \Delta \circ \beta \right)  
=\left( \Delta \otimes \alpha \right) \circ \left( \beta ^{\otimes 2}\circ
\Delta \right) \circ \beta  \\
&&=\left( \left( \Delta \otimes \alpha \right) \circ \Delta \right) \circ
\beta \circ \beta  
=\left( \alpha \otimes \Delta \right) \circ \Delta \circ \beta \circ \beta 
=\left( \alpha \circ \beta \otimes \Delta \circ \beta \right) \circ \Delta
\circ \beta  \\
&&=\left( \alpha _{\beta }\otimes \Delta _{\beta }\right) \circ \Delta
_{\beta }.
\end{eqnarray*}

The second assertion is proved similarly, so $\left( C,\Delta _{\beta
},\varepsilon ,\alpha _{\beta }\right) $ is a counital
Hom-coassociative coalgebra.
\end{proof}

If $\alpha =id_{C}$, this proposition shows how to construct a
Hom-coassociative coalgebra starting from a coalgebra and a coalgebra
morphism  (\cite{AM 2008HBial}). It is a coalgebra version of Proposition \ref{C}.

\begin{theorem}
\label{H}Let $\left( C,\Delta ,\varepsilon ,\beta \right) $ be a counital
Hom-coassociative coalgebra and $C^{\ast }$ be the linear dual of $C$. We
define the maps $\mu :C^{\ast }\otimes C^{\ast }\overset{\rho }{%
\longrightarrow }\left( C\otimes C\right) ^{\ast }\overset{\Delta ^{\ast }}{%
\longrightarrow }C^{\ast },$ $\mu =\Delta ^{\ast }\rho ,$ where $\rho $ is
defined by $\rho \left( f^{\ast }\otimes g^{\ast }\right) \left( m\otimes
n\right) =f^{\ast }\left( m\right) g^{\ast }\left( n\right) $, and $\eta
:\Bbbk \overset{\phi }{\longrightarrow }\Bbbk ^{\ast }\overset{\varepsilon
^{\ast }}{\longrightarrow }C^{\ast },$ $\eta =\varepsilon ^{\ast }\phi $
where $\phi :\Bbbk \longrightarrow \Bbbk ^{\ast }$ is the canonical
isomorphism, and $\eta \left( 1_{\Bbbk }\right) =1_{C^{\ast }}$ where $%
1_{C^{\ast }}\left( x\right) =\varepsilon \left( x\right) $.  The
homomorphism $\alpha :C^{\ast }\longrightarrow C^{\ast }$ is defined as $\alpha \left(
h\right) =h\circ \beta .$ Then $\left( C^{\ast },\mu ,\eta ,\alpha \right) $
is a unital Hom-associative algebra.
\end{theorem}

This is checked in exactly the same way as for Hom-coassociative coalgebras,
as was done in \cite[Corollary 4.12]{AM 2008HBial}.

Conversely,  does a unital Hom-associative algebra $%
\left( A,\mu ,\alpha \right) $ lead to a 
counital Hom-coassociative coalgebra on $A^{\ast }?$ It turns out  that it is not
possible to perform a construction similar to the one of the dual unital
Hom-associative algebra, due to the inexistence of a canonical morphism $%
\left( A\otimes A\right) ^{\ast }\longrightarrow A^{\ast }\otimes A^{\ast }.$
However, if $A$ is finite-dimensional, the canonical morphism $\rho :\left(
A\otimes A\right) ^{\ast }\longrightarrow A^{\ast }\otimes A^{\ast }$ is
bijective.

\begin{theorem}\cite[Corollary 4.12]{AM 2008HBial}
\label{E}Let $\left( A,\mu ,\eta ,\alpha \right) $ be a  finite dimensional unital
Hom-associative algebra  and $A^{\ast }$ be the linear
dual of $A$. We define the comultiplication by the composition%
\begin{equation*}
\Delta :A^{\ast }\overset{\mu ^{\ast }}{\longrightarrow }\left( A\otimes
A\right) ^{\ast }\overset{\rho ^{-1}}{\longrightarrow }A^{\ast }\otimes
A^{\ast }\text{ }by\text{ }\Delta =\rho ^{-1}\mu ^{\ast },
\end{equation*}%
and%
\begin{equation*}
\varepsilon :A^{\ast }\overset{\eta ^{\ast }}{\longrightarrow }\Bbbk ^{\ast }%
\overset{\psi }{\longrightarrow }\Bbbk ,\text{ }\varepsilon =\psi \eta
^{\ast },
\end{equation*}%
where $\psi $ is the canonical isomorphism, $\varepsilon \left( f\right)
=f\left( 1_{A}\right) $ for $f\in A^{\ast },$ where $1_{A}=\eta \left(
1_{\Bbbk }\right) $ and the homomorphism 
\begin{equation*}
\beta :A^{\ast }\longrightarrow A^{\ast },\text{ }\beta \left( h\right)
=h\circ \alpha .
\end{equation*}%
Then $\left( A^{\ast },\Delta ,\varepsilon ,\beta \right) $ is a counital
Hom-coassociative coalgebra.
\end{theorem}
Such a construction could be extended to a so called finite dual.
\begin{proposition}
Let $\left( C,\Delta ,\varepsilon ,\beta \right) $ and $\left( D,\Delta
^{\prime },\varepsilon ^{\prime },\beta ^{\prime }\right) $ be counital
Hom-coassociative coalgebras, and let $\left( A,\mu ,\eta ,\alpha \right) $
and $\left( B,\mu ^{\prime },\eta ^{\prime },\alpha ^{\prime }\right) $ be
finite dimensional unital Hom-associative algebras.

\begin{description}
\item[1)] If $f:C\longrightarrow D$ is a Hom-coassociative
coalgebras morphism, then $f^{\ast }:D^{\ast }\longrightarrow C^{\ast }$ is a  Hom-associative algebras morphism.

\item[2)] If $f:A\longrightarrow B$ is a Hom-associative
algebras morphism, then $f^{\ast }:B^{\ast }\longrightarrow A^{\ast }$ is a  Hom-coassociative coalgebras morphism.
\end{description}
\end{proposition}

\begin{proof}

1) We verify that $f^{\ast }$\ is a Hom-associative algebra morphism. 
Let $d^{\ast },e^{\ast }\in D^{\ast }$ and $c\in C.$ where $\mu _{D^{\ast }},
$ $\eta _{D^{\ast }},$ and $\alpha _{D^{\ast }}$\ are defined as in Theorem %
\ref{H}. We have 
\begin{eqnarray*}
&& f^{\ast }\circ \mu _{D^{\ast }}\left( d^{\ast }\otimes e^{\ast }\right)
\left( c\right)  =\mu _{D^{\ast }}\left( d^{\ast }\otimes e^{\ast }\right)
\left( f\left( c\right) \right) \text{ } 
=\rho \left( d^{\ast }\otimes e^{\ast }\right) \circ \Delta _{D}\circ
f\left( c\right)  \\
&&=\rho \left( d^{\ast }\otimes e^{\ast }\right) \circ \left( f\otimes
f\right) \circ \Delta _{C}\left( c\right) \begin{small}\text{ (}f\text{ is a
Hom-coassociative coalgebra morphism)} \end{small}\\
&&=\rho \left( d^{\ast }\left( f\right) \otimes e^{\ast }\left( f\right)
\right) \Delta _{C}\left( c\right)  
=\mu _{C^{\ast }}\left( f^{\ast }\left( d^{\ast }\right) \otimes f^{\ast
}\left( e^{\ast }\right) \right) \left( c\right)  
=\mu _{C^{\ast }}\left( \left( f^{\ast }\otimes f^{\ast }\right) \left(
d^{\ast }\otimes e^{\ast }\right) \right) \left( c\right) .
\end{eqnarray*}%
Furthermore 
\begin{equation*}
f^{\ast }\circ \eta _{D^{\ast }}\left( 1_{\Bbbk }\right) =f^{\ast }\left(
\varepsilon _{D}\right) =\varepsilon _{D}\left( f\right) =\varepsilon
_{C}=\eta _{C^{\ast }}\left( 1_{\Bbbk }\right) .
\end{equation*}%
For $d^{\ast }\in D^{\ast }$ we have 
\begin{eqnarray*}
f^{\ast }\circ \alpha _{D^{\ast }}\left( d^{\ast }\right)  &=&\alpha
_{D^{\ast }}\left( d^{\ast }\right) \left( f\right) =d^{\ast }\circ \beta
_{D}\left( f\right)  \\
&=&d^{\ast }\circ f\circ \beta _{C} \begin{small}\text{ (}f\text{ is a Hom-coalgebra
morphism)} \end{small} \\
&=&f^{\ast }\left( d^{\ast }\right) \circ \beta _{C}=\alpha _{C^{\ast
}}\circ f^{\ast }\left( d^{\ast }\right) .
\end{eqnarray*}%
Then $f^{\ast }$\ is a Hom-associative algebra morphism.

2) We have to show that the following diagram is commutative%
\begin{equation*}
\begin{array}{ccc}
B^{\ast } & \overset{f^{\ast }}{\longrightarrow } & A^{\ast } \\ 
\Delta _{B^{\ast }}\downarrow  &  & \downarrow \Delta _{A^{\ast }} \\ 
B^{\ast }\otimes B^{\ast } & \overset{f^{\ast }\otimes f^{\ast }}{%
\longrightarrow } & A^{\ast }\otimes A^{\ast }.%
\end{array}%
\end{equation*}%
Let $b^{\ast }\in B^{\ast },$ where $\Delta _{A^{\ast }},$ $\varepsilon
_{A^{\ast }},$ and $\beta _{A^{\ast }}$\ are defined as in Theorem \ref{E}.
We have 
\begin{eqnarray*}
\left( \Delta _{A^{\ast }}\circ f^{\ast }\right) \left( b^{\ast }\right) 
&=&\Delta _{A^{\ast }}\circ \left( b^{\ast }\left( f\right) \right) \text{ }
\\
&=&b^{\ast }\circ f\circ \mu _{A}=b^{\ast }\circ \mu _{B}\left( f\otimes
f\right)  \begin{small}\text{ (}f\text{ is a Hom-associative algebra morphism)}  \end{small}\\
&=&\left( f^{\ast }\otimes f^{\ast }\right) \left( b^{\ast }\circ \mu
_{B}\right) =\left( f^{\ast }\otimes f^{\ast }\right) \Delta _{B^{\ast
}}\left( b^{\ast }\right) ,
\end{eqnarray*}%
which proves the commutativity of the diagram. Also%
\begin{eqnarray*}
\left( \varepsilon _{A^{\ast }}\circ f^{\ast }\right) \left( b^{\ast
}\right)  &=&\varepsilon _{A^{\ast }}\left( b^{\ast }\left( f\right) \right)
=b^{\ast }\left( f\right) \left( 1_{A}\right) =b^{\ast }\circ f\left( \eta
_{A}\left( 1_{\Bbbk }\right) \right)  \\
&=&b^{\ast }\circ \left( \eta _{B}\left( 1_{\Bbbk }\right) \right) =b^{\ast
}\left( 1_{B}\right) =\varepsilon _{B^{\ast }}\left( b^{\ast }\right) ,
\end{eqnarray*}%
and%
\begin{eqnarray*}
f^{\ast }\circ \beta _{B^{\ast }}\left( b^{\ast }\right)  &=&\beta _{B^{\ast
}}\left( b^{\ast }\right) \left( f\right) =b^{\ast }\circ \alpha _{B}\left(
f\right)  \\
&=&b^{\ast }\circ f\circ \alpha _{A}\text{ (}f\text{ is a Hom-coassociative
coalgebra morphism)} \\
&=&f^{\ast }\left( b^{\ast }\right) \circ \alpha _{A}=\beta _{A^{\ast
}}\circ f^{\ast }\left( b^{\ast }\right) .
\end{eqnarray*}%
So $f^{\ast }$\ is a Hom-coassociative coalgebra morphism.
\end{proof}

\begin{proposition}
\label{6}Let $\left( C,\mu ,\eta ,\alpha \right)$ and  $\left( C,\Delta
,\varepsilon ,\alpha \right) $ be respectively  unital Hom-associative algebra and
counital Hom-coassociative coalgebra. The following statements
are equivalent

\begin{enumerate}
\item The maps $\mu $ and $\eta $ are  morphisms of counital
Hom-coassociative coalgebras.

\item The maps $\Delta $ and $\varepsilon $ are  morphisms of unital
Hom-associative algebras.
\end{enumerate}
\end{proposition}

\begin{proof}
Let $\mu $ be a morphism of the Hom-coassociative coalgebra 
\begin{equation*}
\mu :(C\otimes C,\tilde{\Delta},\tilde{\varepsilon},\tilde{\beta}%
)\longrightarrow \left( C,\Delta ,\varepsilon ,\alpha \right) 
\end{equation*}%
such that $\tilde{\Delta}=\tau _{2,3}\circ \Delta \otimes \Delta ,$ $\tilde{%
\varepsilon}=\varepsilon \otimes \varepsilon ,$ $\tilde{\beta}=\alpha
\otimes \alpha $, then the morphism $\mu $ satisfies  Axiom (\ref{9})%
\begin{equation}
(\mu \otimes \mu )\circ \tau _{2,3}\circ \Delta \otimes \Delta =\Delta \circ
\mu ,~and\ \mu \circ \left( \alpha \otimes \alpha \right) =\alpha \circ \mu .
\label{(11)}
\end{equation}%
And $\eta $ is a morphism of the Hom-coassociative coalgebra 
\begin{equation*}
\eta :\left( \Bbbk ,id_{\Bbbk ,\Bbbk ^{\otimes 2}},id_{\Bbbk ,\Bbbk
},id_{\Bbbk ,\Bbbk }\right) \longrightarrow \left( C,\Delta ,\varepsilon
,\alpha \right) ,
\end{equation*}%
then%
\begin{equation}
\left( \eta \otimes \eta \right) \circ id_{\Bbbk ,\Bbbk ^{\otimes 2}}=\Delta
\circ \eta ,\qquad \eta \circ id_{\Bbbk ,\Bbbk }=\alpha \circ \eta ,
\label{(12)}
\end{equation}

by  relations (\ref{(11)}), (\ref{(12)}) and$\ \alpha $ is a
homomorphism, the comultiplication 
\begin{equation*}
\Delta :\left( C,\mu ,\eta ,\alpha \right) \longrightarrow \left( C\otimes C,%
\tilde{\mu},\tilde{\eta},\tilde{\alpha}\right) 
\end{equation*}%
is a morphism of Hom-associative algebra, such that 
\begin{equation*}
\tilde{\mu}=(\mu \otimes \mu )\circ \tau _{2,3},\quad \tilde{\eta}=\eta
\otimes \eta \ and\ \tilde{\alpha}=\alpha \otimes \alpha .
\end{equation*}

By the relations (\ref{(11)}), (\ref{(12)}) and$\ \alpha $ is a
homomorphism, the counit%
\begin{equation*}
\varepsilon :\left( C,\mu ,\eta ,\alpha \right) \longrightarrow \left( \Bbbk
,id_{C,\Bbbk ^{\otimes 2}},id_{C,\Bbbk },id_{C,\Bbbk }\right) 
\end{equation*}%
is a morphism of Hom-associative algebra.
\end{proof}

\begin{lemma}
\label{lemma1}Let $\left( C,\Delta ,\varepsilon ,\beta \right) $ be a
counital Hom-coassociative coalgebra and  $f:C\rightarrow C$ be a linear map 
which commutes with $\beta $ and  satisfies $\beta \circ f=f\circ \beta .$ Then

\begin{enumerate}
\item $\left( \beta \otimes \left( id_{C}\otimes f\right) \circ \Delta
\right) \circ \Delta =(\Delta \otimes \left( \beta \circ f\right) )\circ
\Delta ;$

\item $(\beta \otimes \left( f\otimes id_{C}\right) \circ \Delta )\circ
\Delta =(\left( \left( id_{C}\otimes f\right) \circ \Delta \right) \otimes
\beta )\circ \Delta ;$

\item $((f\otimes id_{C})\circ \Delta \otimes \beta )\circ \Delta =(\left(
\beta \circ f\right) \otimes \Delta )\circ \Delta .$
\end{enumerate}
\end{lemma}

\begin{proof}
We have 

\begin{enumerate}
\item The first equality%
\begin{align*}
& (\beta \otimes \left( id_{C}\otimes f\right) \circ \Delta )\circ \Delta 
=(id_{C}\otimes \left( id_{C}\otimes f\right) \circ (\beta \otimes \Delta
)\circ \Delta  \\
& \overset{\eqref{(7)}}{=} ((id_{C}\otimes id_{C})\otimes f)\circ (\Delta \otimes \beta )\circ
\Delta   \
 =((id_{C}\otimes id_{C})\circ \Delta \otimes f\circ \beta )\circ \Delta  \\
& =(\Delta \otimes (\beta \circ f))\circ \Delta .
\end{align*}

\item The second equality%
\begin{align*}
& (\beta \otimes \left( f\otimes id_{C}\right) \circ \Delta )\circ \Delta 
=(id_{C}\otimes \left( f\otimes id_{C}\right) (\beta \otimes \Delta )\circ
\Delta  \\
& \overset{\eqref{(7)}}{=} ((id_{C}\otimes f)\otimes id_{C})(\Delta \otimes \beta )\circ \Delta
 =((id_{C}\otimes f)\circ \Delta \otimes \beta )\circ \Delta .
\end{align*}

\item The third equality%
\begin{align*}
& ((f\otimes id_{C})\circ \Delta \otimes \beta )\circ \Delta  \overset{\eqref{(7)}}{=} ((f\otimes
id_{C})\otimes id_{C})(\Delta \otimes \beta )\circ \Delta   
 =(f\otimes (id_{C}\otimes id_{C}))(\beta \otimes \Delta )\circ \Delta  
\\
& =(f\circ \beta \otimes (id_{C}\otimes id_{C})\circ \Delta )\circ \Delta   =(f\circ \beta \otimes \Delta )\circ \Delta  
 =(\beta \circ f\otimes \Delta )\circ \Delta .
\end{align*}
\end{enumerate}

This finishes the proof
\end{proof}


\begin{lemma}\label{4} Let $\left( C,\Delta ,\varepsilon ,\beta )\right) $be a
Hom-counital coassociative coalgebra and $f$ be  a linear map $f:C\rightarrow
C^{\otimes m}$ satisfying $f\circ \beta =\beta ^{\otimes m}\circ f,$ then

\begin{enumerate}
\item $((\Delta \otimes \beta ^{\otimes n})(\Delta \otimes \beta ^{\otimes
(n-1)})\circ f=\left( \left( \beta \otimes \Delta \right) \circ \Delta
\otimes \left( \beta \circ \beta \right) ^{\otimes (n-1)}\right) \circ f;$

\item $\left( \beta ^{n}\otimes \left( \beta ^{n-1}\otimes f\right) \circ
\Delta \right) \circ \Delta =\left( \Delta \otimes \beta ^{\otimes m}\right)
\left( \beta ^{n-1}\otimes f\right) \circ \Delta ;$

\item $\left( \left( \left( f\otimes \beta ^{n-1}\right) \circ \Delta
\right) \otimes \beta ^{n}\right) \circ \Delta =\left( \left( \beta
^{\otimes m}\otimes \Delta \right) \left( f\otimes \beta ^{n-1}\right)
\right) \circ \Delta .$
\end{enumerate}
\end{lemma}

\begin{proof}
This proof is completely analogous to that of Lemma \ref{lemma1}.

\begin{enumerate}
\item The first equality%
\begin{align*}
&((\Delta \otimes \beta ^{\otimes n})(\Delta \otimes \beta ^{\otimes
(n-1)})\circ f =\left( \Delta \otimes \beta \otimes \beta ^{\otimes
(n-1)}\right) (\Delta \otimes \beta ^{\otimes (n-1)})\circ f \\ &
 \overset{\eqref{(7)}}{=}\left( \left( \Delta \otimes \beta \right) \Delta \otimes \left( \beta
\circ \beta \right) ^{\otimes (n-1)}\right) \circ f  
 =\left( \left( \beta \otimes \Delta \right) \Delta \otimes \left( \beta
\circ \beta \right) ^{\otimes (n-1)}\right) \circ f.
\end{align*}

\item The second equality%
\begin{align*}
& \left( \beta ^{n}\otimes \left( \beta ^{n-1}\otimes f\right) \circ \Delta
\right) \circ \Delta  =(\left( \beta ^{n-1}\circ \beta )\otimes \left(
\beta ^{n-1}\otimes f\right) \circ \Delta \right) \circ \Delta  \\ &
 \overset{\eqref{(7)}}{=}\left( \beta ^{n-1}\otimes (\beta ^{n-1}\otimes f))\circ (\beta \otimes
\Delta \right) \circ \Delta   
 =\left( \beta ^{n-1}\otimes (\beta ^{n-1}\otimes f))\circ (\Delta \otimes
\beta \right) \circ \Delta  \\ &
 =((\beta ^{n-1}\otimes \beta ^{n-1})\circ \Delta \otimes f\circ \beta
)\circ \Delta  
 =(\Delta \circ \beta ^{n-1}\otimes \beta ^{\otimes m}\circ f)\circ \Delta 
\\
& =(\Delta \otimes \beta ^{\otimes m})\circ (\beta ^{n-1}\otimes f)\circ
\Delta .
\end{align*}

\item The third equality%
\begin{align*}
& \left( \left( \left( f\otimes \beta ^{n-1}\right) \circ \Delta \right)
\otimes \beta ^{n}\right) \circ \Delta  =\left( \left( \left( f\otimes
\beta ^{n-1}\right) \circ \Delta \right) \otimes \beta ^{n-1}\left( \beta
\right) \right) \circ \Delta  \\
& \overset{\eqref{(7)}}{=}\left( \left( f\otimes \beta ^{n-1}\right) \otimes \beta ^{n-1}\right)
\left( \Delta \otimes \beta \right) \circ \Delta  
 =\left( f\otimes \left( \beta ^{n-1}\otimes \beta ^{n-1}\right) \right)
\left( \beta \otimes \Delta \right) \circ \Delta  \\
& =\left( f\left( \beta \right) \otimes \left( \beta ^{n-1}\right) ^{\otimes
2}\Delta \right) \circ \Delta  
 =\left( \beta ^{\otimes m}\left( f\right) \otimes \Delta \left( \beta
^{n-1}\right) \right) \circ \Delta  \\
& =\left( \left( \beta ^{\otimes m}\otimes \Delta \right) \left( f\otimes
\beta ^{n-1}\right) \right) \circ \Delta .
\end{align*}
\end{enumerate}

This finishes the proof.
\end{proof}

This Lemma will be used in the proof of Proposition \ref{S}.

\begin{remark}
\label{5}The following identity is an immediate consequence of Lemma \ref{4} 
\begin{equation*}
\left( \left( \left( \beta ^{n-1}\otimes f\right) \circ \Delta \right)
\otimes \beta ^{n}\right) \circ \Delta =\left( \left( \beta ^{n}\otimes
\left( f\otimes \beta ^{n-1}\right) \right) \circ \Delta \right) \circ
\Delta .
\end{equation*}
\end{remark}

\subsection{Hom-bialgebras and Hom-Hopf algebras}


The notion of  Hom-bialgebra was introduced in \cite{AMS 2007Hbial,AM 2008HBial}, see also \cite{DY 2010}. 
\begin{definition}
A \emph{Hom-bialgebra} is a tuple $\left( B,\mu ,\eta ,\alpha
,\Delta ,\varepsilon ,\beta \right) $ in which $\left( B,\mu ,\eta ,\alpha
\right) $ is a unital Hom-associative algebra, $\left( B,\Delta
,\varepsilon ,\beta \right) $ is a counital Hom-coassociative coalgebra and  the
linear maps $\Delta $ and $\varepsilon $ are morphisms of Hom-associative
algebras, that is%
\begin{equation}
\Delta \circ \mu =\mu ^{\otimes 2}\circ \tau _{2,3}\circ \Delta ^{\otimes 2}%
\text{ and }\varepsilon \otimes \varepsilon =\varepsilon \circ \mu .
\label{(14)}
\end{equation}
\end{definition}
\begin{remark}
\begin{enumerate}
\item (\cite{AMS 2007Hbial}) In terms of elements,  condition ($\ref{(14)}$)
could be expressed by the following identities :

$\left\{ 
\begin{array}{c}
\Delta \left( 1_{B}\right) =1_{B}\otimes 1_{B},\ \text{ }\alpha \left(
1_{B}\right) =1_{B},\text{ }and\text{ }\beta \left( 1_{B}\right) =1_{B},\
where\ 1_{B}=\eta \left( 1_{\Bbbk }\right),  \\ 
\Delta \left( \mu \left( x\otimes y\right) \right) =\Delta \left( x\right)
\cdot \Delta \left( y\right) =\sum\limits_{\left( x\right) \left( y\right)
}\mu \left( x_{1}\otimes y_{1}\right) \otimes \mu \left( x_{2}\otimes
y_{2}\right) ,\text{ \ \ \ \ \ \ \ \ \ \ \ } \\ 
\varepsilon \left( 1_{B}\right) =1_{\Bbbk },\quad  
\varepsilon \left( \mu \left( x\otimes y\right) \right) =\varepsilon \left(
x\right) \varepsilon \left( y\right),\quad  
\varepsilon \circ \alpha \left( x\right) =\varepsilon \left( x\right) .\text{
\ }%
\end{array}%
\right.  $

where  the dot "$\cdot $"
denotes the multiplication on tensor product.

\item If $\alpha =\beta $ the Hom-bialgebra is denoted $\left( B,\mu ,\eta
,\Delta ,\varepsilon ,\alpha \right) .$ 

\item Observe that a Hom-bialgebra is neither associative nor coassociative,
unless $\alpha =\beta =id_{B},$ in which case we have a bialgebra.
\end{enumerate}
\end{remark}

Compatibility conditions could be formulated in a different way according to 
Proposition \ref{6}.\\

A \emph{morphism of Hom-bialgebra} (resp. weak morphism of Hom-bialgebra)
which is either a morphisms (resp. weak morphism) of Hom-associative algebra
and Hom-coassociative coalgebra.

\begin{example}
Let $G$ be a group and $\Bbbk G$  the corresponding  group algebra over $\Bbbk$. As a vector space, $%
\Bbbk G$ is generated by $\left\{ e_{g}:g\in G\right\} $. If $\alpha
:G\longrightarrow G$ is a group homomorphism, then it can be extended to an
algebra endomorphism of $\Bbbk G$ by setting%
\begin{equation*}
\alpha \left( \sum_{g\in G}a_{g}e_{g}\right) =\sum_{g\in G}a_{g}\alpha
\left( e_{g}\right) =\sum_{g\in G}a_{g}e_{\alpha \left( g\right) }.
\end{equation*}%
Consider the usual bialgebra structure on $\Bbbk G$ and $\alpha $ a bialgebra
morphism. Then, we define a Hom-bialgebra $\left( \Bbbk G,\mu
,\Delta ,\alpha \right) $ over $\Bbbk G$ by setting:%
\begin{equation*}
\mu \left( e_{g}\otimes e_{g^{\prime }}\right) =\alpha \left( e_{gg^{\prime
}}\right) ,\text{ \ }\Delta \left( e_{g}\right) =\alpha \left( e_{g}\right)
\otimes \alpha \left( e_{g}\right) .
\end{equation*}
\end{example}

Combining Propositions \ref{C} and \ref{D}, we obtain the following
proposition:

\begin{proposition}
\label{TwistBialg} Let $B=\left( B,\mu ,\eta ,\Delta ,\varepsilon ,\alpha
\right) $ be a Hom-bialgebra and $\beta :B\longrightarrow B$ be  a Hom-bialgebra morphism,
 then $B_{\beta }=\left( B,\mu _{\beta },\eta _{\beta },\Delta
_{\beta },\varepsilon _{\beta },\alpha _{\beta }\right) $ is a Hom-bialgebra.

Hence $\left( B,\mu _{\beta ^{n}},\eta _{\beta ^{n}},\Delta
_{\beta ^{n}},\varepsilon _{\beta ^{n}},\alpha _{\beta ^{n}}\right) $ is a
Hom-bialgebra.
\end{proposition}

\begin{proof}
According to  Propositions \ref{C} and \ref{D},  $\left( B,\mu _{\beta },\eta_{\beta } ,\alpha_{\beta }
\right) $ is a unital Hom-associative algebra, and $\left( B,\Delta_{\beta }
,\varepsilon_{\beta } ,\alpha_{\beta } \right) $ is a counital Hom-coassociative coalgebra. It
remains to establish condition (\ref{(14)}) for $B_{\beta }$. Using $\mu
_{\beta }=\beta \circ \mu =\mu \circ \beta ^{\otimes 2},$ $\Delta _{\beta
}=\Delta \circ \beta =\Delta \circ \beta ^{\otimes 2}$, $\tau \circ \beta
^{\otimes 2}=\beta ^{\otimes 2}\circ \tau ,$ and the condition (\ref{(14)})
for the Hom-bialgebra $B$, we compute as follows:%
\begin{eqnarray*}
&& \Delta _{\beta }\circ \mu _{\beta } =\Delta \circ \beta \circ \beta \circ
\mu  
=\beta ^{\otimes 2}\circ \beta ^{\otimes 2}\circ \Delta \circ \mu  
=\beta ^{\otimes 2}\circ \beta ^{\otimes 2}\circ \mu ^{\otimes 2}\circ
\tau _{2,3}\circ \Delta ^{\otimes 2} \\
&&= \beta ^{\otimes 2}\circ \left( \beta \circ \mu \right) ^{\otimes 2}\circ
\tau _{2,3}\circ \Delta ^{\otimes 2} 
=\left( \beta \circ \mu \circ \beta ^{\otimes 2}\right) ^{\otimes 2}\circ
\tau _{2,3}\circ \Delta ^{\otimes 2} \\
&&=\mu _{\beta }^{\otimes 2}\circ \beta ^{\otimes 4}\circ \tau _{2,3}\circ
\Delta ^{\otimes 2} 
=\mu _{\beta }^{\otimes 2}\circ \tau _{2,3}\circ \beta ^{\otimes 4}\circ
\Delta ^{\otimes 2} 
=\mu _{\beta }^{\otimes 2}\circ \tau _{2,3}\circ \Delta _{\beta }^{\otimes
2}.
\end{eqnarray*}

We have shown that $B_{\beta }$ is a Hom-bialgebra.
\end{proof}

\begin{example}
[Hom-Type Taft-Sweedler bialgebra] \label{Taft} We consider $T_{2}$, the
4-dimensional unital Taft-Sweedler algebra generated by $g,x$ and the
relations $(g^{2}=1,\ x^{2}=0,\ xg=-gx).$ The comultiplication is defined by 
$\Delta (g)=g\otimes g$ and $\Delta (x)=x\otimes 1+g\otimes x$, the counit
is given by $\varepsilon (g)=1,\ \varepsilon (x)=0.$ Set $\{e_{1}=1,\
e_{2}=g,\ e_{3}=x,\ e_{4}=gx\}$ be a basis.


Pick any $\lambda \in \mathbb{K}$, the map $\alpha :T_{2}\rightarrow T_{2}$
defined by $\alpha (e_1)=e_1,\ \alpha (e_2)=e_2,\ \alpha (e_3)=\lambda e_3$, $\alpha
(e_4)=\lambda e_4$ is a bialgebra morphism. Therefore, we obtain a
Hom-bialgebra $(T_{2})_{\lambda }$ which is defined by the following table
that describes multiplying the $i$th row elements by the $j$th column
elements.

\begin{equation*}
\begin{array}{|c|c|c|c|c|}
\hline
\  & e_{1} & e_{2} & e_{3} & e_{4} \\ \hline
e_{1} & e_{1} & e_{2} & \lambda e_{3} & \lambda e_{4} \\ \hline
e_{2} & e_{2} & e_{1} & \lambda e_{4} & \lambda e_{3} \\ \hline
e_{3} & \lambda e_{3} & -\lambda e_{4} & 0 & 0 \\ \hline
e_{4} & \lambda e_{4} & -\lambda e_{3} & 0 & 0 \\ \hline
\end{array}%
\end{equation*}%
and 
\begin{equation*}
\Delta (e_{1})=e_{1}\otimes e_{1},\Delta (e_{2})=e_{2}\otimes e_{2},\Delta
(e_{3})=\lambda (e_{3}\otimes e_{1}+e_{2}\otimes e_{3}),\Delta
(e_{4})=\lambda (e_{4}\otimes e_{2}+e_{1}\otimes e_{4}).
\end{equation*}%
\begin{equation*}
\varepsilon (e_{1})=\varepsilon (e_{2})=1,\quad \varepsilon
(e_{3})=\varepsilon (e_{4})=0.
\end{equation*}%
\end{example}

\begin{example}
\begin{enumerate}
\item A unital Hom-associative algebra $\left( A,\mu ,\eta ,\alpha \right) $
becomes a Hom-bialgebra when equipped with the trivial comultiplication $%
\Delta =0$. Likewise, a counital Hom-coassociative coalgebra $\left(
C,\Delta ,\varepsilon ,\beta \right) $ becomes a Hom-bialgebra when equipped
with the trivial multiplication $\mu =0$.

\item Let $\left( B,\mu ,\eta ,\Delta ,\varepsilon ,\alpha \right) $ be a
Hom-bialgebra.  Then so are\\  $\left( B,-\mu ,\eta ,-\Delta ,\varepsilon
,\alpha \right) ,$ and $\left( B,\mu ^{op},\eta ,\Delta ^{cop},\varepsilon
,\alpha \right) $ where $\mu ^{op}=\mu \circ \tau _{1,2}$ and $\Delta
^{cop}=\tau _{1,2}\circ \Delta $.
\end{enumerate}
\end{example}

\begin{proposition}
\label{J}Let $B=\left( B,\mu ,\eta ,\Delta ,\varepsilon ,\alpha \right) $ be
a finite dimensional Hom-bialgebra. \\ Then $B^{\ast }=\left( B^{\ast },\Delta
^{\ast },\varepsilon ^{\ast },\mu ^{\ast },\eta ^{\ast },\alpha ^{\ast
}\right) $ is a Hom-bialgebra, together with the Hom-associative algebra
structure which is dual to the Hom-coassociative coalgebra structure of $B,$
and with the Hom-coassociative coalgebra structure which is a dual to the
Hom-associative algebra structure of $B$, is a Hom-bialgebra called
 dual Hom-bialgebra of $B.$
\end{proposition}

\begin{proposition}
If $B$ is a finite dimensional Hom-bialgebra, then $B$ is cocommutative if
and only if $B^{\ast }$ is commutative, and $B$ is commutative if and only
if $B^{\ast }$ is cocommutative.
\end{proposition}

In order to define Hom-Hopf algebras, we define first a convolution product.

\begin{proposition}
Let $(A,\mu ,\eta ,\alpha )$ be a unital Hom-associative algebra and $%
\left( C,\Delta ,\varepsilon ,\beta \right) $ be a counital
Hom-coassociative coalgebra. Then, the vector space $Hom(C,A)$ of  $\Bbbk $%
-linear mappings of $C$ to $A$ equipped with the \emph{convolution product}
defined by 
\begin{equation*}
\left( f\ast g\right) \left( x\right) =\mu \circ \left( f\otimes g\right)
\circ \Delta \left( x\right) \quad x\in C,
\end{equation*}%
and the unit being $\eta \circ \varepsilon $ is a unital Hom-associative
algebra with the homomorphism map defined by $\gamma \left( f\right) =\alpha
\circ f\circ \beta $.
\end{proposition}

\begin{proposition}
\label{B}Let $(C,\Delta ,\varepsilon ,\beta )$ be a counital
Hom-coassociative coalgebra and $(A,\mu ,\eta ,\alpha )$ a unital
Hom-associative algebra. Then the vector space $Hom(C\otimes C,A)$ and $%
Hom(C,A\otimes A)$ are  unital Hom-associative algebras.
\end{proposition}

\begin{proof}
Consider $C\otimes C$ with the tensor product of Hom-coassociative
coalgebras structure (resp. $A\otimes A$\ with the tensor product of
Hom-associative algebras structure)  and $A$ with the Hom-associative
algebra structure (resp. $C$ with the Hom-coassociative coalgebra
structure). Then it makes sense to speak about the unital Hom-associative
algebra $Hom(C\otimes C,A)$ (resp. unital Hom-associative algebra $%
Hom(C,A\otimes A)$)$,$ with the multiplication given by convolution, defined
by 
\begin{equation*}
\left( f\ast g\right) \left( x\otimes y\right) =\mu \circ \left( f\otimes
g\right) \circ \tilde{\Delta}\left( x\otimes y\right) \text{ where }\tilde{%
\Delta}=\left( id_{C}\otimes \tau \otimes id_{C}\right) \circ \Delta \otimes
\Delta ,\text{ }
\end{equation*}%
\begin{equation*}
\text{(resp.}\left( f\ast g\right) \left( x\right) =\tilde{\mu}\circ \left(
f\otimes g\right) \circ \Delta \left( x\right) \text{ where }\tilde{\mu}%
=\left( \mu \otimes \mu \right) \left( id_{A}\otimes \tau \otimes
id_{A}\right) .
\end{equation*}%
The identity element of the Hom-associative algebra $Hom(C\otimes C,A)$ is $%
\eta \circ \left( \varepsilon \otimes \varepsilon \right) :C\otimes
C\longrightarrow A$ (resp. the unit of  Hom-associative algebra $%
Hom(C,A\otimes A)$ is $\left( \eta \otimes \eta \right) \circ \varepsilon
:C\longrightarrow A\otimes A$).
\end{proof}\\

Now let $\left( B,\mu ,\eta ,\Delta ,\varepsilon ,\alpha \right) $ be a
Hom-bialgebra.

An endomorphism $S$ is an \emph{antipode} if it is the inverse of the
identity over $B$ for the unital Hom-associative algebra $Hom(B,B)$ with the
multiplication given by the convolution product. The unit being $\eta \circ
\varepsilon $, (recall that concatenation denotes composition of maps).

The conditions may be expressed by the identities :%
\begin{equation}
\mu \circ \left( id_{B}\otimes S\right) \circ \Delta =\mu \circ \left(
S\otimes id_{B}\right) \circ \Delta =\eta \circ \varepsilon .  \label{(15)}
\end{equation}

Condition  (\ref{(15)}) means that $S$ is the convolution
inverse of the identity mapping, that is, 
\begin{equation*}
S\ast id_{V}=id_{V}\ast S=\eta \circ \varepsilon .
\end{equation*}

\begin{definition} A \emph{Hom-Hopf algebra} is a Hom-bialgebra with an antipode. It is
denoted by the tuple $\left( B,\mu ,\eta ,\Delta ,\varepsilon ,\alpha
,S\right) .$

Let $H$ and $H^{\prime }$ be two Hom-Hopf algebras. A map $%
f:H\longrightarrow H^{\prime }$ is a called a \emph{ Hom-Hopf
algebras morphism} if it is a  Hom-bialgebras morphism.
\end{definition}
It is natural to ask whether a morphism of Hom-Hopf algebra should preserve
antipodes. The following result shows that this is indeed the case. Let $%
\left( H,\mu ,\eta ,\Delta ,\varepsilon ,\alpha ,S\right) $ be a Hom-Hopf
algebra.
For any element $x\in H$ , using the counity and Sweedler notation, one may
write%
\begin{equation}
\alpha \left( x\right) =\sum_{\left( x\right) }x_{\left( 1\right) }\otimes
\varepsilon \left( x_{\left( 2\right) }\right) =\sum_{\left( x\right)
}\varepsilon \left( x_{\left( 1\right) }\right) \otimes x_{\left( 2\right) }.
\label{101}
\end{equation}

Then, for any $f\in End_{\Bbbk }(H)$, we have%
\begin{equation}
f\circ \alpha \left( x\right) =\sum_{\left( x\right) }f\left( x_{\left(
1\right) }\right) \varepsilon \left( x_{\left( 2\right) }\right)
=\sum_{\left( x\right) }\varepsilon \left( x_{\left( 1\right) }\right)
f\left( x_{\left( 2\right) }\right) .  \label{102}
\end{equation}%
The convolution product of $f,g\in End_{\Bbbk }(H)$. One may write%
\begin{equation}
\left( f\ast g\right) \left( x\right) =\sum_{\left( x\right) }\mu \left(
f\left( x_{\left( 1\right) }\right) \otimes g\left( x_{\left( 2\right)
}\right) \right) \quad x\in H.  \label{103}
\end{equation}%
Since the antipode $S$ is the inverse of the identity for the convolution
product, then $S$ satisfies%
\begin{equation}
\sum \mu \left( S\left( x_{\left( 1\right) }\right) \otimes x_{\left(
2\right) }\right) =\sum \mu \left( x_{\left( 1\right) }\otimes S\left(
x_{\left( 2\right) }\right) \right) =\varepsilon \left( x\right) \eta \left(
1_{\Bbbk }\right) \text{ \ \ }\text{ for all }\text{ }x\in H.  \label{104}
\end{equation}

\begin{example} Let $\left( H,\mu ,\eta ,\Delta ,\varepsilon ,\alpha ,S\right) $ is a 
Hom-Hopf algebra and $\beta:H\rightarrow H$ be a Hom-bialgebra morphism.  Then, the antipode $S$ satisfies%
\begin{eqnarray*}
\mu \circ \left( id_{H}\otimes S\right) \circ \Delta  &=&\mu \circ \left(
S\otimes id_{H}\right) \circ \Delta =\eta \circ \varepsilon  \\
\beta \circ \mu \circ \left( id_{H}\otimes S\right) \circ \Delta \circ
\beta  &=&\beta \circ \mu \circ \left( S\otimes id_{H}\right) \circ \Delta
\circ \beta =\beta \circ \eta \circ \varepsilon \circ \beta  \\
\mu _{\beta }\circ \left( id_{H}\otimes S\right) \circ \Delta _{\beta }
&=&\mu _{\beta }\circ \left( S\otimes id_{H}\right) \circ \Delta _{\beta
}=\eta \circ \varepsilon .
\end{eqnarray*}%
Therefore,  $\left( H,\mu _{\beta },\eta ,\Delta _{\beta },\varepsilon ,\alpha
,S\right) $ is also a Hom-Hopf algebra.


\end{example}
\begin{proposition}
Let $\left( H,\mu ,\eta ,\Delta ,\varepsilon ,\alpha ,S\right) $ and $\left(
H^{\prime },\mu ^{\prime },\eta ^{\prime },\Delta ^{\prime },\varepsilon
^{\prime },\alpha ^{\prime },S^{\prime }\right) $ be two Hom-Hopf algebras
with antipodes $S$ and $S^{\prime }.$ If $f:H\longrightarrow H^{\prime }$ is
a morphism of Hom-bialgebras, then 
\begin{equation}
S^{\prime }\circ f=f\circ S.  \label{105}
\end{equation}
\end{proposition}

\begin{proof}
Consider the Hom-algebra $Hom\left( H,H^{\prime }\right) $ with the
convolution product, and  elements $S^{\prime }\circ f$ and $f\circ S$
from this Hom-associative algebra. We show that they are
equal. Indeed%
\begin{eqnarray*}
&& \left( \left( S^{\prime }\circ f\right) \ast f\right) \left( x\right) 
=\mu ^{\prime }\circ \left( S^{\prime }\circ f\otimes f\right) \circ
\Delta \left( x\right)  
=\mu ^{\prime }\circ \left( S^{\prime }\otimes id_{H}\right) \left(
f\otimes f\right) \circ \Delta \left( x\right)  \\
&& =\mu ^{\prime }\circ \left( S^{\prime }\otimes id_{H}\right) \circ \Delta
^{\prime }\circ f\left( x\right) 
= \eta ^{\prime }\circ \varepsilon ^{\prime }\circ f\left( x\right)
=\varepsilon ^{\prime }\left( f\left( x\right) \right) \eta ^{\prime }\left(
1_{\Bbbk }\right)  
 =\varepsilon \left( x\right) \eta ^{\prime }\left( 1_{\Bbbk }\right) .
\end{eqnarray*}

So $S^{\prime }\circ f$ is a left inverse for $f$. Also 
\begin{eqnarray*}
&&\left( f\ast \left( f\circ S\right) \right) \left( x\right)  =\mu ^{\prime
}\circ \left( f\otimes f\circ S\right) \circ \Delta \left( x\right)  
=\mu ^{\prime }\circ \left( f\otimes f\right) \left( id_{H}\otimes
S\right) \circ \Delta \left( x\right)  \\
&&=f\circ \mu \circ \left( id_{H}\otimes S\right) \circ \Delta \left(
x\right)  
=f\circ \eta \circ \varepsilon \left( x\right) =\eta ^{\prime }\circ
\varepsilon \left( x\right)  
=\varepsilon \left( x\right) \eta ^{\prime }\left( 1_{\Bbbk }\right) .
\end{eqnarray*}

Hence $f\circ S$ is also a right inverse for $f$. It follows that $f$ is
(convolution) invertible, and that the left and right inverses are equal.
\end{proof}

\begin{remark}

\begin{enumerate}
\item Since $\alpha $ is a Homomorphism of Hom-bialgebra, so%
\begin{equation}
S\circ \alpha =\alpha \circ S.  \label{106}
\end{equation}

\item The antipode $S$ of a Hom-Hopf algebra is unique.
\end{enumerate}
\end{remark}

The next proposition gives some important properties of the antipode (see 
\cite{SC 2009}, \cite{AM 2008HBial}). We show  that the antipode of a Hom-Hopf algebra 
is an anti-morphism of Hom-algebras and 
anti-morphism of Hom-coalgebras. This means that $S:H\longrightarrow H^{op}$
is a Hom-algebra morphism and $S:H\longrightarrow H^{cop}$ is a
Hom-coalgebra morphism.

\begin{proposition}
Let $H=\left( H,\mu ,\eta ,\Delta ,\varepsilon ,\alpha ,S\right) $ be a
Hom-Hopf algebra with antipode $S$. Then the following identities hold

\begin{enumerate}
\item $S\left( x\cdot y\right) =S\left( y\right) \cdot S\left( x\right) \
\ \ \ or\ \ \ S\circ \mu =\mu \circ \left( S\otimes S\right) \circ \tau ;$

\item$S\circ \eta =\eta ;$

\item $\Delta \left( S\left( x\right) \right) =S\left( x_{\left(
2\right) }\right) \otimes S\left( x_{\left( 1\right) }\right) \ or\ \Delta
\circ S=\left( S\otimes S\right) \circ \tau \circ \Delta ;$

\item $\varepsilon \circ S=\varepsilon $.
\end{enumerate}
\end{proposition}

\begin{proof}
(1) Consider $H\otimes H$ with the tensor product of Hom-coassociative
coalgebra structure $\tilde{\Delta}_{\alpha }$, $H$ with the Hom-associative
algebra structure $\mu $ and a Hom-associative algebra $Hom\left( H\otimes
H,H\right) $ with the multiplication given by convolution, (see Proposition $%
\ref{B})$%
\begin{equation*}
f\ast g\left( x\otimes y\right) =\mu \circ \left( f\otimes g\right) \circ 
\tilde{\Delta}_{\alpha }\left( x\otimes y\right), 
\end{equation*}%
where $\tilde{\Delta}_{\alpha }=\tilde{\Delta}\circ \left( \alpha \otimes
\alpha \right) $ is defined in Definition \ref{F}.

Consider the maps $F,G:H\otimes H\longrightarrow H$ defined by 
\begin{equation*}
F\left( x\otimes y\right) =S\left( y\right) \cdot S\left( x\right) ,\text{ }%
G\left( x\otimes y\right) =S\left( x\cdot y\right) 
\end{equation*}
for all $x,$ $y\in H.$ We show that $\mu $ is a left inverse (with respect
to convolution) for $F$, and a right inverse for $G.$ Indeed, for $x,$ $y\in
H$ we have 
\begin{align*}
&\mu \ast F\left( x\otimes y\right)  =\mu \circ \left( \mu \otimes F\right)
\circ \tilde{\Delta}_{\alpha }\left( x\otimes y\right) 
=\sum_{\left( x\right) ,\left( y\right) }\mu \left( \left( \alpha \otimes
\alpha \right) \left( x\otimes y\right) _{\left( 1\right) }\right) \cdot
F\left( \left( \alpha \otimes \alpha \right) \left( x\otimes y\right)
_{\left( 2\right) }\right)  \\
&=\sum_{\left( x\right) ,\left( y\right) }\mu \left( \alpha \left(
x_{\left( 1\right) }\right) \otimes \alpha \left( y_{\left( 1\right)
}\right) \right) \cdot F\left( \alpha \left( x_{\left( 2\right) }\right)
\otimes \alpha \left( y_{\left( 2\right) }\right) \right)  
\\ &
\overset{\eqref{106}}{=}\sum_{\left( x\right) ,\left( y\right) }\left( \left( \alpha \left(
x_{\left( 1\right) }\right) \cdot \alpha \left( y_{\left( 1\right) }\right)
\right) \cdot \alpha \circ \left( S\left( y_{\left( 2\right) }\right) \cdot
S\left( x_{\left( 2\right) }\right) \right) \right) 
\\&
\overset{\eqref{(1)}}{=}\sum_{\left( x\right) ,\left( y\right) }\alpha ^{2}\circ \left( x_{\left(
1\right) }\right) \cdot \left( \alpha \left( y_{\left( 1\right) }\right)
\cdot \left( S\left( y_{\left( 2\right) }\right) \cdot S\left( x_{\left(
2\right) }\right) \right) \right)  \\
&=\sum_{\left( x\right) ,\left( y\right) }\alpha ^{2}\circ \left( x_{\left(
1\right) }\right) \cdot \left( \left( y_{\left( 1\right) }\cdot S\left(
y_{\left( 2\right) }\right) \right) \cdot \alpha \left( S\left( x_{\left(
2\right) }\right) \right) \right)  
=\sum_{\left( x\right) }\alpha ^{2}\circ \left( x_{\left( 1\right)
}\right) \cdot \left( \left( \eta \circ \varepsilon \left( y\right) \right)
\cdot \alpha \left( S\left( x_{\left( 2\right) }\right) \right) \right)  \\
&=\sum_{\left( x\right) }\alpha ^{2}\circ \left( x_{\left( 1\right)
}\right) \cdot \left( \eta \left( 1_{\Bbbk }\right) \cdot \alpha \left(
S\left( x_{\left( 2\right) }\right) \right) \right) \varepsilon \left(
y\right)  
=\sum_{\left( x\right) }\left( \alpha ^{2}\circ \left( x_{\left( 1\right)
}\right) \cdot \alpha ^{2}\circ \left( S\left( x_{\left( 2\right) }\right)
\right) \right) \varepsilon \left( y\right)  \\
&=\alpha ^{2}\circ \eta \circ \varepsilon \left( x\right) \varepsilon
\left( y\right)  
=\eta _{H}\circ \varepsilon _{H\otimes H}\left( x\otimes y\right) .
\end{align*}

It  shows that $\mu \ast F=\eta _{H}\circ \varepsilon _{H\otimes H},$%
\begin{align*}
&G\ast \mu \left( x\otimes y\right)  =\mu \circ \left( G\otimes \alpha
\circ \mu \right) \circ \tilde{\Delta}_{\alpha }\left( x\otimes y\right)  
=\sum_{x\otimes y}G\left( \left( \alpha \otimes \alpha \right) \left(
x\otimes y\right) _{\left( 1\right) }\right) \cdot \mu \left( \left( \alpha
\otimes \alpha \right) \left( x\otimes y\right) _{\left( 2\right) }\right) 
\\
&=\sum_{\left( x\right) ,\left( y\right) }G\left( \alpha \left( x_{\left(
1\right) }\right) \otimes \alpha \left( y_{\left( 1\right) }\right) \right)
\cdot \mu \left( \alpha \left( x_{\left( 2\right) }\right) \otimes \alpha
\left( y_{\left( 2\right) }\right) \right)  \\
&=\sum_{\left( x\right) ,\left( y\right) }S\left( \alpha \left( x_{\left(
1\right) }\right) \cdot \alpha \left( y_{\left( 1\right) }\right) \right)
\cdot \left( \alpha \left( x_{\left( 2\right) }\right) \cdot \alpha \left(
y_{\left( 2\right) }\right) \right)
  =\alpha \circ \left( \sum_{\left( x\right) ,\left( y\right) }S\left(
x_{\left( 1\right) }\cdot y_{\left( 1\right) }\right) \cdot \left( x_{\left(
2\right) }\cdot y_{\left( 2\right) }\right) \right)  \\
&=\alpha \circ \left( \sum_{\left( x\right) ,\left( y\right) }S\left(
\left( x\cdot y\right) _{\left( 1\right) }\right) \cdot \left( \left( x\cdot
y\right) _{\left( 2\right) }\right) \right) 
=\alpha \circ \left( \sum_{xy}S\left( \left( x\cdot y\right) _{\left(
1\right) }\right) \cdot \left( \left( x\cdot y\right) _{\left( 2\right)
}\right) \right)  \\
&=\alpha \circ \eta _{H}\circ \varepsilon _{H\otimes H}\left( x\otimes
y\right) =\eta _{H}\circ \varepsilon _{H\otimes H}\left( x\otimes y\right) ,
\end{align*}

and $G\ast \mu =\eta _{H}\circ \varepsilon _{H\otimes H}.$ Hence $\mu $ is a
left inverse for $F$ and a right inverse for $G$ in a Hom-associative
algebra, and therefore $F=G.$ This means that i) holds.

(2) See ( \cite{SC 2009}, \cite{AMS 2007Hbial}).

(3)We use the same technique that we applied in part (i). We consider the
linear maps%
\begin{equation*}
Q:H\longrightarrow H\otimes H\text{ \ and }R:H\longrightarrow H\otimes H
\end{equation*}%
in the convolution Hom-associative algebra $Hom\left( H,H\otimes H\right) $
(see Corollary $\ref{B})$ where $\tilde{\mu}_{\alpha }=\alpha \circ \tilde{%
\mu}$ is defined in Exemple \ref{A}. 
\begin{equation*}
f\ast g\left( x\right) =\tilde{\mu}_{\alpha }\circ \left( f\otimes g\right)
\circ \Delta \left( x\right) 
\end{equation*}%
given by 
\begin{equation*}
Q\left( x\right) =S\left( x\right) _{1}\otimes S\left( x\right) _{2}\text{
and }R\left( x\right) =S\left( x_{2}\right) \otimes S\left( x_{1}\right) 
\end{equation*}%
for all $x\in H$. We will again prove that $Q=R$ by showing that $Q$ and $R$
are both the convolution inverse of $\Delta :H\longrightarrow H\otimes H.$


(4) See ( \cite{SC 2009}, \cite{AMS 2007Hbial}).
\end{proof}

\begin{proposition}
Let $H=\left( H,\mu ,\eta ,\Delta ,\varepsilon ,\alpha ,S\right) $ be a
Hom-Hopf algebra with antipode $S.$ Then the following assertions are
equivalent:

\begin{enumerate}
\item $\sum_{\left( x\right) }S\left( x_{\left( 2\right) }\right) x_{\left(
1\right) }=\varepsilon \left( x\right) 1_{H}$ for any $x\in H.$ Or $\mu
\circ \left( S\otimes id_{H}\right) \circ \tau \circ \Delta =\varepsilon
\circ \eta ;$

\item $\sum_{\left( x\right) }x_{\left( 2\right) }S\left( x_{\left( 1\right)
}\right) =\varepsilon \left( x\right) 1_{H}$ for any $x\in H.$ Or $\mu \circ
\left( id_{H}\otimes S\right) \circ \tau \circ \Delta =\varepsilon \circ
\eta ;$

\item $S^{2}=id_{H}$ (by $S^{2}$ we mean the composition of $S$ with itself).
\end{enumerate}
\end{proposition}

\begin{proof}
(1)$\Longrightarrow $(3) We know that $id_{H}$ is inverse of $S$ with
respect to convolution. We show that $S^{2}$ is a right convolution inverse
of $S,$ and by the uniqueness of the inverse it will follow that $%
S^{2}=id_{H}.$ We have 
\begin{eqnarray*}
\left( S\ast S^{2}\right) \left( x\right)  &=&\mu \circ \left( S\otimes
S^{2}\right) \circ \Delta \left( x\right) =\sum_{\left( x\right) }S\left(
x_{\left( 1\right) }\right) \cdot S^{2}\left( x_{\left( 2\right) }\right)  \\
&=&\sum_{\left( x\right) }S\left( S\left( x_{\left( 2\right) }\right) \cdot
\left( x_{\left( 1\right) }\right) \right) \text{ (}S\text{ is an
anti-morphism of Hom-associative algebras)} \\
&=&S\left( \varepsilon \left( x\right) 1_{H}\right) =\varepsilon \left(
x\right) \eta \left( 1_{\Bbbk }\right) =\eta \circ \varepsilon \left(
x\right) .
\end{eqnarray*}

This shows that indeed $S\ast S^{2}=\eta \circ \varepsilon .$

(3)$\Longrightarrow $(2) We know that $\sum_{\left( x\right) }x_{\left(
1\right) }\cdot S\left( x_{\left( 2\right) }\right) =\varepsilon \left(
x\right) 1_{H}$.  Applying the anti-morphism of Hom-associative algebra $S$, we
obtain $\sum_{\left( x\right) }S^{2}\left( x_{\left( 2\right) }\right) \cdot
S\left( x_{\left( 1\right) }\right) =\varepsilon \left( x\right) 1_{H}.$
Since $S^{2}=id_{H},$ this becomes $\sum_{\left( x\right) }x_{\left(
2\right) }\cdot S\left( x_{\left( 1\right) }\right) =\varepsilon \left(
x\right) 1_{H}.$

(2)$\Longrightarrow $(3) We proceed as in (1)$\Longrightarrow $(3)$,$ and 
show that $S^{2}=id_{H}$ is a left convolution inverse for $S.$ Indeed,%
\begin{eqnarray*}
\left( S^{2}\ast S\right) \left( x\right)  &=&\mu \circ \left( S^{2}\otimes
S\right) \circ \Delta \left( x\right) =\sum_{\left( x\right) }S^{2}\left(
x_{\left( 1\right) }\right) \cdot S\left( x_{\left( 2\right) }\right)  \\
&=&S\left( \sum_{\left( x\right) }x_{\left( 2\right) }\cdot S\left(
x_{\left( 1\right) }\right) \right) =S\left( \varepsilon \left( x\right)
1_{H}\right) =\eta \circ \varepsilon \left( x\right) .
\end{eqnarray*}

(3)$\Longrightarrow $(1) We apply $S$ to the equation $\sum_{\left( x\right)
}S\left( x_{\left( 1\right) }\right) \cdot x_{\left( 2\right) }=\varepsilon
\left( x\right) 1_{H},$ and using $S^{2}=id_{H}$ we obtain $\sum_{\left(
x\right) }S\left( x_{\left( 2\right) }\right) \cdot x_{\left( 1\right)
}=\varepsilon \left( x\right) 1_{H}$.
\end{proof}

\begin{corollary}
Let $H$ be a commutative or cocommutative Hom-Hopf algebra then $S^{2}=id_{H}.
$
\end{corollary}

\begin{proof}
If $H$ is commutative ( $\mu \left( x\otimes y\right) =y\otimes x)$, then by 
$\sum_{\left( x\right) }S\left( x_{\left( 1\right) }\right) \cdot x_{\left(
2\right) }=\varepsilon \left( x\right) 1_{H}$, it follows that $\sum_{\left(
x\right) }x_{\left( 2\right) }\cdot S\left( x_{\left( 1\right) }\right)
=\varepsilon \left( x\right) 1_{H},$ i.e. (2) from the preceding proposition.

If $H$ is cocommutative, then 
\begin{equation*}
\sum_{\left( x\right) }x_{\left( 1\right) }\otimes x_{\left( 2\right)
}=x_{\left( 2\right) }\otimes x_{\left( 1\right) },
\end{equation*}%
and then by $\sum_{\left( x\right) }S\left( x_{\left( 1\right) }\right)
\cdot x_{\left( 2\right) }=\varepsilon \left( x\right) 1_{H}$, it follows
that $\sum_{\left( x\right) }S\left( x_{\left( 2\right) }\right) \cdot
x_{\left( 1\right) }=\varepsilon \left( x\right) 1_{H},$ i.e. (1) from the
preceding proposition.
\end{proof}

\begin{remark}
Let $H=\left( H,\mu ,\eta ,\Delta ,\varepsilon ,\alpha \right) $ be a
Hom-Hopf algebra with antipode $S.$ Then the Hom-bialgebra $%
H^{op,cop}=\left( H,\mu ^{op},\eta ,\Delta ^{cop},\varepsilon ,\alpha
\right) $ is a Hom-Hopf algebra with the same antipode $S.$
\end{remark}

In Proposition \ref{J},  we claim  that if $H$ is a finite dimensional
Hom-bialgebra, then its dual is a Hom-bialgebra. The following result shows
that if $H$ is  a Hom-Hopf algebra, then its dual also has a Hom-Hopf
algebra structure.

\begin{proposition}
Let $H=$ be a finite dimensional Hom-Hopf algebra, with antipode $S.$ Then
the Hom-bialgebra $H^{\ast }$ is a Hom-Hopf algebra, with antipode $S^{\ast
}.$
\end{proposition}

\begin{proof}
We know already that $H^{\ast }$ is a Hom-bialgebra. We therefore need only
show that $S^{\ast }$ is the antipode of $H^{\ast }$. To this end, we have
that, 
\begin{eqnarray*}
&& \mu _{H^{\ast }}\circ \left( S^{\ast }\otimes id_{H^{\ast }}\right) \circ
\Delta _{H^{\ast }}\left( f^{\ast }\right) =\left( \rho \circ \left(
S^{\ast }\otimes id_{H^{\ast }}\right) \circ \Delta _{H}\right) \left(
f^{\ast }\circ \mu _{H}\right) 
= \left( S\otimes id_{H}\right) ^{\ast }\Delta _{H}\left( f^{\ast }\circ
\mu _{H}\right) \\
&&=\left( f^{\ast }\circ \mu _{H}\right) \left( S\otimes id_{H}\right) \circ
\Delta _{H} 
=\left( \rho \circ f^{\ast }\circ \eta _{H}\circ \varepsilon _{H}\right)
=\eta _{H^{\ast }}\circ \varepsilon _{H^{\ast }}\left( f^{\ast }\right) ,
\end{eqnarray*}%
using $\rho ^{-1}\left( S\otimes id_{H}\right) ^{\ast }\rho =\left( S^{\ast
}\otimes id_{H^{\ast }}\right) $.

Similarly,%
\begin{equation*}
\mu _{H^{\ast }}\circ \left( id_{H^{\ast }}\otimes S^{\ast }\right) \circ
\Delta _{H^{\ast }}\left( f^{\ast }\right) =\eta _{H^{\ast }}\circ
\varepsilon _{H^{\ast }}\left( f^{\ast }\right) .
\end{equation*}%
For all $f^{\ast }\in H^{\ast },$ since $\varepsilon _{H^{\ast }}\left(
f^{\ast }\right) =f^{\ast }\left( 1_{H}\right) =f^{\ast }\circ \eta
_{H}\left( 1_{\Bbbk }\right) ,$ and $\eta _{H^{\ast }}\left( 1_{\Bbbk
}\right) =1_{H^{\ast }}=\varepsilon $, we have that $S^{\ast }$ is the
convolution inverse of $id_{H^{\ast }}$ and therefore $S^{\ast }$ is the
antipode of $H^{\ast }$.
\end{proof}

\section{Modules and Comodules of Hom-Hopf algebras}

In this section, we recall the definitions of modules and comodules over
Hom-associative algebras and study their tensor products.

The definitions of action and coaction are simply a polarisation of those
of Hom-algebras and Hom-coalgebras, so we include them now among the basic
definitions.

A Hom-module is a pair $(M,\alpha )$  in which $M$ is a vector
space and $\alpha :M$ $\longrightarrow M$ is a linear map \cite{DY 2010}. A morphism $%
(M,\alpha _{M})$ $\longrightarrow $ $(N,\alpha _{N})$ of Hom-modules is a
linear map $f:M\longrightarrow N$ such that $\alpha _{N}$ $\circ f=f\circ
\alpha _{M}$ . We will often abbreviate a Hom-module $(M,\alpha )$ to $M$ .
The tensor product of  Hom-modules $(M,\alpha _{M})$ and $(N,\alpha _{N})$
consists of the vector space $M\otimes N$ and the linear self-map $\alpha
_{M}\otimes \alpha _{N}.$

\subsection{Modules over Hom-associative algebras}

Let $A=(A,\mu _{A},\eta_A,\alpha _{A})$ be a unital Hom-associative algebra and $%
(M,\alpha _{M})$ be a Hom-module.

The vector space $M$ is called a \emph{left }$A$\emph{-module} (or left
module over Hom-algebra $A$) if there exists a morphism $\lambda _{l}$ $%
:A\otimes M\longrightarrow M$ of Hom-modules, written as $\lambda _{l}\left(
a\otimes m\right) =a\rhd m,$ called the structure map, such that%
\begin{equation}
\lambda _{l}\circ \left( \alpha _{A}\otimes \lambda _{l}\right) =\lambda
_{l}\circ \left( \mu _{A}\otimes \alpha _{M}\right) \quad and\quad \lambda
_{l}\circ \left( \eta _{A}\otimes id_{M}\right) =\alpha _{M} , \label{(16)}
\end{equation}

or, equivalently, that for all $a,b$ $\in A$ and $m\in M$%
\begin{equation*}
\alpha _{A}\left( a\right) \rhd \left( b\rhd m\right) =\left( ab\right) \rhd
\alpha _{M}\left( m\right) \quad and\quad 1\rhd m=\alpha _{M}\left( m\right),
\end{equation*}

where the first identity of ($\ref{(16)}$) acts on $A\otimes A\otimes M$ and
the second of ($\ref{(16)}$) on $\Bbbk \otimes M\simeq M$. The map $\lambda
_{l}$ is then called  left action of $A$ on $M.$

If $\left( M,\alpha _{M}\right) $ and $\left( M^{\prime },\alpha _{M^{\prime
}}\right) $ are left $A$-modules, then a \emph{morphism} \emph{of left }%
$A$\emph{-modules} $f:M\longrightarrow M^{\prime }$ is a morphism of the
underlying Hom-modules such that%
\begin{equation}
f\circ \lambda _{l}=\lambda _{l}^{\prime }\circ \left( id_{A}\otimes
f\right) \qquad or\qquad f\left( a\rhd m\right) =a\rhd f\left( m\right) .
\label{(17)}
\end{equation}

If $f$ is invertible, it is a  left $A$-modules \emph{isomorphism}.

A \emph{right }$A$\emph{-module }(or right module over Hom-associative
algebra $A$) is a vector space $M$ with a morphism $\lambda _{r}$ $:M\otimes
A\longrightarrow M$ of Hom-modules, right action of $A$ on $M$ and written
as $\lambda _{r}\left( m\otimes a\right) =m\lhd a$, such that%
\begin{equation}
\lambda _{r}\circ \left( \lambda _{r}\otimes \alpha _{A}\right) =\lambda
_{r}\circ \left( \alpha _{M}\otimes \mu _{A}\right) \quad and\quad \lambda
_{r}\circ \left( id_{M}\otimes \eta _{A}\right) =\alpha _{M}  \label{(18)}
\end{equation}

The two conditions on a right $A$-module ($\ref{(18)}$) can be expressed as
for all $a,b$ $\in A$ and $m\in M$

\begin{equation*}
\left( m\lhd a\right) \lhd \alpha _{A}\left( b\right) =\alpha _{M}\left(
m\right) \lhd \left( ab\right) \quad and\quad m\lhd 1=\alpha _{M}\left(
m\right)
\end{equation*}

or as the commutativity of the diagrams%
\begin{equation*}
\begin{array}{ccc}
M\otimes A\otimes A & \overset{\ \alpha _{M}\otimes \mu _{A}}{%
\longrightarrow } & \ M\otimes A \\ 
{\small \ \lambda _{r}\otimes \alpha _{A}}\downarrow \qquad  &  & \downarrow
\lambda _{r} \\ 
\quad M\otimes A\mathit{\ } & \overset{\lambda _{r}}{\longrightarrow } & M%
\end{array}%
\qquad 
\begin{tabular}{lll}
$M\otimes \Bbbk $ & $\overset{\alpha _{M}}{\longrightarrow }$ & $M\otimes
\Bbbk $ \\ 
$id_{M}\otimes \eta \downarrow $ &  & $\cong $ \\ 
$M\otimes A$ & $\overset{\lambda _{r}}{\longrightarrow }$ & $M$%
\end{tabular}%
\end{equation*}

The map $\lambda _{r}$ is then called  right action of $A$ on $M$. We will
often denote a left or right $A$-module by $\left( M,\lambda ,\alpha
_{M}\right) $ and refer to $\lambda $ as the left or right action of $A$ on $%
M.$

\begin{example}
If $A$ is a unital Hom-associative algebra, then we may consider $A$ as
either a left or right $A$-module where the action is given by the
multiplication $\mu .$
\end{example}

If $M$ and $M^{\prime }$ are \ right $A$-modules, then a morphism of right $%
A$-module $f:M\longrightarrow M^{\prime }$ if%
\begin{equation}
f\circ \lambda _{r}=\lambda _{r}^{\prime }\circ \left( f\otimes
id_{A}\right) \qquad or\qquad f\left( m\lhd a\right) =f\left( m\right) \lhd a.
\end{equation}

If $f$ is invertible, it is an \emph{isomorphism of right }$A$\emph{%
-modules}.

\begin{proposition}
A right $A$-module is nothing else than a left module over the opposite
unital Hom-associative algebra $A^{op}.$ Therefore we need only consider left
modules which shall for simplicity be called Hom-module in the sequel.
\end{proposition}

\begin{proof}Indeed,
\begin{eqnarray*}
&& \lambda _{r}^{op}\circ \left( \alpha _{A}\otimes \lambda _{r}^{op}\right)
\left( x\otimes y\otimes m\right)  =\lambda _{r}^{op}\circ \left( \alpha
_{A}\left( x\right) \otimes \lambda _{r}^{op}\left( m\otimes y\right)
\right)  \\
&&= \lambda _{r}\circ \left( \lambda _{r}\left( m\otimes y\right) \otimes
\alpha _{A}\left( x\right) \right)  
=\lambda _{r}\circ \left( \alpha _{M}\left( m\right) \otimes \mu \left(
y\otimes x\right) \right)  \\
&&=\lambda _{r}^{op}\circ \left( \mu ^{op}\left( x\otimes y\right) \otimes
\alpha _{M}\left( m\right) \right)  
=\lambda _{r}^{op}\circ \left( \mu ^{op}\otimes \alpha _{M}\right) \left(
x\otimes y\otimes m\right) .
\end{eqnarray*}
\end{proof}

\begin{theorem}
Let $(A,\mu ,\eta ,\alpha _{A})$ be a unital Hom-associative algebra and $%
\left( M,\lambda _{l},\alpha _{M}\right) $ be a left $A$-module Then $\left(
M,\alpha _{M}\circ \lambda _{l},\alpha _{M}^{2}\right) $ is  another left $A
$-module over unital Hom-associative algebra $(A,\mu _{\alpha },\eta
_{\alpha },\alpha _{A}^{2})$ defined  in Proposition \ref{C}.
\end{theorem}

\begin{proof}
Straightforward.
\end{proof}

A vector space $M$ is called an $A$\emph{-bimodule} if $M$ is both a
left $A$-module with action $a\rhd m$ and a right $A$-module with action $%
m\lhd a$ satisfying the compatibility condition%
\begin{equation*}
(a\rhd m)\lhd \alpha _{A}\left( b\right) =\alpha _{A}\left( a\right) \rhd
(m\lhd b),
\end{equation*}%
for $a,b\in A$ and $m\in M$. Or%
\begin{equation}
\lambda _{r}\circ \left( \lambda _{l}\otimes \alpha _{A}\right) =\lambda
_{l}\circ \left( \alpha _{A}\otimes \lambda _{r}\right).   \label{(20)}
\end{equation}%
Then, we refer to the tuple  $\left( M,\lambda _{l},\lambda _{r},\alpha _{M}\right) $ for  an 
$A$-bimodule $M$.

If $M$ and $M^{\prime }$ are $A$-bimodules, a map $f$ $:M\longrightarrow
M^{\prime }$ is a \emph{morphism of }$A$\emph{-bimodules} if it is a
morphism of both left $A$-module and right $A$-module.

Every Hom-associative algebra $(A,\mu _{A},\alpha _{A})$ is an  $A$-bimodule
with  $\lambda _{l}=\lambda _{r}=\mu _{A}$.

\begin{example}
\label{2}Let $M^{\prime }=V\otimes M\otimes V$, and consider structure maps $%
\lambda _{l}^{\mu }=\mu \otimes \alpha _{M}\otimes \alpha _{V}$ and $\lambda
_{r}^{\mu }=\alpha _{V}\otimes \alpha _{M}\otimes \mu $. Then,  $\left(
V\otimes M\otimes V,\lambda _{l}^{\mu },\lambda _{r}^{\mu }\right) $ is an
exterior $A$-bimodule; $\lambda _{l}^{\mu }$ and $\lambda _{r}^{\mu }$ are 
called exterior bimodule structure maps. 

Indeed, 
the left $A$-module axioms for $\lambda _{l}^{\mu }$ now follows from that
of $\lambda _{l}$ and the identities%
\begin{eqnarray*}
\lambda _{l}^{\mu }\circ \left( \alpha _{V}\otimes \lambda _{l}^{\mu
}\right)  &=&\left( \mu \otimes \alpha _{M}\otimes \alpha _{V}\right) \circ
\left( \alpha _{V}\otimes \mu \otimes \alpha _{M}\otimes \alpha _{V}\right) 
\\
&=&\mu \circ \left( \alpha _{V}\otimes \mu \right) \otimes \alpha
_{M}^{\otimes 2}\otimes \alpha _{V}^{\otimes 2} \\
&\overset{\eqref{(1)}}{=}&\mu \circ \left( \mu \otimes \alpha _{V}\right) \otimes \alpha
_{M}^{\otimes 2}\otimes \alpha _{V}^{\otimes 2}.
\end{eqnarray*}%
\begin{eqnarray*}
\lambda _{l}^{\mu }\circ \left( \mu \otimes \alpha _{V\otimes M\otimes
V}\right)  &=&\left( \mu \otimes \alpha _{M}\otimes \alpha _{V}\right) \circ
\left( \mu \otimes \alpha _{V}\otimes \alpha _{M}\otimes \alpha _{V}\right) 
\\
&=&\mu \circ \left( \mu \otimes \alpha _{V}\right) \otimes \alpha
_{M}^{\otimes 2}\otimes \alpha _{V}^{\otimes 2},
\end{eqnarray*}

and%
\begin{eqnarray*}
\lambda _{l}^{\mu }\circ \left( \eta \otimes id_{V\otimes M\otimes V}\right)
&=&\left( \mu \otimes \alpha _{M}\otimes \alpha _{V}\right) \circ \left(
\eta \otimes id_{V\otimes M\otimes V}\right)  \\
&\overset{\eqref{(2)}}{=}&\left( \mu \circ \left( \eta \otimes id_{V}\right) \otimes \alpha
_{M}\otimes \alpha _{V}\right) =\alpha _{V\otimes M\otimes V}.
\end{eqnarray*}

Likewise, the right $A$-module axioms for $\lambda _{r}^{\mu }$ follow from
that of $\lambda _{r}$ and the identity$\quad \lambda _{r}\circ
\left( id_{M}\otimes \eta \right) =\alpha _{M}$.%
\begin{eqnarray*}
\lambda _{r}^{\mu }\circ \left( \lambda _{r}^{\mu }\otimes \alpha
_{V}\right)  &=&\left( \alpha _{V}\otimes \alpha _{M}\otimes \mu \right)
\circ \left( \alpha _{V}\otimes \alpha _{M}\otimes \mu \otimes \alpha
_{V}\right)  \\
&=&\alpha _{V}^{\otimes 2}\otimes \alpha _{M}^{\otimes 2}\otimes \mu \circ
\left( \mu \otimes \alpha _{V}\right)  \\
&=&\alpha _{V}^{\otimes 2}\otimes \alpha _{M}^{\otimes 2}\otimes \mu \circ
\left( \alpha _{V}\otimes \mu \right) \quad by\text{ (}\ref{(1)}\text{)} \\
\lambda _{r}^{\mu }\circ \left( \alpha _{_{V\otimes M\otimes V}}\otimes \mu
\right)  &=&\left( \alpha _{V}\otimes \alpha _{M}\otimes \mu \right) \circ
\left( \alpha _{V}\otimes \alpha _{M}\otimes \alpha _{V}\otimes \mu \right) 
\\
&=&\alpha _{V}^{\otimes 2}\otimes \alpha _{M}^{\otimes 2}\otimes \mu \circ
\left( \alpha _{V}\otimes \mu \right) ,
\end{eqnarray*}

and%
\begin{eqnarray*}
\lambda _{r}^{\mu }\circ \left( id_{V\otimes M\otimes V}\otimes \eta \right)
&=&\left( \alpha _{V}\otimes \alpha _{M}\otimes \mu \right) \circ \left(
id_{V\otimes M\otimes V}\otimes \eta \right)  \\
&=&\alpha _{V}\otimes \alpha _{M}\otimes \mu \circ \left( id_{V}\otimes \eta
\right)  \\
&\overset{\eqref{(2)}}{=}&\alpha _{V\otimes M\otimes V}.
\end{eqnarray*}

Finally,  compatibility conditions ($\ref{(20)})$ follows from the
following calculation%
\begin{eqnarray*}
\lambda _{r}^{\mu }\circ \left( \lambda _{l}^{\mu }\otimes \alpha
_{V}\right)  &=&\left( \alpha _{V}\otimes \alpha _{M}\otimes \mu \right)
\circ \left( \mu \otimes \alpha _{M}\otimes \alpha _{V}\otimes \alpha
_{V}\right)  \\
&=&\alpha _{V}\circ \mu \otimes \alpha _{M}^{\otimes 2}\otimes \mu \circ
\alpha _{V} \\
&=&\mu \circ \left( \alpha _{V}\otimes \alpha _{V}\right) \otimes \alpha
_{M}^{\otimes 2}\otimes \left( \alpha _{V}\otimes \alpha _{V}\right) \circ
\mu  \\
&=&\lambda _{l}^{\mu }\circ \left( \alpha _{V}\otimes \lambda _{r}^{\mu
}\right) .
\end{eqnarray*}%
We have shown that $\left( V\otimes M\otimes V,\lambda _{l}^{\mu },\lambda
_{r}^{\mu }\right) $ is an exterior $A$-bimodule.
\end{example}

\subsubsection{Comodules over Hom-coassociative coalgebras}

Dualizing actions of Hom-associative algebras on Hom-modules leads to
coactions of Hom-coassociative coalgebras on Hom-comodules. Let $(C,\Delta
_{C},\varepsilon _{C},\beta _{C})$ be a counital Hom-coassociative coalgebra
and $(M,\beta _{M})$ be a Hom-module.

A right coaction of $(C,\Delta _{C},\varepsilon _{C},\beta _{c})$ on a
vector space $M$, called then a \emph{right }$C$\emph{-comodule}, is a
morphism $\rho _{r}$ $:M\longrightarrow M\otimes C$ of Hom-modules,
satisfying the identities%
\begin{equation}
\left( \rho _{r}\otimes \beta _{C}\right) \circ \rho _{r}=\left( \beta
_{M}\otimes \Delta _{C}\right) \circ \rho _{r}\quad and\quad \left(
id_{M}\otimes \varepsilon _{C}\right) \circ \rho _{r}=\beta _{M}.
\label{(21)}
\end{equation}

If $\rho _{r}$ and $\rho _{r}^{\prime }$ are  right comodules of $C$ on
Hom-module $(M,\beta _{M})$ and $(M^{\prime },\beta _{M^{\prime }})$, then a 
\emph{morphism of right }$C$\emph{-comodules} $f:M\longrightarrow
M^{\prime }$ is a morphism of the underlying Hom-comodules such that%
\begin{equation}
\rho _{r}^{\prime }\circ f=\left( id_{C}\otimes f\right) \circ \rho _{r}.
\label{(22)}
\end{equation}

A \emph{left }$C$\emph{-comodule} ( or left comodule over
Hom-coassociative coalgebra $C$) on $(M,\beta _{M})$ is a linear mapping $%
\rho _{l}:M\longrightarrow C\otimes M$ such that%
\begin{equation}
\left( \beta _{C}\otimes \rho _{l}\right) \circ \rho _{l}=\left( \Delta
_{C}\otimes \beta _{M}\right) \circ \rho _{l}\quad and\quad \left(
\varepsilon _{C}\otimes id_{M}\right) \circ \rho _{l}=\beta _{M}.
\label{(23)}
\end{equation}

The two conditions of ($\ref{(21)}$) are equivalent to the requirement that
the two diagrams 
\begin{equation*}
\begin{array}{ccc}
M\otimes C\otimes C & \overset{\ \beta _{M}\otimes \Delta _{C}}{%
\longleftarrow } & \ M\otimes C \\ 
{\small \ \rho _{r}\otimes \beta _{C}}\uparrow \qquad  &  & \uparrow {\small %
\rho _{r}} \\ 
\quad M\otimes C\mathit{\ } & \overset{{\small \rho _{r}}}{\longleftarrow }
& M%
\end{array}%
\text{\qquad }%
\begin{tabular}{lll}
$M\otimes \Bbbk $ & $\cong $ & $M$ \\ 
$id_{M}\otimes \varepsilon \uparrow $ &  & $\uparrow \beta _{M}$ \\ 
$M\otimes C$ & $\overset{{\small \rho _{r}}}{\longleftarrow }$ & $M.$%
\end{tabular}%
\end{equation*}

are commutative. The diagrams of ($\ref{(23)}$) for a right comodule are
obtained from the diagrams of ($\ref{(18)}$) for a right module by reversing
arrows and replacing the multiplication $\mu $ by the comultiplication $%
\Delta $ and the unit $\eta $ by the counit $\varepsilon $. Note also that
the identities of ($\ref{(21)}$) and ($\ref{(23)}$) which characterize right
and left comodules are just the "duals" of the identities  of ($\ref{(16)}$)
and ($\ref{(18)}$) which define right and left modules, respectively.

The map $\rho _{r}$ (resp. $\rho _{l}$) is then called a right (resp. left)
coaction of $(M,\beta _{M})$ on $C$.

Then, a \emph{morphism of right }$C$\emph{-comodule} is $f:M\rightarrow
M^{\prime }$ such that%
\begin{equation}
\rho _{r}^{\prime }\circ f=\left( f\otimes id_{C}\right) \circ \rho _{r}.
\label{(24)}
\end{equation}

If $f$ is invertible, it is an isomorphism of right $C$-modules.

\begin{proposition}
A right $C$-comodule is the same as a left comodule over the coopposite
counital Hom-coassociative coalgebra $C^{cop}.$
\end{proposition}

\begin{remark}
The preceding proposition shows that any result that we obtain for right
Hom-comodules has an analogue for left Hom-comodules. This is why we are
going to work generally with right $C$-comodules, without explicitly
mentioning the similar results for left $C$-comodules.
\end{remark}

\begin{remark}
Clearly, $\Delta $ is a right bicomodule and also a left comodule of $A$
on itself. Indeed, in the case $\rho _{r}=\rho _{l}=\Delta $,  equations ($%
\ref{(21)}$) and  ($\ref{(23)}$) become the Hom-coalgebra
axioms ($\ref{(7)}$) and ($\ref{(8)}$).
\end{remark}

A vector space $M$ is called a $C$\emph{-bicomodule} if $M$ is both a
left $C$-comodule and a right $C$-comodule satisfying the compatibility
condition%
\begin{equation}
\left( \beta _{C}\otimes \rho _{r}\right) \circ \rho _{l}=\left( \rho
_{l}\otimes \beta _{C}\right) \circ \rho _{r}  \label{(25)}
\end{equation}

Then, we call $\left( M,\rho _{l},\rho _{r},\alpha _{M}\right) $  a $C$%
-bicomodule.

If $M$ and $M^{\prime }$ are $C$-bicomodules, a map $f$ $%
:M\longrightarrow M^{\prime }$ is a \emph{morphism of }$C$\emph{%
-bicomodules} if it is morphism of left $C$-comodule and it is morphism
of right $C$-comodule.

\begin{example}
Let $N^{\prime }=V\otimes N\otimes V,$ and consider structure maps $\rho
_{l}^{\Delta }=\Delta \otimes \alpha _{N}\otimes \alpha _{V}$ and $\rho
_{r}^{\Delta }=\alpha _{V}\otimes \alpha _{N}\otimes \Delta $. Then $\left(
V\otimes N\otimes V,\rho _{l}^{\Delta },\rho _{r}^{\Delta }\right) $ is an
exterior $A$-bicomodule; $\rho _{l}^{\Delta }$ and $\rho _{r}^{\Delta }$
are called exterior bicomodule structure maps.
\end{example}

\begin{proof}
This proof is completely analogous to that of Example \ref{2}.
\end{proof}\\

Let $C$ be a counital Hom-coassociative coalgebra, and $C^{\ast }$ the dual
unital Hom-associative algebra. If $M$ is a vector space, and $\omega
:M\longrightarrow M\otimes C$ is a linear map, we define $\psi _{\omega }$ $%
:C^{\ast }\otimes M\longrightarrow M$ by%
\begin{equation*}
\psi _{\omega }=\phi \left( \gamma \otimes id_{M}\right) \left( id_{C^{\ast
}}\otimes \tau _{M\otimes C}\right) \left( id_{C^{\ast }}\otimes \omega
\right) ,
\end{equation*}%
where $\gamma :C^{\ast }\otimes C\longrightarrow \Bbbk $ is defined by $%
\gamma \left( x^{\ast }\otimes x\right) =x^{\ast }\left( x\right) ,$ and $%
\phi :\Bbbk \otimes M\longrightarrow M$ is the canonical isomorphism. If $%
\omega \left( m\right) =\sum_{i}m_{i}\otimes x_{i},$ then $\psi _{\omega
}\left( x^{\ast }\otimes m\right) =\sum_{i}x^{\ast }\left( x_{i}\right)
m_{i}.$

\begin{proposition}
Let $\left( C,\Delta ,\varepsilon ,\beta \right) $ be a counital
Hom-coassociative coalgebra. $\left( M,\omega ,\beta _{M}\right) $ is a
right $C$-comodule if and only if $\left( M,\psi _{\omega },\alpha
_{M}\right) $ is a left $C^{\ast }$-module where $\alpha _{M}=\beta _{M}$.
\end{proposition}

\begin{proof}
From the previous Proposition we know that $\left( C^{\ast },\mu _{C^{\ast
}},\eta _{C^{\ast }},\alpha _{C^{\ast }}\right) $ is a unital
Hom-associative algebra when we define $\mu _{C^{\ast }},\eta _{C^{\ast
}},\alpha _{C^{\ast }}$, as in  Theorem \ref{H}.

Assume that $\left( M,\omega ,\beta _{M}\right) $ is a right $C$-comodule.
Denoting  $\omega \left( m\right) =\underset{\left( m\right) }{\sum }%
m_{\left( 0\right) }\otimes x_{\left( 1\right) },$ we have 
\begin{equation*}
\psi _{\omega }\left( x^{\ast }\otimes m\right) =\sum x^{\ast }\left(
x_{\left( 1\right) }\right) m_{\left( 0\right) }.
\end{equation*}

First, we have that 
\begin{eqnarray*}
\psi _{\omega }\left( 1_{C^{\ast }}\otimes m\right)  &=&\psi _{\omega
}\left( \varepsilon _{C}\otimes m\right) =\phi \left( \gamma \otimes
id_{M}\right) \left( id_{C^{\ast }}\otimes \tau _{M\otimes C}\right) \left(
id_{C^{\ast }}\otimes \omega \right) \left( \varepsilon _{C}\otimes m\right) 
\text{ } \\
&=&\sum \varepsilon _{C}\left( x_{\left( 1\right) }\right) m_{\left(
0\right) }=\left( id_{M}\otimes \varepsilon _{C}\right) \circ \omega \left(
m\right) =\beta _{M}\left( m\right) =\alpha _{M}\left( m\right) \text{ by (}%
\ref{(21)}\text{)}
\end{eqnarray*}

from the definition of a right $C$-comodule. Then, for $x^{\ast },$ $y^{\ast
}\in C^{\ast },$ $m\in M$%
\begin{align*}
&\psi _{\omega }\circ \left( \alpha _{C^{\ast }}\otimes \psi _{\omega
}\right) \left( x^{\ast }\otimes y^{\ast }\otimes m\right)  =\psi _{\omega
}\circ \left( \alpha _{C^{\ast }}\left( x^{\ast }\right) \otimes \psi
_{\omega }\left( y^{\ast }\otimes m\right) \right)  \\
&=\psi _{\omega }\circ \left( x^{\ast }\left( \beta _{C}\right) \otimes
\sum y^{\ast }\left( x_{\left( 1\right) }\right) m_{\left( 0\right) }\right) 
=\sum y^{\ast }\left( x_{\left( 1\right) }\right) \psi _{\omega }\left(
x^{\ast }\left( \beta _{C}\right) \otimes m_{\left( 0\right) }\right)  \\
&=\sum y^{\ast }\left( x_{\left( 1\right) }\right) x^{\ast }\left( \beta
_{C}\right) \left( x_{\left( 2\right) }\right) \left( m_{\left( 0\right)
}\right) _{\left( 0\right) } 
=\sum x^{\ast }\circ \beta _{C}\left( x_{\left( 2\right) }\right) y^{\ast
}\left( x_{\left( 1\right) }\right) \left( m_{\left( 0\right) }\right)  \\
&\overset{\eqref{(21)}}{=}\sum \rho \left( x^{\ast }\otimes y^{\ast }\right) \left( \beta
_{C}\left( x_{\left( 2\right) }\right) \otimes x_{\left( 1\right) }\otimes
m_{\left( 0\right) }\right) 
=\sum \rho \left( x^{\ast }\otimes y^{\ast }\right) \left( x_{\left(
2\right) }\otimes x_{\left( 1\right) }\otimes \beta _{M}\left( m_{\left(
0\right) }\right) \right)  \\
&=\sum \Delta ^{\ast }\rho \left( x^{\ast }\otimes y^{\ast }\right)
x_{\left( 1\right) }\otimes \beta _{M}\left( m_{\left( 0\right) }\right)  
=\psi _{\omega }\circ \left( \mu _{C^{\ast }}\otimes \alpha _{M}\right)
\left( x^{\ast }\otimes y^{\ast }\otimes m\right) ,
\end{align*}%
which shows that $\left( M^{\ast },\psi _{\omega },\alpha _{M}\right) $ is a
left $C^{\ast }$-module. Assume now that $\left( M^{\ast },\psi _{\omega
},\alpha _{M}\right) $ is a left $C^{\ast }$-module. From $\psi _{\omega
}\circ \left( \eta _{C^{\ast }}\otimes id_{M^{\ast }}\right) =\alpha
_{M^{\ast }}$, one obtains  
$
\sum \varepsilon _{C}\left( x_{\left( 1\right) }\right) m_{\left( 0\right)
}=\alpha _{M}\left( m\right) .
$ 
It follows %
\begin{equation*}
\left( id_{M}\otimes \varepsilon _{C}\right) \circ \omega \left( m\right)
=\sum \varepsilon _{C}\left( x_{\left( 1\right) }\right) m_{\left( 0\right)
}=\alpha _{M}\left( m\right) =\beta _{M}\left( m\right) .
\end{equation*}%
Hence, the second condition from the definition of a right $C$-comodule is
checked.

If $x^{\ast },$ $y^{\ast }\in C^{\ast },$ $m\in M$, then%
\begin{eqnarray*}
&&\psi _{\omega }\circ \left( \mu _{C^{\ast }}\otimes \alpha _{M}\right)
\left( x^{\ast }\otimes y^{\ast }\otimes m\right)  =\psi _{\omega }\circ
\left( \Delta ^{\ast }\circ \rho \left( x^{\ast }\otimes y^{\ast }\right)
\otimes \alpha _{M}\left( m\right) \right)  \\
&&=\sum \Delta _{C}^{\ast }\circ \rho \left( x^{\ast }\otimes y^{\ast
}\right) \left( x_{\left( 1\right) }\right) \alpha _{M}\left( m\right)
_{\left( 0\right) } 
=\sum \left( x^{\ast }\left( x_{\left( 1\right) }\right) _{\left( 1\right)
}y^{\ast }\left( x_{\left( 1\right) }\right) _{\left( 2\right) }\right)
\alpha _{M}\left( m\right) _{\left( 0\right) } \\
&&=\phi ^{\prime }\left( id_{M}\otimes y^{\ast }\otimes x^{\ast }\right)
\left( \beta _{M}\otimes \Delta _{C}\right) \omega \left( m\right) ,
\end{eqnarray*}%
where $\phi ^{\prime }:M\otimes \Bbbk \otimes \Bbbk \longrightarrow M$ is
the canonical isomorphism, and $\alpha _{M}=\beta _{M}$. Also 
\begin{eqnarray*}
&& \psi _{\omega }\circ \left( \alpha _{C^{\ast }}\otimes \psi _{\omega
}\right) \left( x^{\ast }\otimes y^{\ast }\otimes m\right)  =\psi _{\omega
}\circ \left( \alpha _{C^{\ast }}\left( x^{\ast }\right) \otimes \psi
_{\omega }\left( y^{\ast }\otimes m\right) \right)  \\
&&=
\psi _{\omega }\circ \left( \alpha _{C^{\ast }}\left( x^{\ast }\right)
\otimes \sum y^{\ast }\left( x_{\left( 1\right) }\right) m_{\left( 0\right)
}\right)  
=\sum y^{\ast }\left( x_{\left( 1\right) }\right) \psi _{\omega }\circ
\left( x^{\ast }\left( \beta _{C}\right) \otimes m_{\left( 0\right) }\right) 
\\
&&=\sum y^{\ast }\left( x_{\left( 1\right) }\right) x^{\ast }\left( \beta
_{C}\left( x_{\left( 2\right) }\right) \right) m_{\left( 0\right) } 
=\phi ^{\prime }\left( id_{M}\otimes y^{\ast }\otimes x^{\ast }\right)
\left( \omega \otimes \beta _{C}\right) \omega \left( m\right) .
\end{eqnarray*}%
Denoting 
\begin{equation*}
z=\left( \beta _{M}\otimes \Delta _{C}\right) \omega \left( m\right) -\left(
\omega \otimes \beta _{C}\right) \omega \left( m\right) \in M\otimes
C\otimes C, 
\end{equation*}%
we have $\left( id_{M}\otimes y^{\ast }\otimes x^{\ast }\right) \left(
z\right) =0$ for any $y^{\ast },x^{\ast }\in C^{\ast }$. This shows that $%
z=0.$

Indeed, if we denote by $\left( e_{i}\right) _{i}$ a basis of $C,$ we can
write $y=\sum_{i,j}m_{ij}\otimes e_{i}\otimes e_{j}$ for some $m_{ij}\in M.$
Fix $i_{0}$ and $j_{0}$ and consider the maps $e_{i}^{\ast }\in C^{\ast }$
defined by $e_{i}^{\ast }\left( e_{j}\right) =\delta _{i,j}$ for any $j.$
Then $m_{i_{0}j_{0}}=\left( id_{M}\otimes e_{i_{0}}^{\ast }\otimes
e_{j_{0}}^{\ast }\right) \left( z\right) =0,$ and from this we get $z=0.$

Then $\left( \beta _{M}\otimes \Delta _{C}\right) \omega \left( m\right)
=\left( \omega \otimes \beta _{C}\right) \omega \left( m\right) .$ The proof
is similar to the opposite case.
\end{proof}

\begin{proposition}
Let $\left( A,\mu ,\eta ,\alpha \right) $ be a finite dimensional unital
Hom-associative algebra. If $M$ is a left $A$-module, then $M$ is a right $%
A^{\ast }$-comodule.
\end{proposition}

\begin{theorem}
Let $\left( C,\Delta ,\varepsilon ,\beta \right) $ be a counital
Hom-coassociative coalgebra. Then for any right $C$-comodule $M$, $M^{\ast }$
is a left $C^{\ast }$module. Conversely, Let $\left( A,\mu ,\eta ,\alpha
\right) $ be a finite dimensional unital Hom-associative algebra. If
\ $N$ is a left $A$-module then  $N^{\ast }$ is a right $A^{\ast }$-comodule.
\end{theorem}

\subsubsection{ Tensor product of bimodules and bicomodules}

We will need to consider tensor products over $\Bbbk $ of bimodules. The
tensor product of bimodules can be endowed with bimodule structures;
if $M$ and $N$ are bimodules over $A$; we shall consider two bimodule structures on $M\otimes N$ (dual of each other), which we will
denote by $M\overline{\otimes }N$ and $M\underline{\otimes }N$ (see 
for example \cite{R.U 1996} for details).
These notations will also be used for the tensor product of bimodules or
bicomodules, `forgetting' some of the structures.

\begin{proposition}
Let $(H,%
\mu
_{H},\Delta _{H},\alpha _{H})$ be a Hom-bialgebra and $\left( M,\lambda
_{l},\lambda _{r}\right) $, $\left( M^{\prime },\lambda _{l}^{\prime
},\lambda _{r}^{\prime }\right) $ be $H$-bimodules. The internal
(bimodule) tensor product of $M$ with $M^{\prime }$ is the so called
interior $H$-bimodule $M\overline{\otimes }M^{\prime }=\left( M\otimes
M^{\prime },\lambda _{l}\overline{\otimes }\lambda _{l}^{\prime },\lambda
_{r}\overline{\otimes }\lambda _{r}^{\prime }\right) $ with%
\begin{equation}
\lambda _{l}^{2}=\lambda _{l}\overline{\otimes }\lambda _{l}^{\prime
}=\left( \lambda _{l}\otimes \lambda _{l}^{\prime }\right) \circ \left(
id_{H}\otimes \tau _{H,M}\otimes id_{M^{\prime }}\right) \circ \left( \Delta
_{H}\otimes id_{M}\otimes id_{M^{\prime }}\right)   \label{182}
\end{equation}%
and%
\begin{equation}
\lambda _{r}^{2}=\lambda _{r}\overline{\otimes }\lambda _{r}^{\prime
}=\left( \lambda _{r}\otimes \lambda _{r}^{\prime }\right) \circ \left(
id_{M}\otimes \tau _{M^{\prime },H}\otimes id_{H}\right) \circ \left(
id_{M}\otimes id_{M^{\prime }}\otimes \Delta _{H}\right) .
\end{equation}
\end{proposition}

\begin{proof}
It is already shown in \cite{DY 2010}.
\end{proof}

\begin{remark}
Since $\Delta _{H}$ is Hom-coassociative, $$\left( \lambda _{l}\overline{%
\otimes }\lambda _{l}^{\prime }\right) \overline{\otimes }\alpha _{H}\left(
\lambda _{l}^{\prime \prime }\right) =\alpha _{H}\left( \lambda _{l}\right) 
\overline{\otimes }\left( \lambda _{l}^{\prime }\overline{\otimes }\lambda
_{l}^{\prime \prime }\right)  \text{ and }  \left( \lambda _{r}\overline{\otimes }%
\lambda _{r}^{\prime }\right) \overline{\otimes }\alpha _{H}\left( \lambda
_{r}^{\prime \prime }\right) =\alpha _{H}\left( \lambda _{r}\right) 
\overline{\otimes }\left( \lambda _{r}^{\prime }\overline{\otimes }\left(
\lambda _{r}^{\prime \prime }\right) \right) .$$Thus, the internal tensor
product can be Hom-associatively applied to any finite family of $H$-bimodules. 
\end{remark}

\begin{proof}
We use the Hom-coassociativity axiom (\ref{(7)}). We have%
\begin{eqnarray*}
&& \left( \lambda _{l}\overline{\otimes }\lambda _{l}^{\prime }\right) 
\overline{\otimes }\alpha _{H}\left( \lambda _{l}^{\prime \prime }\right) 
=\left( \left( \lambda _{l}\overline{\otimes }\lambda _{l}^{\prime
}\right) \otimes \alpha _{H}\left( \lambda _{l}^{\prime \prime }\right)
\right) \circ \tau _{2,3}\circ \left( \Delta _{H}\otimes id_{M}^{\otimes
3}\right)  \\
&&= \left[ \left( \left( \lambda _{l}\otimes \lambda _{l}^{\prime }\right)
\circ \tau _{2,3}\circ \left( \Delta _{H}\otimes id_{M\otimes M^{\prime
}}\right) \right) \otimes \alpha _{H}\left( \lambda _{l}^{\prime \prime
}\right) \right] \circ \tau _{2,3}\circ \left( \Delta _{H}\otimes
id_{M\otimes M^{\prime }\otimes M^{\prime \prime }}\right)  \\
&&= \left[ \left( \lambda _{l}\otimes \lambda _{l}^{\prime }\right) \otimes
\lambda _{l}^{\prime \prime }\right] \circ \left[ \tau _{2,3}\circ \left(
\Delta _{H}\otimes id_{M\otimes M^{\prime }}\right) \otimes \alpha
_{H\otimes M^{\prime \prime }}\right] \circ \tau _{2,3}\circ \left( \Delta
_{H}\otimes id_{M\otimes M^{\prime }\otimes M^{\prime \prime }}\right)  \\
&&= \left[ \lambda _{l}\otimes \left( \lambda _{l}^{\prime }\otimes \lambda
_{l}^{\prime \prime }\right) \right] \circ \left[ \alpha _{H\otimes
M}\otimes \tau _{2,3}\circ \left( \Delta _{H}\otimes id_{M^{\prime }\otimes
M^{\prime \prime }}\right) \right] \circ \tau _{2,3}\circ \left( \Delta
_{H}\otimes id_{M\otimes M^{\prime }\otimes M^{\prime \prime }}\right)  \\
&&= \left( \alpha _{H}\left( \lambda _{l}\right) \otimes \left( \lambda
_{l}^{\prime }\overline{\otimes }\lambda _{l}^{\prime \prime }\right)
\right) \circ \tau _{2,3}\circ \left( \Delta _{H}\otimes id_{M\otimes
M^{\prime }\otimes M^{\prime \prime }}\right)  \\
&&= \alpha _{H}\left( \lambda _{l}\right) \overline{\otimes }\left( \lambda
_{l}^{\prime }\overline{\otimes }\lambda _{l}^{\prime \prime }\right) .
\end{eqnarray*}

The second assertion is proved similarly.
\end{proof}

\begin{corollary}
Let $(M,\lambda _{l},\lambda _{r})$be an $H$-bimodule. The structure
maps $\overline{\lambda }_{l}^{\mu }$ $=\mu \overline{\otimes }\lambda _{l}%
\overline{\otimes }\mu $ and $\overline{\lambda }_{r}^{\mu }$ $=\mu 
\overline{\otimes }\lambda _{r}\overline{\otimes }\mu $ on the interior $H$%
-bimodule $H\overline{\otimes }M\overline{\otimes }H$ are called the
two-sided interior extensions of $\lambda _{l}$ and $\lambda _{r}$,
respectively,  by $\mu $.
\end{corollary}

\begin{proposition}
Let $(M,\lambda _{l},\lambda _{r})$be an $H$-bimodule. The interior $H$%
-bimodule $M^{\overline{\otimes }n}=(M^{\otimes n},\lambda
_{l}^{n},\lambda _{r}^{n})$, with%
\begin{equation}
\lambda _{l}^{n}=\left( \lambda _{l}\overline{\otimes }\lambda
_{l}^{n-1}\right) =\left( \lambda _{l}\otimes \lambda _{l}^{n-1}\right)
\circ \tau _{2,3}\circ \left( \Delta \otimes id_{M}^{\otimes n}\right) 
\end{equation}%
and 
\begin{equation}
\lambda _{r}^{n}=\left( \lambda _{r}^{n-1}\overline{\otimes }\lambda
_{r}\right) =\left( \lambda _{r}^{n-1}\otimes \lambda _{r}\right) \circ \tau
_{n,n+1}\circ \left( id_{M}^{\otimes n}\otimes \Delta \right) 
\end{equation}%
is called the $n$-fold interior (bimodule) tensor power of $M.$
\end{proposition}

\begin{proof}
By induction.
\end{proof}

\begin{example}
Let $(H,%
\mu
_{H},\Delta _{H},\alpha _{H})$ be a Hom-bialgebra; for each $n$ $\geq 1$,
the $n$-fold interior (bimodule) tensor power of $H$ is the interior $H$%
-bimodule $H^{\overline{\otimes }n}=(H^{\otimes n},\lambda
_{l}^{n},\lambda _{r}^{n})$ with%
\begin{equation}
\lambda _{l}^{n}=\mu _{H}^{\otimes n}\circ \left( 135...\left( 2n-1\right)
246...\left( 2n\right) \right) \circ \prod\limits_{i=n}^{2n-2}\left( \Delta
_{H}\otimes id_{M}^{\otimes \left( 3n-i-2\right) }\right)
\end{equation}%
and%
\begin{equation}
\lambda _{r}^{n}=\mu _{H}^{\otimes n}\circ \left( 135...\left( 2n-1\right)
246...\left( 2n\right) \right) \circ \prod\limits_{i=n}^{2n-2}\left(
id_{M}^{\otimes \left( 3n-i-2\right) }\otimes \Delta _{H}\right),
\end{equation}%
where $\lambda _{l}=\lambda _{r}=%
\mu
_{H}$
\end{example}

\begin{proposition}
Let $(H,%
\mu
_{H},\Delta _{H},\alpha _{H})$ be a Hom-bialgebra and $(N,\rho _{l},\rho
_{r})$, $(N^{\prime },\rho _{l}^{\prime },\rho _{r}^{\prime })$ be $H$%
-bicomodules. The internal (bicomodule) tensor product of \ $N$ with 
$N^{\prime }$ is the so-called interior $H$-bicomodule $N\underline{%
\otimes }N^{\prime }=\left( N\otimes N^{\prime },\rho _{l}\underline{\otimes 
}\rho _{l}^{\prime },\rho _{r}\underline{\otimes }\rho _{r}^{\prime }\right) 
$ with%
\begin{equation}
\rho _{l}^{2}=\rho _{l}\underline{\otimes }\rho _{l}^{\prime }=\left( \mu
_{H}\otimes id_{N\otimes N^{\prime }}\right) \circ \tau _{2,3}\circ \left(
\rho _{l}\otimes \rho _{l}^{\prime }\right)   \label{111}
\end{equation}%
and%
\begin{equation}
\rho _{r}^{2}=\rho _{r}\underline{\otimes }\rho _{r}^{\prime }=\left(
id_{N\otimes N^{\prime }}\otimes \mu _{H}\right) \circ \tau _{2,3}\circ
\left( \rho _{r}\otimes \rho _{r}^{\prime }\right) .  \label{110}
\end{equation}
\end{proposition}

\begin{remark}
Since $\mu _{H}$ is Hom-associative, we have $$\left( \rho _{l}\underline{\otimes }%
\rho _{l}^{\prime }\right) \underline{\otimes }\alpha \left( \rho
_{l}^{\prime \prime }\right) =\alpha \left( \rho _{l}\right) \underline{%
\otimes }\left( \rho _{l}^{\prime }\underline{\otimes }\rho _{l}^{\prime
\prime }\right) \text{ and } \left( \rho _{r}\underline{\otimes }\rho _{r}^{\prime
}\right) \underline{\otimes }\alpha \left( \rho _{r}^{\prime \prime }\right)
=\alpha \left( \rho _{r}\right) \underline{\otimes }\left( \rho _{r}^{\prime
}\underline{\otimes }\rho _{r}^{\prime \prime }\right) .$$ Thus, the internal
tensor product can be Hom-associatively applied to any finite family of $H$%
-bicomodules.
\end{remark}

\begin{corollary}
Let $(N,\rho _{l},\rho _{r})$ be an $H$-bicomodule. The structure maps $%
\overline{\rho }_{l}^{\Delta }$ $=\Delta \underline{\otimes }\rho _{l}%
\underline{\otimes }\Delta $ and $\overline{\rho }_{r}^{\Delta }$ $=\Delta 
\underline{\otimes }\rho _{r}\underline{\otimes }\Delta $ on the interior $H$%
-bicomodule $H\underline{\otimes }N\underline{\otimes }H$ are called the
two-sided interior extensions of \ $\rho _{l}$ and $\rho _{r}$ by $\Delta $,
respectively.
\end{corollary}

\begin{proposition}
Let $(N,\rho _{l},\rho _{r})$be an $H$-bicomodule.\\ The interior $H$%
-bicomodule $N^{\underline{\otimes }n}=(N^{\otimes n},\rho _{l}^{n},\rho
_{r}^{n})$, with%
\begin{equation}
\rho _{l}^{n}=\rho _{l}\underline{\otimes }\rho _{l}^{n-1}=\left( \mu
\otimes \alpha _{N}^{\otimes n}\right) \circ \tau _{2,3}\circ \left( \rho
_{l}\otimes \rho _{l}^{n-1}\right) 
\end{equation}%
and 
\begin{equation}
\rho _{r}^{n}=\rho _{r}^{n-1}\underline{\otimes }\rho _{r}=\left(
id_{N}^{\otimes n}\otimes \mu \right) \circ \tau _{n,n+1}\circ \left( \rho
_{r}^{n-1}\otimes \rho _{r}\right) 
\end{equation}%
is called the $n$-fold internal (bicomodule) tensor power of $N$.
\end{proposition}

\begin{proof}
By induction
\end{proof}

\begin{example}
Let $(H,%
\mu
_{H},\Delta _{H},\alpha _{H})$ be a Hom-bialgebra, consider the $H$%
-bicomodule $(H,\rho _{l},\rho _{r})$ where $\rho _{l}=\rho _{r}=\Delta
_{H}$. For each $n\geq 1$, the $n$-fold internal (bicomodule) tensor
power of $H$ is the interior $H$-bicomodule $H^{\underline{\otimes }%
n}=(H^{\otimes n},\rho _{l}^{n},\rho _{r}^{n})$ with%
\begin{equation}
\rho _{l}^{n}=\prod\limits_{i=n}^{2n-2}\left( \mu _{H}\otimes
id_{N}^{\otimes i}\right) \circ \left( 135...\left( 2n-1\right) 246...\left(
2n\right) \right) ^{-1}\circ \Delta _{H}^{\otimes n}
\end{equation}

\begin{equation}
\rho _{r}^{n}=\prod\limits_{i=n}^{2n-2}\left( id_{N}^{\otimes i}\otimes \mu
_{H}\right) \circ \left( 135...\left( 2n-1\right) 246...\left( 2n\right)
\right) ^{-1}\circ \Delta _{H}^{\otimes n}.
\end{equation}
\end{example}

\begin{lemma}
\label{7}Consider the interior $H$-bicomodules, $H^{\overline{\otimes }%
n}=(H^{\otimes n},\lambda _{l}^{n},\lambda _{r}^{n})$ and $H^{\underline{%
\otimes }n}=(H^{\otimes n},\rho _{l}^{n},\rho _{r}^{n})$ with $\lambda
_{l}^{1}=\lambda _{l}^{1}=\mu ,$ $\rho _{l}^{1}=\rho _{l}^{1}=\Delta $ and
let $f:V^{\otimes q}\rightarrow V^{\otimes p}$ be  a linear map which commutes
with $\beta $ and such that  $\alpha ^{\otimes p}\circ f=f\circ \alpha ^{\otimes
q},$ $p,q$ $\in 
\mathbb{N}
^{\ast }$. Then

\begin{enumerate}
\item $\lambda _{l}^{p+1}\circ \left( \alpha ^{p-1}\otimes \left( \alpha
^{p-1}\otimes f\right) \circ \rho _{l}^{p}\right) =\left( \alpha
^{p-1}\otimes \lambda _{l}^{p}\circ \left( \alpha ^{p-1}\otimes f\right)
\right) \circ \rho _{l}^{p+1}$

\item $\lambda _{r}^{p+1}\circ \left( \left( f\otimes \alpha ^{p-1}\right)
\circ \rho _{r}^{p}\otimes \alpha ^{p-1}\right) =\left( \lambda
_{r}^{p}\circ \left( f\otimes \alpha ^{p-1}\right) \otimes \alpha
^{p-1}\right) \circ \rho _{r}^{p+1}.$
\end{enumerate}
\end{lemma}


\section{Gerstenhaber-Schack Cohomology for Hom-bialgebras}

Gerstenhaber-Schack cohomology of Hom-bialgebras  is a twisted 
generalization of bialgebras cohomology, which was first discovered by
Gerstenhaber and  Schack  \cite{G.S90,G.S92}, extending  associative algebras cohomology introduced by Hochschild in  \cite{GH 1945} to bialgebra. Deformation theories are intimately related to cohomology. It turns out that, we do not need a    cohomology of (Hom-)Hopf algebra since it is enough to deform 
(Hom-)Hopf algebra as a (Hom-)bialgebra.
We refer  to \cite{G.S90,G.S92} for  the definition of bialgebra cohomology and its truncated version due to Gerstenhaber and
Schack. 
We define the bicomplex extending Hochschild  cohomology for horizontal faces and coalgebra Cartier cohomology for the vertical faces.\\ 
Let $B=\left( B,\mu ,\eta ,\Delta
,\varepsilon ,\alpha \right) $ be a Hom-bialgebra on the $\Bbbk $-vector
space $B$.  
We set for the  cochains: 
\begin{equation*}
\mathcal{C}_{Hom}^{p,q}=Hom_{\Bbbk }\left( B^{\otimes q},B^{\otimes
p}\right) ,p,q\geq 1,
\end{equation*}%
\begin{equation*}
\mathcal{C}_{Hom}^{p,q}=\left\{ f:B^{\otimes q}\longrightarrow B^{\otimes
p},f\text{ }\ is\ a\ linear\ map,\ f\circ \alpha ^{\otimes q}=\alpha
^{\otimes p}\circ f\right\} .
\end{equation*}%
We define the horizontal faces $
\delta _{Hom,H}^{p,q}:\mathcal{C}_{Hom}^{p,q}\longrightarrow \mathcal{C}%
_{Hom}^{p,q+1}
$ as 
\begin{equation}
\delta _{Hom,H}^{p,q}\left( f\right) =\lambda _{l}^{p}\circ (\alpha
^{q-1}\otimes f)+\sum\limits_{i=1}^{q}(-1)^{i}f\circ \left( \alpha
^{\otimes (i-1)}\otimes \mu \otimes \alpha ^{\otimes (q-i)}\right)
+(-1)^{q+1}\lambda _{r}^{p}\circ (f\otimes \alpha ^{q-1}).  \label{(38)}
\end{equation}

The vertical faces  
 $
\delta _{Hom,C}^{p,q}:\mathcal{C}_{Hom}^{p,q}\longrightarrow \mathcal{C}%
_{Hom}^{p+1,q}
$ are defined as:
\begin{equation}
\delta _{Hom,C}^{p,q}\left( f\right) =(\alpha ^{p-1}\otimes f)\circ \rho
_{l}^{q}+\sum\limits_{j=1}^{p}(-1)^{j}\left( \alpha ^{\otimes (j-1)}\otimes
\Delta \otimes \alpha ^{\otimes (p-j)}\right) \circ f+(-1)^{p+1}(f\otimes
\alpha ^{p-1})\circ \rho _{r}^{q}.  \label{(39)}
\end{equation}

\begin{proposition}
\label{S}The composite 
\begin{equation*}
\delta _{Hom,C}^{2,1}\circ \delta _{Hom,C}^{1,1}=0,\quad \delta
_{Hom,C}^{1,2}\circ \delta _{Hom,H}^{1,1}=\delta _{Hom,H}^{2,1}\circ \delta
_{Hom,C}^{1,1},\quad \delta _{Hom,H}^{1,2}\circ \delta _{Hom,H}^{1,1}=0.
\end{equation*}
\end{proposition}

\begin{proof}
We prove the first identity. We have $\delta _{Hom,C}^{2,1}\left( f\right) =(\alpha
\otimes f)\circ \rho _{l}^{1}-\left( \Delta \otimes \alpha \right) \circ
f+\left( \alpha \otimes \Delta \right) \circ f-(f\otimes \alpha )\circ \rho
_{r}^{1}$, $\delta _{Hom,C}^{1,1}\left( f\right) =(id_{B}\otimes f)\circ
\rho _{l}^{1}-\Delta \circ f+(f\otimes id_{B})\circ \rho _{r}^{1},$ and $%
\rho _{l}^{1}=\rho _{r}^{1}=\Delta $%
\begin{align*}
\delta _{Hom,C}^{2,1}\circ \delta _{Hom,C}^{1,1}\left( f\right) & =(\alpha
\otimes \delta _{Hom,C}^{1,1}\left( f\right) )\circ \Delta -\left( \Delta
\otimes \alpha \right) \circ \delta _{Hom,C}^{1,1}\left( f\right) +\left(
\alpha \otimes \Delta \right) \circ \delta _{Hom,C}^{1,1}\left( f\right)  \\
& -(\delta _{Hom,C}^{1}\left( f\right) \otimes \alpha )\circ \Delta  \\
& =(\alpha \otimes \left( (id_{B}\otimes f)\circ \Delta -\Delta \circ
f+\left( f\otimes id_{B}\right) \circ \Delta \right) )\circ \Delta  \\
& -\left( \Delta \otimes \alpha \right) \circ \left( (id_{B}\otimes f)\circ
\Delta -\Delta \circ f+\left( f\otimes id_{B}\right) \circ \Delta \right)  \\
& +\left( \alpha \otimes \Delta \right) \circ \left( (id_{B}\otimes f)\circ
\Delta -\Delta \circ f+\left( f\otimes id_{B}\right) \circ \Delta \right)  \\
& -\left( \left( id_{B}\otimes f\right) \circ \Delta -\Delta \circ f+\left(
f\otimes id_{B}\right) \circ \Delta \right) \otimes \alpha )\circ \Delta  \\
& =\left( \alpha \otimes (id_{B}\otimes f)\circ \Delta \right) \circ \Delta
-\left( \alpha \otimes \left( \Delta \circ f\right) \right) \circ \Delta
+\left( \alpha \otimes \left( f\otimes id_{B}\right) \circ \Delta \right)
\circ \Delta  \\
& -\left( \Delta \otimes \alpha \right) (id_{B}\otimes f)\circ \Delta
+\left( \Delta \otimes \alpha \right) (\Delta \circ f)-\left( \Delta \otimes
\alpha \right) \left( f\otimes id_{B}\right) \circ \Delta  \\
& +\left( \alpha \otimes \Delta \right) (id_{B}\otimes f)\circ \Delta
-\left( \alpha \otimes \Delta \right) \left( \Delta \circ f\right) +\left(
\alpha \otimes \Delta \right) \left( f\otimes id_{B}\right) \circ \Delta  \\
& -\left( \left( (id_{B}\otimes f)\circ \Delta \right) \otimes \alpha
\right) \circ \Delta +\left( \left( \Delta \circ f\right) \otimes \alpha
\right) \circ \Delta -\left( \left( (f\otimes id_{B})\circ \Delta \right)
\otimes \alpha \right) \circ \Delta  \\
& \overset{(\ast )}{=}\left( \alpha \otimes (id_{B}\otimes f)\circ \Delta
\right) \circ \Delta +\left( \alpha \otimes \left( f\otimes id_{B}\right)
\circ \Delta \right) \circ \Delta  \\
& -\left( \Delta \otimes \alpha \right) (id_{B}\otimes f)\circ \Delta
+\left( \alpha \otimes \Delta \right) \left( f\otimes id_{B}\right) \circ
\Delta  \\
& -\left( \left( (id_{B}\otimes f)\circ \Delta \right) \otimes \alpha
\right) \circ \Delta -\left( \left( (f\otimes id_{B})\circ \Delta \right)
\otimes \alpha \right) \circ \Delta ,
\end{align*}

where (*) was obtained by the Hom-coassociativity of $\Delta $ and  Lemma %
\ref{lemma1}.

From Lemma \ref{4}, we immediately obtain $\delta
_{Hom,C}^{2,1}\circ \delta _{Hom,C}^{1,1}\left( f\right) =0$.

For the second identity, we have $\delta _{Hom,C}^{1,2}\left( f\right) =(id_{B}\otimes f)\circ \rho
_{l}^{2}-\Delta \circ f+\left( f\otimes id_{B}\right) \circ \rho _{r}^{2},$ $%
\delta _{Hom,C}^{1,1}\left( f\right) =(id_{B}\otimes f)\circ \rho
_{l}^{1}-\Delta \circ f+\left( f\otimes id_{B}\right) \circ \rho _{r}^{1},$ $%
\delta _{Hom,H}^{2,1}\left( f\right) =\lambda _{l}^{2}\circ (id_{B}\otimes
f)-f\circ \mu +\lambda _{r}^{2}\circ (f\otimes id_{B}),$ $\delta
_{Hom,H}^{1,1}\left( f\right) =\lambda _{l}^{1}\circ (id_{B}\otimes
f)-f\circ \mu +\lambda _{r}^{1}\circ (f\otimes id_{B}),$ and $\rho
_{l}^{1}=\rho _{r}^{1}=\Delta ,$ $\lambda _{l}^{1}=\lambda _{r}^{1}=\mu ,$ 
\begin{align*}
& \delta _{Hom,C}^{1,2}\circ \delta _{Hom,H}^{1,1}\left( f\right) 
=(id_{B}\otimes \left( \mu \circ (id_{B}\otimes f)-f\circ \mu +\mu \circ
(f\otimes id_{B})\right) )\circ \rho _{l}^{2} \\ &
-\Delta \circ \left( \mu \circ (id_{B}\otimes f)-f\circ \mu +\mu \circ
(f\otimes id_{B})\right)  \\&
+\left( \left( \mu \circ (id_{B}\otimes f)-f\circ \mu +\mu \circ (f\otimes
id_{B})\right) \otimes id_{B}\right) \circ \rho _{r}^{2} \\
&=(\left( id_{B}\otimes \mu \circ (id_{B}\otimes f)\right) \circ \rho
_{l}^{2}-\left( id_{B}\otimes f\circ \mu \right) \circ \rho _{l}^{2}+ \\
&\left( id_{B}\otimes \mu \circ (f\otimes id_{B})\right) \circ \rho
_{l}^{2}-\Delta \circ \mu \circ (id_{B}\otimes f)-\Delta \circ f\circ \mu  \\
&+\Delta \circ \mu \circ (f\otimes id_{B})+\left( \mu \circ (id_{B}\otimes
f)\otimes id_{B}\right) \circ \rho _{r}^{2} \\
&-\left( f\circ \mu \otimes id_{B}\right) \circ \rho _{r}^{2}+\left( \mu
\circ (f\otimes id_{B})\otimes id_{B}\right) \circ \rho _{r}^{2} \\
& \overset{(\ast \ast )}{=}\lambda _{l}^{2}\circ \left( id_{B}\otimes
(id_{B}\otimes f)\circ \Delta \right) -\left( id_{B}\otimes f\right) \circ
\Delta \circ \mu + \\
&\lambda _{r}^{2}\circ \left( \left( \left( id_{B}\otimes f\right) \circ
\Delta \right) \otimes id_{B})\right) -\lambda _{l}^{2}\circ \left(
id_{B}\otimes \left( \Delta \circ f\right) \right)  \\
&-\Delta \circ f\circ \mu +\lambda _{r}^{2}\circ \left( \left( \Delta \circ
f\right) \otimes id_{B})\right) +\lambda _{l}^{2}\circ \left( id_{B}\otimes
(f\otimes id_{B})\circ \Delta \right)  \\
&-\left( f\otimes id_{B}\right) \circ \Delta \circ \mu +\lambda
_{r}^{2}\circ \left( (f\otimes id_{B})\circ \Delta \otimes id_{B}\right)  \\
&=\delta _{Hom,H}^{2,1}\circ \delta _{Hom,C}^{1,1}\left( f\right) ,
\end{align*}

where ($\ast \ast $) is obtained by the compatibility condition and  Lemma \ref{7}.

For the third identity see (\cite{F. A 2010}).
\end{proof}

\begin{proposition}
\label{P} Let $\ D_{i}^{p,q}:\mathcal{C}_{Hom}^{p,q}\left(
B^{\otimes q},B^{\otimes p}\right) \longrightarrow \mathcal{C}%
_{Hom}^{p,q+1}\left( B^{\otimes q+1},B^{\otimes p}\right) $ be  linear
operators defined for \ $f\in \mathcal{C}_{Hom}^{p,q}\left( B^{\otimes
q},B^{\otimes p}\right) $ by:%
\begin{equation}
D_{i}^{p,q}\left( f\right) =\left\{ 
\begin{array}{c}
-\lambda _{l}^{p}\circ (\alpha ^{q-1}\otimes f)+f\circ \left( \mu \otimes
\alpha ^{\otimes (q-1)}\right) \text{ \ \ }if\text{ \ }i=0, \\ 
f\circ \left( \alpha ^{\otimes i}\otimes \mu \otimes \alpha ^{\otimes
(q-i-1)}\right) \text{ \ }if\text{ \ }\forall 1\leq i\leq q-2, \\ 
f\circ \left( \alpha ^{\otimes q-1}\otimes \mu \right) -\lambda
_{r}^{p}\circ (f\otimes \alpha ^{q-1})\text{ \ \ }if\text{ \ \ }i=q-1.%
\end{array}%
\right. 
\end{equation}%
Then%
\begin{equation}
D_{i}^{1,q+1}\circ D_{j}^{1,q}=D_{j+1}^{1,q+1}\circ D_{i}^{1,q}\ \ \ 0\leq
i<j\leq q \ Êand\ \ \delta
_{Hom,H}^{1,q}=\sum\limits_{i=0}^{q-1}(-1)^{i+1}D_{i}^{1,q}.  \label{(41)}
\end{equation}
\end{proposition}

\begin{proposition}
$\label{3}$Let $\ S_{i}^{p,q}:\mathcal{C}_{Hom}^{p,q}\left( B^{\otimes
q},B^{\otimes p}\right) \longrightarrow \mathcal{C}_{Hom}^{p+1,q}\left(
B^{\otimes q},B^{\otimes p+1}\right) $ be  linear operators defined for \ 
$g\in \mathcal{C}_{Hom}^{p,q}\left( B^{\otimes q},B^{\otimes p}\right) $ by: 
\begin{equation}
S_{i}^{p,q}\left( g\right) =\left\{ 
\begin{array}{c}
-(\alpha ^{p-1}\otimes g)\circ \rho _{l}^{q}+\left( \Delta \otimes \alpha
^{\otimes (p-1)}\right) \circ g\text{ \ \ }if\text{ \ }i=0, \\ 
\left( \alpha ^{\otimes i}\otimes \Delta \otimes \alpha ^{\otimes
(p-i-1)}\right) \circ g\text{ \ }if\text{ \ }\forall 1\leq i\leq p-2, \\ 
\left( \alpha ^{\otimes p-1}\otimes \Delta \right) \circ g-(g\otimes \alpha
^{p-1})\circ \rho _{r}^{q}\text{ \ \ }if\text{ \ \ }i=p-1.%
\end{array}%
\right. 
\end{equation}%
Then%
\begin{equation}
S_{i}^{p+1,1}\circ S_{j}^{p,1}=S_{j+1}^{p+1,1}\circ S_{i}^{p,1}\ \ \ 0\leq
i\leq j\leq n\ and\ \ \delta
_{Hom,C}^{p,1}=\sum\limits_{i=0}^{p-1}(-1)^{i+1}S_{i}^{p,1}.  \label{(43)}
\end{equation}
\end{proposition}

\begin{theorem}
\label{G}Let $B=\left( B,\mu ,\eta ,\Delta ,\varepsilon ,\alpha \right) $ be
a Hom-bialgebra and $\delta _{Hom,H}^{p,q}:\mathcal{C}_{Hom}^{p,q}%
\longrightarrow \mathcal{C}_{Hom}^{p,q+1}$, $\delta _{Hom,C}^{p,q}:\mathcal{C%
}_{Hom}^{p,q}\longrightarrow \mathcal{C}_{Hom}^{p+1,q}$ be  the operators
defined in ($\ref{(38)}$), ($\ref{(39)}$).  \\ Then $\left( \mathcal{C}%
_{Hom}^{p,q},\delta _{Hom,H}^{p,q},\delta _{Hom,C}^{p,q}\right) $ is a
bicomplex 
i.e.%
\begin{equation}
\delta _{Hom,C}^{p+1,q}\circ \delta _{Hom,C}^{p,q}=0 \text{ and  }\delta
_{Hom,H}^{p,q+1}\circ \delta _{Hom,H}^{p,q}=0.
\end{equation}
\end{theorem}

\begin{proof}
We prove the first identity.%
\begin{align*}
& \delta _{Hom,C}^{p+1,q}\circ \delta _{Hom,C}^{p,q} =\left(
\sum\limits_{i=0}^{p}(-1)^{i+1}S_{i}^{p+1,q}\right) \circ \left(
\sum\limits_{j=0}^{p-1}(-1)^{j+1}S_{j}^{p,q}\right) \\
&=\sum\limits_{i=0}^{p}\sum\limits_{j=0}^{p-1}(-1)^{i+j}S_{i}^{p+1,q}%
\circ S_{j}^{p,q} 
=\sum\limits_{0\leq j<i\leq n}(-1)^{i+j}S_{i}^{p+1,q}\circ
S_{j}^{p,q}+\sum\limits_{0\leq i\leq j\leq n-1}(-1)^{i+j}S_{i}^{p+1,q}\circ
S_{j}^{p,q} \\
&\overset{\eqref{(43)}}{=}\sum\limits_{0\leq j<i\leq
n}(-1)^{i+j}S_{i}^{p+1,q}S_{j}^{p,q}+\sum\limits_{0\leq i\leq j\leq
n-1}(-1)^{i+j}S_{j+1}^{p+1,q}S_{i}^{p,q} \\
&=\sum\limits_{0\leq j<i\leq
n}(-1)^{i+j}S_{i}^{p+1,q}S_{j}^{p,q}+\sum\limits_{0\leq i<k\leq
n}(-1)^{i+k-1}S_{k}^{p+1,q}S_{i}^{p,q}=0.
\end{align*}

For the second identity, see (\cite{F. A 2010}).
\end{proof}

There is a canonical way to construct a complex from a given bicomplex.
The cochains are given by
\begin{equation*}
\mathcal{\hat{C}}_{Hom}=\sum_{n}\oplus \mathcal{\hat{C}}_{Hom}^{n},\quad 
\mathcal{\hat{C}}_{Hom}^{n}=\sum_{p+q=n+1,p,q\geq 1}\oplus \mathcal{C}%
_{Hom}^{p,q},\quad \mathcal{C}_{Hom}^{p,q}=Hom_{\Bbbk }\left( B^{\otimes
q},B^{\otimes p}\right) \text{ }n\geq 1.
\end{equation*}

The coboundary operator is $\delta _{Hom}^{n} :\mathcal{\hat{C}}_{Hom}^{n}\longrightarrow \mathcal{%
\hat{C}}_{Hom}^{n+1}$ defined as
\begin{eqnarray}
{\delta _{Hom}^{n}}_{ \left\vert \mathcal{C}_{Hom}^{n+1-q,q}\right.  }&=&\delta
_{Hom,H}^{p,q}\oplus \left( -1\right) ^{q}\delta _{Hom,C}^{p,q},\quad 1\leq
q\leq n,\text{ }p=n+1-q.
\end{eqnarray}

Hence, for each $n\geq 1$, we get a complex%
\begin{equation*}
\begin{array}{c}
0\longrightarrow \mathcal{\hat{C}}_{Hom}^{1}\longrightarrow ^{\delta
_{Hom}^{1}}\mathcal{\hat{C}}_{Hom}^{2}\longrightarrow ^{\delta _{Hom}^{2}}%
\mathcal{\hat{C}}_{Hom}^{3}\longrightarrow ^{\delta _{Hom}^{3}}\mathcal{\hat{%
C}}_{Hom}^{4}\cdots%
\end{array}%
\end{equation*}

\begin{remark}
The composite $\delta _{Hom}^{2}\circ \delta _{Hom}^{1}=0$,  according to 
Proposition $\ref{S}$.
\end{remark}

We define the $n-th$ cohomology group of the above complex to be the Hom-bialgebra cohomology of $B$, which will be denoted by $H_{Hom}^{n}\left(
B,B\right) $, $n\geq 1$.

The kernel of $\delta _{Hom}^{n}$ in $\mathcal{\hat{C}}_{Hom}^{n}$\ is the
space of $n$-cocycles  defined by:%
\begin{equation}
Z_{Hom}^{n}\left( B,B\right) =\left\{ \rho \in \mathcal{\hat{C}}%
_{Hom}^{n}:\quad \delta _{Hom}^{n}\left( \rho \right) =0\right\} .
\end{equation}

The image of $\delta _{Hom}^{n}$ is the space of $n$-coboundaries  defined
by:%
\begin{equation}
B_{Hom}^{n}\left( B,B\right) =\left\{ \rho \in \mathcal{\hat{C}}%
_{Hom}^{n}:~\quad \rho =\delta _{Hom}^{n-1}\left( \psi \right) ,\quad \psi
\in \mathcal{\hat{C}}_{Hom}^{n-1}\right\} .
\end{equation}

The Gerstenhaber-Shack cohomology group of the Hom-bialgebra $B=\left( B,\mu ,\eta
,\Delta ,\varepsilon ,\alpha \right) $ with coefficient in it self is%
\begin{equation}
H_{Hom}^{n}\left( B,B\right) =Z_{Hom}^{n}\left( B,B\right)
/B_{Hom}^{n}\left( B,B\right) .
\end{equation}

In particular,

\begin{equation*}
 H_{Hom}^{1}\left( B,B\right) =\left\{ f:B\longrightarrow B: \text{ }%
\delta _{Hom,H}^{0,1}\left( f\right) =0\text{ and }\delta
_{Hom,C}^{1,0}\left( f\right) =0\right\} ,
\end{equation*}

where%
\begin{eqnarray*}
\delta _{Hom,H}^{0,1}\left( f\right)  &=&\mu \circ (id_{B}\otimes f)-f\circ
\mu +\mu \circ (g\otimes id_{B}) , \\
\delta _{Hom,C}^{1,0}\left( f\right)  &=&(id_{B}\otimes f)\circ \Delta
-\Delta \circ f+(f\otimes id_{B})\circ \Delta .
\end{eqnarray*}


The  cohomology groups $H_{Hom}^{2}\left( B,B\right) $ and $
H_{Hom}^{3}\left( B,B\right) $ play an important role in deformation theory.
We provide their explicit definitions.
\begin{equation*}
\bullet \ Z_{Hom}^{2}\left( B,B\right) =\left\{ (f,g)\in \mathcal{\hat{C}}%
_{Hom}^{2},\text{ }\delta _{Hom,H}^{1,2}\left( f\right) =0,\ \delta
_{Hom,C}^{1,2}\left( f\right) +\delta _{Hom,H}^{2,1}\left( g\right) =0,\
\delta _{Hom,C}^{2,1}\left( g\right) =0\right\} ,
\end{equation*}%
where for $f:B\otimes B$ $\longrightarrow B$ and $g:
B\longrightarrow B\otimes B$, we have 
\begin{equation*}
\delta _{Hom,H}^{1,2}\left( f\right) =\lambda _{l}^{1}\circ (\alpha \otimes
f)-f\circ \left( \mu \otimes \alpha \right) +f\circ \left( \alpha \otimes
\mu \right) -\lambda _{r}^{1}\circ (f\otimes \alpha );
\end{equation*}%
\begin{eqnarray*}
\delta _{Hom,C}^{1,2}\left( f\right) +\delta _{Hom,H}^{2,1}\left( g\right) 
&=&\left( (id_{B}\otimes f)\circ \rho _{l}^{2}-\Delta \circ f+(f\otimes
id_{B})\circ \rho _{r}^{2}\right) + \\
&&(\lambda _{l}^{2}\circ (id_{B}\otimes g)-g\circ \mu +\lambda _{r}^{2}\circ
(g\otimes id_{B})),
\end{eqnarray*}%
\begin{equation*}
\delta _{Hom,C}^{2,1}\left( g\right) =(\alpha \otimes g)\circ \rho
_{l}^{1}-\left( \Delta \otimes \alpha \right) \circ g+\left( \alpha \otimes
\Delta \right) \circ g-\left( f\otimes \alpha \right) \circ \rho _{r}^{1},
\end{equation*}%
where $\rho _{l}^{1}=\rho _{r}^{1}=\Delta ,$ $\lambda _{l}^{1}=\lambda
_{r}^{1}=\mu .$ 
\begin{equation*}
 \bullet \ B_{Hom}^{2}\left( B,B\right) =\left\{ (f,g)\in \mathcal{\hat{C}}%
_{Hom}^{2},\text{ }\exists h:B\longrightarrow B,\text{ }f=\delta
_{Hom,H}^{0,1}\left( h\right) ,\text{ }g=\delta _{Hom,C}^{1,0}\left(
h\right) \right\} ,
\end{equation*}%
where%
\begin{eqnarray*}
\delta _{Hom,H}^{0,1}\left( h\right)  &=&\mu \circ (id_{B}\otimes h)-h\circ
\mu +\mu \circ (h\otimes id_{B}), \\
\delta _{Hom,C}^{1,0}\left( h\right)  &=&(id_{B}\otimes h)\circ \Delta
-\Delta \circ h+\left( h\otimes id_{B}\right) \circ \Delta .
\end{eqnarray*}%
\begin{equation*}
\bullet \ Z_{Hom}^{3}\left( B,B\right) =\left\{ 
\begin{array}{c}
(F,H,G)\in \mathcal{\hat{C}}_{Hom}^{3},\text{ }\delta _{Hom,H}^{1,3}\left(
F\right) =0,\delta _{Hom,H}^{2,2}\left( H\right) -\delta
_{Hom,C}^{1,3}\left( F\right) =0, \\ 
\delta _{Hom,C}^{2,2}\left( H\right) +\delta _{Hom,H}^{3,1}\left( G\right)
=0,\delta _{Hom,C}^{3,1}\left( G\right) =0%
\end{array}%
\right\},
\end{equation*}
and%
\begin{equation*}
 \bullet \   B_{Hom}^{3}\left( B,B\right) =\left\{ 
\begin{array}{c}
(F,H,G)\in \mathcal{\hat{C}}_{Hom}^{3},\text{ }\exists \left( f,g\right) \in 
\mathcal{\hat{C}}_{Hom}^{2},,\text{ }F=\delta _{Hom,H}^{1,2}\left( f\right) ,
\\ 
H=\delta _{Hom,C}^{1,2}\left( f\right) +\delta _{Hom,H}^{2,1}\left( g\right)
,G=\delta _{Hom,C}^{2,1}\left( g\right) 
\end{array}%
\right\},
\end{equation*}%
where  $F:B\otimes B\otimes B$ $\longrightarrow B$, $H:$ $B\otimes
B\longrightarrow B\otimes B$ and $G:B\longrightarrow B\otimes B\otimes B.$


\begin{example}
We consider $(T_{2})_{\lambda }$, the $4$-dimensional Taft-Sweedler
Hom-bialgebra defined in Example \ref{Taft} for which we compute for  $\lambda \neq 1$ and $\lambda \neq 0,$  the first cohomology groups.

The space of $1$-cohomology classes of $\ (T_{2})_{\lambda }$ 
\[
H_{Hom}^{1}\left( (T_{2})_{\lambda },(T_{2})_{\lambda }\right) =\left\{ f:(T_{2})_{\lambda }\longrightarrow (T_{2})_{\lambda }\text{: }\delta
_{Hom,H}^{1,1}\left( f\right) =0\text{ and }\delta _{Hom,C}^{1,1}\left(
f\right) =0\right\}
\]%
The elements are defined with respect to a basis $\{e_{1},e_{2},e_{3},e_{4}\}$ by 
\[
f\left( e_{1}\right) =0,\text{ }f\left( e_{2}\right) =0,f\left( e_{3}\right)
=ae_{3},f\left( e_{4}\right) =a e_{4},\text{where }a \text{ is a free parameter.}
\]

The 2-cocycles of the Hom-bialgebras $(T_{2})_{\lambda }$ 
\begin{small}
\[
 Z_{Hom}^{2}\left( (T_{2})_{\lambda },(T_{2})_{\lambda }\right) =\left\{ (f,g)\in \mathcal{\hat{C}}%
_{Hom}^{2} : \text{ }\delta _{Hom,H}^{1,2}\left( f\right) =0,\delta
_{Hom,C}^{2,1}\left( g\right) =0,\ \delta _{Hom,C}^{1,2}\left( f\right)
+\delta _{Hom,H}^{2,1}\left( g\right) =0\right\}.
\]%
\end{small}
They are defined with respect to the basis $\{e_{1},e_{2},e_{3},e_{4}\}$, by the table which 
describes multiplying the $i^{th}$ row elements by the $j^{th}$ column elements
 with respect to the same basis :%
\[
\begin{array}{|c|c|c|c|c|}
\hline
f & e_{1} & e_{2} & e_{3} & e_{4} \\ \hline
e_{1} & a\left( e_{1}+e_{2}\right) & a\left( e_{1}+e_{2}\right) & \lambda
a\left( e_{3}+e_{4}\right) & \lambda a\left( e_{3}+e_{4}\right) \\ \hline
e_{2} & a\left( e_{1}+e_{2}\right) & a\left( e_{1}-3e_{2}\right) & \lambda
\left( ce_{4}-ae_{3}\right) & \lambda \left( \left( 2a-c\right)
e_{3}-ae_{4}\right) \\ \hline
e_{3} & \lambda a\left( e_{3}-e_{4}\right) & -\lambda \left(
ae_{3}+ce_{4}\right) & 0 & 0 \\ \hline
e_{4} & \lambda a\left( e_{4}-e_{3}\right) & -\lambda \left( \left(
2a-c\right) e_{3}+ae_{4}\right) & 0 & 0 \\ \hline
\end{array}%
\]%
and
\begin{align*}
& g\left( e_{1}\right) =-a\left( e_{1}\otimes e_{1}+e_{1}\otimes
e_{2}+e_{2}\otimes e_{1}-e_{2}\otimes e_{2}\right) ,\\
& g\left( e_{2}\right) =-a\left( e_{1}\otimes e_{1}-e_{1}\otimes
e_{2}-e_{2}\otimes e_{1}+e_{2}\otimes e_{2}\right) ,\\
& g\left( e_{3}\right) =\lambda a\left( e_{1}\otimes e_{3}-e_{2}\otimes
e_{3}-e_{3}\otimes e_{1}-e_{3}\otimes e_{2}\right) ,\\
& g\left( e_{4}\right) =-\lambda a\left( e_{1}\otimes e_{4}+e_{2}\otimes
e_{4}-e_{4}\otimes e_{1}+e_{4}\otimes e_{2}\right) .
\end{align*}

The space of $2$-coboundaries of the Hom-bialgebra $(T_{2})_{\lambda }$%
\emph{\ }is defined by 
\[
 B_{Hom}^{2}\left( (T_{2})_{\lambda },(T_{2})_{\lambda }\right) =\left\{ (f,g)\in \mathcal{\hat{C}}%
_{Hom}^{2},\text{ }\exists h:(T_{2})_{\lambda }\longrightarrow (T_{2})_{\lambda },\text{ }f=\delta
_{Hom,H}^{1,1}\left( h\right) ,\text{ }g=\delta _{Hom,C}^{1,1}\left(
h\right) \right\} 
\]%
such that%
\[
\begin{array}{|c|c|c|c|c|}
\hline
f & e_{1} & e_{2} & e_{3} & e_{4} \\ \hline
e_{1} & 0 & 0 & 0 & 0 \\ \hline
e_{2} & 0 & 0 & \lambda ce_{4} & -\lambda ce_{3} \\ \hline
e_{3} & 0 & -\lambda ce_{4} & 0 & 0 \\ \hline
e_{4} & 0 & \lambda ce_{3} & 0 & 0 \\ \hline
\end{array}%
\]%
and $g\left( e_{i}\right) =0$ for $i\in \left\{ 1,2,3,4\right\} ,$ where $%
\lambda ,a,c\in \Bbbk $ are free parameters.

The $2^{th}$ cohomology group of  $(T_{2})_{\lambda }$ is
the quotient $$H_{Hom}^{2}\left( (T_{2})_{\lambda },(T_{2})_{\lambda }\right) =Z_{Hom}^{2}\left( (T_{2})_{\lambda },(T_{2})_{\lambda }\right)
/B_{Hom}^{2}\left( (T_{2})_{\lambda },(T_{2})_{\lambda }\right),$$ which is  defined, with respect to the basis $%
\{e_{1},e_{2},e_{3},e_{4}\}$, by%
\[
\begin{array}{|c|c|c|c|c|}
\hline
f & e_{1} & e_{2} & e_{3} & e_{4} \\ \hline
e_{1} & a\left( e_{1}+e_{2}\right)  & a\left( e_{1}+e_{2}\right)  & \lambda
a\left( e_{3}+e_{4}\right)  & \lambda a\left( e_{3}+e_{4}\right)  \\ \hline
e_{2} & a\left( e_{1}+e_{2}\right)  & a\left( e_{1}-3e_{2}\right)  & 
-\lambda ae_{3} & \lambda a\left( 2e_{3}-e_{4}\right)  \\ \hline
e_{3} & \lambda a\left( e_{3}-e_{4}\right)  & -\lambda ae_{3} & 0 & 0 \\ 
\hline
e_{4} & \lambda a\left( e_{4}-e_{3}\right)  & -\lambda a\left(
2e_{3}+e_{4}\right)  & 0 & 0 \\ \hline
\end{array}%
\]%
and%
\begin{align*}
& g\left( e_{1}\right) =-a\left( e_{1}\otimes e_{1}+e_{1}\otimes
e_{2}+e_{2}\otimes e_{1}-e_{2}\otimes e_{2}\right) ,\\
& g\left( e_{2}\right) =-a\left( e_{1}\otimes e_{1}-e_{1}\otimes
e_{2}-e_{2}\otimes e_{1}+e_{2}\otimes e_{2}\right) ,\\
& g\left( e_{3}\right) =\lambda a\left( e_{1}\otimes e_{3}-e_{2}\otimes
e_{3}-e_{3}\otimes e_{1}-e_{3}\otimes e_{2}\right) ,\\
& g\left( e_{4}\right) =-\lambda a\left( e_{1}\otimes e_{4}+e_{2}\otimes
e_{4}-e_{4}\otimes e_{1}+e_{4}\otimes e_{2}\right) .
\end{align*}
\end{example}

\section{Formal Deformations of Hom-bialgebras}

We discuss, in this section, a deformation theory for Hom-bialgebras following Gerstenhaber's  approach. Let $(B,\mu ,\eta, \Delta ,\varepsilon, \beta )$ be a
Hom-bialgebra and  $\Bbbk \lbrack \lbrack t]]$ be the power series ring in
one variable $t$ and coefficients in $\Bbbk $.  Let $B[[t]]$ be the set of
formal power series whose coefficients are elements of $B$ (note that $B[[t]]
$ is obtained by extending the coefficients domain of $B$ from $\Bbbk $ to $%
\Bbbk \lbrack \lbrack t]]$). Then $B[[t]]$ is a $\Bbbk \lbrack \lbrack t]]$%
-module and when $B$  is finite dimensional, we have $B[[t]]=B\otimes _{\Bbbk
}\Bbbk \lbrack \lbrack t]]$. Notice that $B$ is a submodule of $B[[t]]$.

\begin{definition}
Let $B=\left( B,\mu _{0},\eta _{0},\Delta _{0},\varepsilon _{0},\alpha
_{0}\right) $ be a Hom-bialgebra. A formal Hom-bialgebra deformation of $B$
over $\Bbbk \lbrack \lbrack t]]$ consists of  $\Bbbk \lbrack \lbrack t]]$%
-bilinear maps 
\begin{equation*}
\mu _{t}=\sum\limits_{i\geq 0}\mu _{i}t^{i\text{ }}:B\otimes
B\longrightarrow B[[t]],\quad with\quad \mu _{i}:B\otimes B\longrightarrow B,
\end{equation*}%
\begin{equation*}
\Delta _{t}=\sum\limits_{j\geq 0}\Delta _{j}t^{j\text{ }}:B\longrightarrow
B[[t]]\otimes B[[t]],\quad with\quad \Delta _{j}:B\longrightarrow B\otimes B,
\end{equation*}%
\begin{equation*}
\eta _{t}=\sum\limits_{i\geq 0}\eta _{i}t^{i\text{ }}:B\longrightarrow
\Bbbk \lbrack \lbrack t]],\quad \varepsilon _{t}=\sum\limits_{j\geq
0}\varepsilon _{j}t^{j\text{ }}:B[[t]]\longrightarrow \Bbbk ,
\end{equation*}%
where we  assume 
$\eta _{t}\left( t\right) =t$ and $\varepsilon _{t}\left( t\right) =t,$ 
 and 
\begin{equation*}
\alpha _{t}=\sum\limits_{k\geq \grave{a}}\alpha _{k}t^{k\text{ }%
}:B\longrightarrow B[[t]],\quad with\quad \alpha _{k}:B\longrightarrow B,
\end{equation*}%
such that   $B_{t}=\left( B[[t]],\mu _{t},\eta _{t},\Delta
_{t},\varepsilon _{t},\alpha _{t}\right) $ is  a Hom-bialgebra.
\end{definition}

If we study only deformations of $B=\left( B,\mu _{0},\eta _{0},\Delta
_{0},\varepsilon _{0},\alpha _{0}\right) $ in which the unit and the counit are conserved, that is  $B_{t}=\left( B%
\left[ \left[ t\right] \right] ,\mu _{t},\eta _{0},\Delta _{t},\varepsilon
_{0},\alpha _{t}\right) $, then we as consider a deformation as a triple
 $\left( \mu _{t},\Delta _{t},\alpha _{t}\right) $ 
 satisfying 
\begin{equation}
\left\{ 
\begin{array}{c}
 \mu _{t}\circ \left( \alpha_t \otimes \mu _{t}\right) =\mu _{t}\circ
\left( \mu _{t}\otimes \alpha_t \right) \qquad (\text{ formal
Hom-associativity}), \\ 
 \left( \Delta _{t}\otimes \alpha_t \right) \circ \Delta _{t}=\left(
\alpha_t \otimes \Delta _{t}\right) \circ \Delta _{t}\qquad (\text{ formal 
Hom-coassociativity}), \\ 
 \Delta _{t}\circ \mu _{t}=\mu _{t}^{\otimes 2}\circ \left(
id_{B}\otimes \tau \otimes id_{B}\right) \circ \Delta _{t}^{\otimes 2}\qquad
(\text{ formal compatibility}).%
\end{array}%
\right.   \label{(53)}
\end{equation}
The previous system is called deformation equation.
\subsection{Deformations equation}

Now, we discuss the deformation equation in terms of cohomology. The first
problem is to give conditions about $\mu _{i},$ $\Delta _{j}$ and $\alpha
_{k}$\ such that the deformation $\left( \mu _{t},\Delta _{t},\alpha
_{t}\right) $ satisfies Hom-associativity, Hom-coassociativity and compatibility conditions.

We study  equations ($\ref{(53)}$) and thus characterize the deformations
of Hom-bialgebras. The  coefficients of $t^{s}$
yields :

\begin{equation*}
\left\{ 
\begin{array}{c}
\sum\limits_{\substack{ i+j+k=s  \\ i,j,k\geq 0}}\ \left( \mu _{i}\circ
\left( \mu _{j}\otimes \alpha _{k}\right) -\mu _{i}\circ \left( \alpha
_{k}\otimes \mu _{j}\right) \right) =0\text{ \ \ \ \ }s=0,1,2,\cdots \\ 
\sum\limits_{\substack{ i+j+k=s  \\ i,j,k\geq 0}}\ \left( \Delta
_{j}\otimes \alpha _{k}\right) \circ \Delta _{i}-\left( \alpha _{k}\otimes
\Delta _{j}\right) \circ \Delta _{i}=0\text{ \ \ \ \ }s=0,1,2,\cdots \\ 
\sum\limits_{\substack{ i+j=s  \\ i,j\geq 0}}\left( \Delta _{i}\circ \mu
_{j}\right) =\sum\limits_{\substack{ i+j+k+r=s  \\ i,j,k,r\geq 0}}%
\left( \mu _{i}\otimes \mu _{j}\right) \circ \tau _{2,3}\circ \left( \Delta
_{k}\otimes \Delta _{r}\right) \text{ \ \ \ \ }s=0,1,2,\cdots%
\end{array}%
\right.
\end{equation*}

This infinite system, called the deformation equation, gives the necessary
and sufficient conditions for $B_{t}$ to be a  Hom-bialgebra. It may be written%
\begin{equation*}
\left\{ 
\begin{array}{c}
\sum\limits_{i=0}^{s}\sum\limits_{j=0}^{s-i}\ \left( \mu _{i}\circ \left(
\alpha _{j}\otimes \mu _{s-i-j}\right) -\mu _{i}\circ \left( \mu
_{s-i-j}\otimes \alpha _{j}\right) \right) =0\text{ \ \ \ \ }s=0,1,2,\cdots \\ 
\sum\limits_{i=0}^{s}\sum\limits_{j=0}^{s-i}\ \left( \Delta
_{s-i-j}\otimes \alpha _{j}\right) \circ \Delta _{i}-\left( \alpha
_{j}\otimes \Delta _{s-i-j}\right) \circ \Delta _{i}=0\text{ \ \ \ \ }%
s=0,1,2,\cdots\\ 
\sum\limits_{i=0}^{s}\left( \left( \Delta _{i}\circ \mu _{s-i}\right)
-\sum\limits_{j=0}^{s-i}\sum\limits_{k=0}^{s-i-j}\left( \mu _{i}\otimes
\mu _{j}\right) \circ \tau _{2,3}\circ \left( \Delta _{k}\otimes \Delta
_{s-i-j-k}\right) \right) =0\text{ \ \ \ \ }s=0,1,2,\cdots%
\end{array}%
\right.
\end{equation*}

We call\ $\alpha $-associator the map%
\begin{equation*}
Hom\left( B^{\otimes 2},B\right) \times Hom\left( B^{\otimes 2},B\right)
\longrightarrow Hom\left( B^{\otimes 3},B\right) ,\ \left( \mu _{i},\mu
_{j}\right) \longmapsto \mu _{i}\circ _{\alpha }\mu _{j}
\end{equation*}%
defined by%
\begin{equation*}
\mu _{i}\circ _{\alpha }\mu _{j}=\mu _{i}\circ \left( \alpha \otimes \mu
_{j}\right) -\mu _{i}\circ \left( \mu _{j}\otimes \alpha \right) .
\end{equation*}

We call\ $\alpha $-coassociator the map%
\begin{equation*}
Hom\left( B,B^{\otimes 2}\right) \times Hom\left( B,B^{\otimes 2}\right)
\longrightarrow Hom\left( B,B^{\otimes 3}\right) ,\left( \Delta _{i},\Delta
_{j}\right) \longmapsto \Delta _{i}\circ _{\alpha }\Delta _{j}
\end{equation*}%
defined by%
\begin{equation*}
\Delta _{i}\circ _{\alpha }\Delta _{j}=\left( \Delta _{j}\otimes \alpha
\right) \circ \Delta _{i}-\left( \alpha \otimes \Delta _{j}\right) \circ
\Delta _{i}.
\end{equation*}

By using $\alpha _{j}$-associators and $\alpha _{j}$-coassociators, the
deformation equation may be written as follows%
\begin{equation*}
\left\{ 
\begin{array}{c}
\sum\limits_{i=0}^{s}\sum\limits_{j=0}^{s-i}\ \mu _{i}\circ _{\alpha
_{j}}\mu _{s-i-j}=0\text{ \ \ \ \ }s=0,1,2,\cdots\\ 
\sum\limits_{i=0}^{s}\sum\limits_{j=0}^{s-i}\ \Delta _{i}\circ _{\alpha
_{j}}\Delta _{s-i-j}=0\text{ \ \ \ \ }s=0,1,2,\cdots \\ 
\sum\limits_{i=0}^{s}\left( \left( \Delta _{i}\circ \mu _{s-i}\right)
-\sum\limits_{j=0}^{s-i}\sum\limits_{k=0}^{s-i-j}\left( \mu _{i}\otimes
\mu _{j}\right) \circ \tau _{2,3}\circ \left( \Delta _{k}\otimes \Delta
_{s-i-j-k}\right) \right) =0\text{ \ \ \ \ }s=0,1,2,\cdots%
\end{array}%
\right.
\end{equation*}
The first equations corresponding to $s=0$, are the Hom-associativity
condition for $\mu _{0},$ the Hom-coassociativity condition for $\Delta _{0}$
and the compatibility condition of $\mu _{0}$ and $\Delta _{0}$.\\
If the structure map is not deformed, then one gets the following system, where $\alpha_0=\alpha$, 
\begin{equation*}
\left\{ 
\begin{array}{c}
\sum\limits_{i=1}^{s-1}\ \mu _{i}\circ _{\alpha
}\mu _{s-i}=-\delta _{Hom,H}^{1,2}\left( \mu _{1}\right)\text{ \ \ \ \ }s=1,2,\cdots \\ 
\sum\limits_{i=0}^{s}\ \Delta _{i}\circ _{\alpha
}\Delta _{s-i}=-\delta _{Hom,C}^{2,1}\left( \Delta
_{1}\right)\text{ \ \ \ \ }s=1,2,\cdots\\ 
\sum\limits_{i=1}^{s-1}\left( \left( \Delta _{i}\circ \mu _{s-i}\right)
-\sum\limits_{j=0}^{s-i}\sum\limits_{k=0}^{s-i-j}\left( \mu _{i}\otimes
\mu _{j}\right) \circ \tau _{2,3}\circ \left( \Delta _{k}\otimes \Delta
_{s-i-j-k}\right) \right) =0\text{ \ \ \ \ }s=1,2,\cdots%
\end{array}%
\right.
\end{equation*}

In particular, for  $s=1$ we have$\ $

$\bullet\  \mu _{0}\circ _{\alpha }\mu _{1}+\mu _{1}\circ _{\alpha }\mu _{0}=0,
$ which is equivalent to $\delta _{Hom,H}^{1,2}\left( \mu _{1}\right) $ $=0$,

$\bullet\  \Delta _{0}\circ _{\alpha }\Delta _{1}+\Delta _{1}\circ _{\alpha
}\Delta _{0}=0,$ which is equivalent to $\delta _{Hom,C}^{2,1}\left( \Delta
_{1}\right) $ $=0$,

$\bullet\  $the compatibility condition which is equivalent to $\delta
_{Hom,C}^{1,2}\left( \mu _{1}\right) +\delta _{Hom,H}^{2,1}\left( \Delta
_{1}\right) =0$. \\Therefore, we have

\begin{proposition}

The first term $\left( \mu _{1},\Delta _{1}\right) $ of a deformation of a Hom-bialgebra, where the structure map is not deformed,   is always a $2$%
-cocycle for the Hom-bialgebra Gerstenhaber-Schack cohomology.

\end{proposition}

\begin{definition}
 A  $2$-cocycle $\left( \mu _{1},\Delta
_{1}\right) $ is said to be integrable if there exists a pair $\left( \mu
_{t},\Delta _{t}\right) _{t\geq 0}$ such that $\mu _{t}=\sum\limits_{i\geq
0}\mu _{i}t^{i\text{ }}$ and $\Delta _{t}=\sum\limits_{i\geq 0}\Delta
_{i}t^{i}$ defining a formal Hom-bialgebra deformation $B_{t}=\left( B\left[ \left[ t%
\right] \right] ,\mu _{t},\eta ,\Delta _{t},\varepsilon ,\alpha \right) $ of 
$B.$
\end{definition}

\subsection{Equivalent and trivial deformations}

In this section, we characterize  equivalent and trivial deformations of
Hom-bialgebras.

\begin{definition}
Let $B_{0}=\left( B,\mu _{0},\eta, \Delta _{0},\varepsilon, \alpha \right) $ be a
Hom-bialgebra. Given two deformations of $B_{0}$, $B_{t}=\left(
B,\mu _{t},\eta, \Delta _{t},\varepsilon, \alpha \right) $ and $B_{t}^{\prime }=\left( B,\mu
_{t}^{\prime },\eta, \Delta _{t}^{\prime },\varepsilon, \alpha \right) $, where $\mu
_{t}=\sum\limits_{i\geq 0}\mu _{i}t^{i\text{ }},$ $\Delta
_{t}=\sum\limits_{j\geq 0}\Delta _{j}t^{j\text{ }},$ $\mu _{t}^{\prime
}=\sum\limits_{i\geq 0}\mu _{i}^{\prime }t^{i\text{ }},$ $\Delta
_{t}^{\prime }=\sum\limits_{j\geq 0}\Delta _{j}^{\prime }t^{j\text{ }}$
with $\mu _{0}^{\prime }=\mu _{0},$ and $\Delta _{0}^{\prime }=\Delta _{0}.$

We say that they are equivalent if there is a formal isomorphism $\Phi
_{t}:B\longrightarrow B\left[ \left[ t\right] \right] $ which is a $\Bbbk %
\left[ \left[ t\right] \right] $-linear map that may be written in the form $%
\Phi _{t}=\sum\limits_{i\geq 0}\Phi _{i}t^{i\text{ }}$, where $\Phi _{i}\in
End_{\Bbbk }(B)$ and $\Phi _{0}=id_{B}$, such that%
\begin{equation}
\Phi _{t}\circ \mu _{t}=\mu _{t}^{\prime }\circ \left( \Phi _{t}\otimes \Phi
_{t}\right) ,  \label{(55)}
\end{equation}%
\begin{equation}
\Phi _{t}\otimes \Phi _{t}\circ \Delta _{t}=\Delta _{t}^{\prime }\circ \Phi
_{t} \label{(56)}
\end{equation}%
and%
\begin{equation}
\Phi _{t}\circ \alpha =\alpha \circ \Phi _{t}.  \label{(57)}
\end{equation}
A deformation $B_{t}$ of $B_{0}$ is said to be trivial if and only if $B_{t}$
is equivalent to $B_{0}.$
\end{definition}
We discuss in the following the equivalence of two deformations. 

Equation (\ref{(55)}) is equivalent to%
\begin{equation}
\sum\limits_{i,j\geq 0}\left( \Phi _{i}\circ \mu _{j}\right) t^{i\text{ }%
+j^{\text{ }}}=\sum\limits_{i,j,k\geq 0}\mu _{i}^{\prime }\circ \left( \Phi
_{j}\otimes \Phi _{k}\right) t^{i\text{ }+j+k\text{ }}.
\end{equation}%
By identification of the coefficients, one obtains that the constant
coefficients are identical, i.e.%
\begin{equation*}
\mu _{0}=\mu _{0}^{\prime }\quad and\quad \Phi _{0}=id_{B}.
\end{equation*}%
For the coefficients of $t$, one finds%
\begin{eqnarray*}
\left( \Phi _{0}\circ \mu _{1}\right) +\left( \Phi _{1}\circ \mu _{0}\right)
&=&\mu _{1}^{\prime }\circ \left( \Phi _{0}\otimes \Phi _{0}\right) +\mu
_{0}^{\prime }\circ \left( \Phi _{1}\otimes \Phi _{0}\right) +\mu
_{0}^{\prime }\circ \Phi _{1}\left( \Phi _{0}\otimes \right)  \\
\mu _{1}+\Phi _{1}\circ \mu _{0} &=&\mu _{1}^{\prime }+\mu _{0}\circ \left(
\Phi _{1}\otimes id_{B}\right) +\mu _{0}\circ \left( id_{B}\otimes \Phi
_{1}\right) .
\end{eqnarray*}
Hence
\begin{equation}
\mu _{1}^{\prime }=\mu _{1}-\left( \mu _{0}\circ \left( \Phi _{1}\otimes
id_{B}\right) -\Phi _{1}\circ \mu _{0}+\mu _{0}\circ \left( id_{B}\otimes
\Phi _{1}\right) \right) .  \label{(59)}
\end{equation}

Equation \eqref{(56)} is equivalent to%
\begin{equation*}
\sum\limits_{i,j\geq 0}\left( \Delta _{i}^{\prime }\circ \Phi _{j}\right)
t^{i\text{ }+j^{\text{ }}}=\sum\limits_{i,j,k\geq 0}\left( \Phi _{i}\otimes
\Phi _{j}\right) \circ \Delta _{k}t^{i\text{ }+j+k\text{ }}.
\end{equation*}

Similarly for the comultiplication, setting $\Delta _{0}=\Delta
_{0}^{\prime }$ and $\Phi _{0}=id_{B}$,  the coefficients of $t$ leads to%
\begin{equation*}
\left( \Delta _{0}^{\prime }\circ \Phi _{1}\right) +\left( \Delta
_{1}^{\prime }\circ \Phi _{0}\right) =\left( \Phi _{0}\otimes \Phi
_{0}\right) \circ \Delta _{1}+\left( \Phi _{0}\otimes \Phi _{1}\right) \circ
\Delta _{0}+\left( \Phi _{1}\otimes \Phi _{0}\right) \circ \Delta 
\end{equation*}%
Hence
\begin{equation}
\Delta _{1}^{\prime }=\Delta _{1}+\left( id_{B}\otimes \Phi _{1}\right)
\circ \Delta _{0}-\left( \Delta _{0}\circ \Phi _{1}\right) +\left( \Phi
_{1}\otimes id_{B}\right) \circ \Delta .  \label{(60)}
\end{equation}%
Homomorphisms condition  \eqref{(57)} is equivalent to
$ 
\sum\limits_{i\geq 0}\left( \Phi _{i}\circ \alpha \right) t^{i}
=\sum\limits_{i\geq 0}\left( \alpha \circ \Phi _{i}\right) t^{i}$. Therefore $$
 \Phi _{i}\circ \alpha   = \alpha \circ \Phi
_{i} , \text{ for all } i>0.$$

The first and second order conditions \eqref{(59)},\eqref{(60)} 
may be written%
\begin{equation}
\mu _{1}^{\prime }=\mu _{1}-\delta _{Hom,H}^{1,1}\left( \Phi _{1}\right)
\quad and\quad \Delta _{1}^{\prime }=\Delta _{1}+\delta _{Hom,C}^{1,1}\left(
\Phi _{1}\right) .
\end{equation}


Therefore, we have the following fundamental observation.

\begin{proposition}
The integrability of $\left( \mu _{1},\Delta _{1}\right) $ depends only on
its cohomology class.
\end{proposition}

\begin{proof}
Recall that two elements are cohomologous if their difference is a
coboundary.

Equation $\delta _{Hom}^{2}\left( \mu _{1},\Delta _{1}\right) =0$
implies that 
\begin{equation*}
\delta _{Hom,H}^{1,2}\left( \mu _{1}\right) =0,\text{ }\delta
_{Hom,C}^{2,1}\left( \Delta _{1}\right) =0,\text{ }\delta
_{Hom,C}^{1,2}\left( \mu _{1}\right) +\delta _{Hom,H}^{2,1}\left( \Delta
_{1}\right) =0.
\end{equation*}%
We have%
\begin{equation*}
\begin{array}{c}
\delta _{Hom,H}^{1,2}\left( \mu _{1}^{\prime }\right) =\delta
_{Hom,H}^{1,2}\left( \mu _{1}-\delta _{Hom,H}^{1,1}\left( f\right) \right)
=\delta _{Hom,H}^{1,2}\left( \mu _{1}\right) -\delta _{Hom,H}^{1,2}\circ
\delta _{Hom,H}^{1,1}\left( f\right) =0, \\ 
\delta _{Hom,C}^{2,1}\left( \Delta _{1}^{\prime }\right) =\delta
_{Hom,C}^{2,1}\left( \Delta _{1}+\delta _{Hom,C}^{1,1}\left( g\right)
\right) =\delta _{Hom,C}^{2,1}\left( \Delta _{1}\right) +\delta
_{Hom,C}^{2,1}\circ \delta _{Hom,C}^{1,1}\left( g\right) =0, \\ 
\delta _{Hom,C}^{1,2}\left( \mu _{1}^{\prime }\right) +\delta
_{Hom,H}^{2,1}\left( \Delta _{1}^{\prime }\right) =\delta
_{Hom,C}^{1,2}\left( \mu _{1}-\delta _{Hom,H}^{1,1}\left( f\right) \right)
+\delta _{Hom,H}^{2,1}\left( \Delta _{1}+\delta _{Hom,C}^{1,1}\left(
g\right) \right)  \\ 
=\delta _{Hom,C}^{1,2}\left( \mu _{1}\right) +\delta _{Hom,H}^{2,1}\left(
\Delta _{1}\right) -\left( \delta _{Hom,C}^{1,2}\circ \delta
_{Hom,H}^{1,1}\left( f\right) -\delta _{Hom,H}^{2,1}\circ \delta
_{Hom,C}^{1,1}\left( g\right) \right) =0.%
\end{array}%
\end{equation*}%
Thus $\delta _{Hom}^{2}\left( \mu
_{1}^{\prime },\Delta _{1}^{\prime }\right) =0.$

Equation $\left( \mu _{1},\Delta _{1}\right) =\delta
_{Hom}^{1}\left( f,g\right) $ implies that 
\begin{equation*}
\mu _{1}=\delta _{Hom,H}^{1,1}\left( f\right) ,\Delta _{1}=\delta
_{Hom,C}^{1,1}\left( g\right) 
\end{equation*}%
and 
\begin{equation*}
\mu _{1}^{\prime }=\mu _{1}-\delta _{Hom,H}^{1,1}\left( \Phi _{1}\right)
\quad and\quad \Delta _{1}^{\prime }=\Delta _{1}+\delta _{Hom,C}^{1,1}\left(
\Phi _{1}\right) .
\end{equation*}%
We have%
\begin{eqnarray*}
\mu _{1}^{\prime } &=&\delta _{Hom,H}^{1,1}\left( f\right) -\delta
_{Hom,H}^{1,1}\left( \Phi _{1}\right) =\delta _{Hom,H}^{1,1}\left( f-\Phi
_{1}\right) , \\
\Delta _{1}^{\prime } &=&\delta _{Hom,C}^{1,1}\left( g\right) +\delta
_{Hom,C}^{1,1}\left( \Phi _{1}\right) =\delta _{Hom,C}^{1,1}\left( g+\Phi
_{1}\right) .
\end{eqnarray*}%
Then,  if two integrable $2$-cocycles are cohomologous, then the corresponding
deformations are equivalent.
\end{proof}

\begin{proposition}
Let $B_{0}=\left( B,\mu _{0},\eta ,\Delta _{0},\varepsilon ,\alpha \right) $
be a Hom-bialgebra. There is, over $\Bbbk \lbrack \lbrack t]]/t^{2}$, a one
to-one correspondence between the elements of $H_{Hom}^{2}\left( B,B\right) $
and the infinitesimal deformation of $B_{0}$ defined by%
\begin{equation}
\mu _{t}=\mu _{0}+\mu _{1}t\quad and\quad \Delta _{t}=\Delta _{0}+\Delta
_{1}t.
\end{equation}
\end{proposition}

\begin{proof}
Deformation equation is equivalent to $\delta _{Hom}^{2}\left( \mu
_{1},\Delta _{1}\right) =0$, i.e.  $\left( \mu _{1},\Delta _{1}\right) \in 
$ $Z_{Hom}^{2}\left( B,B\right).$
\end{proof}\\

In the following,  we assume that $H_{Hom}^{2}\left( B,B\right) \neq 0$, then
one may obtain nontrivial one-parameter formal deformations. 
Suppose now that 
\begin{equation*}
\mu _{t}=\mu _{0}+\mu _{1}t+\mu _{2}t^{2}+...\quad and\quad \Delta
_{t}=\Delta _{0}+\Delta _{1}t+\Delta _{2}t^{2}+...
\end{equation*}
are a one parameter family of deformation of $\left( \mu _{0},\Delta
_{0}\right) $ for which 
\begin{equation*}
\mu _{1}=\mu _{2}=...=\mu _{m-1}=0\quad and\quad \Delta _{1}=\Delta
_{2}=...=\Delta _{m-1}=0.
\end{equation*}

The deformation equation implies 
\begin{equation*}
\delta _{Hom}^{1}\left( \mu _{m},\Delta _{m}\right) =0\quad \left( \left(
\mu _{m},\Delta _{m}\right) \in Z_{Hom}^{2}\left( B,B\right) \right) .
\end{equation*}%
If further $\left( \mu _{m},\Delta _{m}\right) \in B_{Hom}^{2}\left(
B,B\right) $, then by considering  the morphism $\Phi _{t}=id_{B}+\Phi _{m}t$,  we
obtain an equivalent deformation of the form
\begin{eqnarray*}
\mu _{t}^{\prime } &=&\Phi _{t}^{-1}\circ \mu _{t}\circ \left( \Phi
_{t}\otimes \Phi _{t}\right) =\mu _{0}+\mu _{m+1}t^{m+1}; \\
\Delta _{t}^{\prime } &=&\left( \Phi _{t}\otimes \Phi _{t}\right) \circ
\Delta _{t}\circ \Phi _{t}^{-1}=\Delta _{0}+\Delta _{m+1}t^{m+1}.
\end{eqnarray*}%
And again $\left( \mu _{m+1},\Delta _{m+1}\right) \in Z_{Hom}^{2}\left(
B,B\right) .$

\begin{theorem}
Let $B=\left( B,\mu ,\eta ,\Delta ,\varepsilon ,\alpha
\right) $ be a Hom-bialgebra. Every  nontrivial formal deformation   is equivalent to a deformation $B_{t}=\left(
B,\mu_{t},\eta, \Delta _{t},\varepsilon, \alpha \right) $ such that  $\mu
_{t}=\mu+\sum\limits_{i\geq p}t^{i} \mu _{i}$  and $\Delta
_{t}=\Delta+\sum\limits_{j\geq p}t^{j}\Delta _{j}$, where $(\mu_p,\Delta_p)$ is a 2-cocycle but not a 2-coboundary.\\ 
Hence, if $H_{Hom}^{2}\left( B,B\right) =0$ then every deformation of the  Hom-bialgebra $%
B$ is equivalent to a trivial deformation.
\end{theorem}

Hom-bialgebras for which every formal deformation is equivalent to a trivial
deformation are said to be \emph{analytically rigid}. The vanishing  of the second
cohomology group ($H_{Hom}^{2}\left( B,B\right) =0$) gives a sufficient
criterion for rigidity.

\subsection{Unital and Counital Hom-bialgebra Deformations}
We discuss  unitality and  counitality  of Hom-bialgebra  deformations.
\begin{proposition}\label{CNSUNit}
The unit (resp. the counit) of Hom-bialgebra $B$ is also the unit (resp. the
counit) of the formal deformation $B_{t}$ of $B$ if and only if  
\begin{eqnarray*}
\mu _{n}\left( x\otimes 1_{B}\right) &=&\mu _{n}\left( 1_{B}\otimes x\right)
=0\quad \forall n\geq 1,\quad \forall x\in B,\quad \eta \left( 1_{\Bbbk }\right)
=1_{B} \\
(resp.\text{ }\left( id_{B}\otimes \varepsilon \right) \circ \Delta _{n}
&=&\left( \varepsilon \otimes id_{B}\right) \circ \Delta _{n}=0\quad \forall
n\geq 1).
\end{eqnarray*}
\end{proposition}

\begin{proof}
The element $1_{B}$ is a unit for $B_{t}$ if 
$
\mu _{t}\circ \left( \eta _{}\otimes id_{B}\right) =\mu _{t}\circ \left(
id_{B}\otimes \eta _{}\right) =\alpha .
$
\begin{equation*}
\forall x\in B,\text{ }\mu _{t}\left( x\otimes 1_{B}\right) =\alpha \left(
x\right),\ \ \mu _{t}\left( 1_{B}\otimes x\right) =\alpha
\left( x\right) ,\text{ where }\mu _{t}=\sum\limits_{i\geq 1}t^{i}\mu _{i}.
\end{equation*}%
We have%
\begin{align*}
\mu _{t}\left( x\otimes 1_{B}\right) & =\alpha \left( x\right) =\mu
_{0}\left( x\otimes 1_{B}\right) +\sum\limits_{i\geq 1}\mu _{i}\left(
x\otimes 1_{B}\right) t^{i} \\
\alpha \left( x\right) & =\alpha \left( x\right) +\sum\limits_{i\geq 1}\mu
_{i}\left( x\otimes 1_{B}\right) t^{i}
\end{align*}%
By identification, we obtain $\mu _{n}\left( x\otimes 1_{B}\right) =0,\quad \forall n\geq 1$, and similarly 
$\mu _{n}\left(  1_{B}\otimes x\right) =0,\quad \forall n\geq 1$.
%
The map $\varepsilon $ is a counit for $B_{t}$ if $ \left( \varepsilon \otimes id_{B}\right) \circ \Delta _{t} =\left(
id_{B}\otimes \varepsilon \right) \circ \Delta _{t}=\alpha$, where $
\Delta _{t}=\sum\limits_{i\geq 1}t^{i}\Delta _{i}$.
\begin{eqnarray*}
\left( \varepsilon \otimes id_{B}\right) \left( \sum\limits_{i\geq 0}\Delta
_{i}t^{i}\right) &=&\alpha =\left( \varepsilon \otimes id_{V}\right) \circ
\Delta _{0}+\sum\limits_{i\geq 1}\left( \varepsilon \otimes id_{V}\right)
\circ \Delta _{i}t^{i} \\
\alpha &=&\alpha +\sum\limits_{i\geq 1}\left( \varepsilon \otimes
id_{B}\right) \circ \Delta _{i}t^{i}.
\end{eqnarray*}%
By identification, we obtain 
$
\left( \varepsilon \otimes id_{B}\right) \circ \Delta _{i}=0,\ \forall
i\geq 1.
$ Similarly  $\left( id_{B}\otimes \varepsilon
_{j}\right) \circ \Delta _{i}=0,\ \forall i\geq 1$.
\end{proof}\\

\begin{theorem}\label{DefUNitCounit}Let $B=\left( B,\mu ,\eta ,\Delta ,\varepsilon ,\alpha
\right) $ be a Hom-bialgebra with a surjective map $\alpha$. Every  nontrivial formal deformation   $B_{t}=\left(
B,\mu_{t},\eta_t, \Delta _{t},\varepsilon_t, \alpha \right) $ 
 is equivalent to a unital and counital deformation with the same unit $\eta$ and counit $\varepsilon$.
\end{theorem}
\begin{proof}We show that the unit is conserved by deformation
Assume  $\mu
_{t}=\mu+\sum\limits_{i\geq p}t^{i} \mu _{i}$.
Two deformations are equivalent if  there is a
formal isomorphism 
$
\Phi_t=Id+t\Phi_1+t^2\Phi_2+\cdots $,  where $\Phi_i\in End_\K\left(
V\right)
$
such that (\ref{(53)}) holds, which  leads to
\begin{equation}\label{uu1}
\mu'_1(x,y)=\mu_1(x,y)+f_1(\mu_0(x,y))-\mu_0(f_1(x),y)-\mu_0(x,f_1(y)).
\end{equation}
Since $\mu_1$ is a 2-cocycle, then
 $
 \sum_ { i+j=1 }{ \mu _i\left( \mu _j\left(
x,y\right) ,\alpha (z)\right) -\mu _i\left( \alpha(x),\mu _j\left( y,z\right)
\right)} =0.
 $\\
We set $y=z=1$,  respectively  $x=y=1$ and $z=x$. Then, we obtain
 \begin{eqnarray}\label{uu2}
 \mu_1(\alpha(x),1)=\mu_0 (\alpha (x),\mu_1(1,1)), \quad
 \mu_1(1,\alpha(x))=\mu_0 (\mu_1(1,1),\alpha (x)).
 \end{eqnarray}
  If $\alpha$ is surjective, then we have  \begin{eqnarray}\label{uu2bis}
 \mu_1(x,1)=\mu_0 (x,\mu_1(1,1)), \quad
 \mu_1(1,x)=\mu_0 (\mu_1(1,1),x).
 \end{eqnarray}
 We consider the formal isomorphism satisfying
 $f_1(1)=\mu_1(1,1), \  f_n=0 \ \ \text{for } n\geq 2.
 $
 Using (\ref{uu1}) and (\ref{uu2}),  the equivalent multiplication leads
 to a new deformed multiplication satisfying
 \begin{eqnarray*}
 \mu'_1(x,1)&=&\mu_1(x,1)+f_1(\mu_0(x,1))-\mu_0(f_1(x),1)-\mu_0(x,\mu_1(1,1))\\
 \ &=&\mu_1(x,1)+f_1(\alpha(x))-\alpha(f_1(x))-\mu_1(x,1)=0.
 \end{eqnarray*}
 Similarly, we obtain $\mu'_1(1,x)=0$.
 By induction on $n$, we show that for all $n\geq 1$,  $\mu'_n(1,x)=\mu'_n(x,1)=0$. 
 Indeed, we assume $\mu'_k(1,x)=\mu'_k(x,1)=0$ for $k=1,\cdots,n-1$. We consider the isomorphism $f_t$ satisfying $f_n(1)=\mu_n(1,1)$ and $f_k=0 \ \forall k\neq n.$ Then, using (\ref{uu1}) and (\ref{uu2}), we obtain $\mu'_n(1,x)=\mu'_n(x,1)=0$.

 Observe that the product $(1+f_1 t^1)\cdots (1+f_n t^n)$ converge when $n $ tends to infinity.\\
 Therefore, according to Proposition \ref{CNSUNit}, the unit is conserved by deformation.  The proof is similar for the counit.
\end{proof}

%
%

\subsection{Twistings and Deformations}
In this section, we discuss the connection between  twistings of Hom-bialgebras (see  Proposition \ref{TwistBialg}) and their formal deformations.

\begin{proposition}
\label{P}Let $B_{t}=\left( B\left[ \left[ t\right] \right] ,\mu _{t},\eta
_{t},\Delta _{t},\varepsilon _{t},\alpha \right) $ be a formal deformation of a
Hom-bialgebra $B=\left( B,\mu _{0},\eta _{0},\Delta _{0},\varepsilon
_{0},\alpha \right) $ and $\beta :B\longrightarrow B$ be a Hom-bialgebra morphism of
 $B$ and $B_{t}$.  Then $B_{t,\beta }=\left( B\left[ \left[ t%
\right] \right] ,\beta \circ \mu _{t},\beta \circ \eta _{t},\Delta _{t}\circ
\beta ,\varepsilon _{t}\circ \beta ,\beta \circ \alpha \right) $ is a formal
deformation of the Hom-bialgebra $B_{\beta }=\left( B,\beta \circ \mu _{0},\eta
_{0},\Delta _{0}\circ \beta ,\varepsilon _{0},\alpha \right) $.

Hence, for any $n\in 
\mathbb{N}, 
$ $B_{t,\beta ^{n}}=\left( B\left[ \left[ t\right] \right] ,\beta ^{n}\circ
\mu _{t},\beta ^{n}\circ \eta _{t},\Delta _{t}\circ \beta ^{n},\varepsilon
_{t}\circ \beta ^{n},\beta ^{n}\circ \alpha \right) $ is a formal
deformation of the Hom-bialgebra $B_{\beta ^{n}}$.
\end{proposition}

\begin{proof}
The proof is  analogous to that of Proposition \ref{TwistBialg}.
\end{proof}

\begin{corollary}
Let $B=\left( B,\mu _{0},\eta _{0},\Delta _{0},\varepsilon _{0}\right) $ \
be a bialgebra and $\alpha :B\longrightarrow B$ be a bialgebra morphism  (i.e. $\alpha \circ \mu _{0}=\mu _{0}\circ \left( \alpha \otimes \alpha
\right) ,$ $\Delta _{0}\circ \alpha =\left( \alpha \otimes \alpha \right)
\circ \Delta _{0},$ $\alpha \circ \eta _{0}=\eta _{0}$ and $\varepsilon
_{0}\circ \alpha =\varepsilon _{0}$).\\  If $B_{t}=\left( B\left[ \left[ t%
\right] \right] ,\mu _{t},\eta_t ,\Delta _{t},\varepsilon _{t}\right) $ is
a formal deformation of the bialgebra $B$ and $\alpha $ is a bialgebra morphism for  $B_{t}$ (i.e. $\alpha \circ \mu _{t}=\mu _{t}\circ \left( \alpha
\otimes \alpha \right) ,$ $\Delta _{t}\circ \alpha =\left( \alpha \otimes
\alpha \right) \circ \Delta _{t},$ $\alpha \circ \eta _{t}=\eta _{t}$ and $%
\varepsilon _{t}\circ \alpha =\varepsilon _{t}$). Then $B_{t,\alpha }=\left(
B\left[ \left[ t\right] \right] ,\alpha \circ \mu _{t},\eta _{t},\Delta
_{t}\circ \alpha ,\varepsilon _{t},\alpha \right) $ is a formal deformation
of the  Hom-bialgebra $B_{\alpha }=\left( B,\alpha \circ \mu _{0},\eta
_{0},\Delta _{0}\circ \alpha ,\varepsilon _{0},\alpha \right) $.
\end{corollary}



\begin{proposition}
\label{T}Let $B_{t}=\left( B,\mu _{t},\eta _{t},\Delta _{t},\varepsilon
_{t},\alpha \right) $ and $B_{t}^{\prime }=\left( B,\mu _{t}^{\prime },\eta
_{t}^{\prime },\Delta _{t}^{\prime },\varepsilon _{t}^{\prime },\alpha
\right) $ be  two equivalent deformations  of a Hom-bialgebra $B=\left( B,\mu
_{0},\eta _{0},\Delta _{0},\varepsilon _{0},\alpha \right) $.  Then, \\Ê $
B_{t,\alpha ^{n}}=\left( B\left[ \left[ t\right] \right] ,\alpha ^{n}\circ
\mu _{t},\eta _{t},\Delta _{t}\circ \alpha ^{n},\varepsilon _{t},\alpha
^{n+1}\right) $ and $B_{t,\alpha ^{n}}^{\prime }=\left( B\left[ \left[ t%
\right] \right] ,\alpha ^{n}\circ \mu _{t}^{\prime },\eta _{t}^{\prime
},\Delta _{t}^{\prime }\circ \alpha ^{n},\varepsilon _{t}^{\prime },\alpha
^{n+1}\right) $ are equivalent deformations of the Hom-bialgebra $B_{\alpha ^{n}}=\left(
B,\alpha ^{n}\circ \mu _{0},\eta _{0},\Delta _{0}\circ \alpha
^{n},\varepsilon _{0},\alpha ^{n+1}\right) ,$ for any $n\in 
\mathbb{N}
$.
\end{proposition}

\begin{proof}
We know that there exists a formal automorphism $\Phi
_{t}=\sum\limits_{i\geq 0}\Phi _{i}t^{i\text{ }}$, where $\Phi _{i}\in
End_{\Bbbk }(B)$ and $\Phi _{0}=id_{B}$ such that 
$ \Phi _{t}\circ \mu _{t}^{\prime } =\mu _{t}\circ \Phi _{t}^{\otimes 2},%
\text{ }\Phi _{t}^{\otimes 2}\circ \Delta _{t}^{\prime }=\Delta _{t}\circ
\Phi _{t},\text{ }and\ \Phi _{t}\circ \alpha =\alpha \circ \Phi _{t}.$\\ Then, we have 
\begin{eqnarray*}
\alpha ^{n}\circ \Phi _{t}\circ \mu _{t}^{\prime } &=&\alpha ^{n}\circ \mu
_{t}\circ \Phi _{t}^{\otimes 2},\text{ }\Phi _{t}^{\otimes 2}\circ \Delta
_{t}^{\prime }\circ \alpha ^{n}=\Delta _{t}\circ \Phi _{t}\circ \alpha ^{n},%
\text{ }and\ \Phi _{t}\circ \alpha \circ \alpha ^{n}=\alpha \circ \Phi
_{t}\circ \alpha ^{n}, \\
\Phi _{t}\circ \left( \alpha ^{n}\circ \mu _{t}^{\prime }\right) &=&\left(
\alpha ^{n}\circ \mu _{t}\right) \circ \Phi _{t}^{\otimes 2},\text{ }\Phi
_{t}^{\otimes 2}\circ \left( \Delta _{t}^{\prime }\circ \alpha ^{n}\right)
=\left( \Delta _{t}\circ \alpha ^{n}\right) \circ \Phi _{t},\text{ }and\
\Phi _{t}\circ \alpha ^{n+1}=\alpha ^{n+1}\circ \Phi _{t}.
\end{eqnarray*}%
Hence $B_{t,\alpha ^{n}}$ is equivalent to $B_{t,\alpha ^{n}}^{\prime },$ for
any $n\in 
\mathbb{N}
.$
\end{proof}

\begin{proposition}
\label{Z}Let  $B_{t}=\left( V,\mu _{t},\eta _{t},\Delta
_{t},\varepsilon _{t},\alpha \right) $  be a formal deformation of a Hom-bialgebra $B=\left( V,\mu
_{0},\eta _{0},\Delta _{0},\varepsilon _{0},\alpha \right) $. 
Then $B_{t,\alpha ^{n}}$ is equivalent to a formal deformation $%
B_{t,\alpha ^{n}}^{\prime }$\ of a Hom-bialgebra $B_{\alpha ^{n}}$\ with the
same unit and counit as $B_{\alpha ^{n}}$.
\end{proposition}
\begin{proof}The proof is similar to Theorem \ref{DefUNitCounit}. Notice that surjectivity is not required.
\end{proof}

\textbf{Addresses.}\\
 A. Makhlouf, Universit\'{e} de Haute Alsace,  Laboratoire de Math\'{e}matiques, Informatique et Applications,
4, rue des Fr\`{e}res Lumi\`{e}re F-68093 Mulhouse, France.\\
K. Dekkar, Universit\'{e} de Bordj Bou Arreridj, Algeria.\\
email : Abdenacer.Makhlouf@uha.fr, k.dekkar@univ-bba.dz
\end{document}